%% file: optimistic_MD.tex
\documentclass[11pt,letterpaper]{article}

\usepackage[utf8]{inputenc}
\usepackage{comment}
\usepackage{microtype}
\usepackage{graphicx}
\usepackage{subfig}
\usepackage{booktabs} %
\usepackage{scalerel}
\usepackage[margin=1in]{geometry}
\usepackage[style=ext-alphabetic, giveninits=true, maxbibnames=15, maxcitenames=15, natbib=true,
  maxalphanames=10, backend=biber, sorting=nty, backref=true]{biblatex}
\DefineBibliographyStrings{english}{backrefpage = {page},backrefpages = {pages},}

\renewbibmacro{in:}{%
  \ifentrytype{article}{}{\printtext{\bibstring{in}\intitlepunct}}}

\DeclareFieldFormat
[article,inbook,incollection,inproceedings,patent,thesis,unpublished]
{titlecase:title}{\MakeSentenceCase*{#1}}

\usepackage{amssymb}
\usepackage{amsthm}
\usepackage{amsmath}
\usepackage{mathtools}

\usepackage{nicefrac}

\usepackage{algorithmic}
\usepackage{algorithm}
\usepackage[shortlabels]{enumitem}

\usepackage{url}
\usepackage{float}

\usepackage{pifont}
\usepackage{hhline}
\usepackage{multirow}
\usepackage{makecell}

\newcommand{\myalert}[1]{{\vspace{1mm}\noindent\textbf{#1}}}
\newcommand{\xmark}{\ding{55}}
\newcommand{\tickmark}{\ding{51}}

\floatstyle{ruled}
\newfloat{subroutine}{htbp}{lok}
\floatname{subroutine}{Subroutine}

\input{math_commands.tex}

\newtheorem{assumption}{Assumption}
\newtheorem{definition}{Definition}
\newtheorem{theorem}{Theorem}
\newtheorem{lemma}[theorem]{Lemma}
\newtheorem{proposition}[theorem]{Proposition}
\newtheorem{corollary}[theorem]{Corollary}
\newtheorem{remark}{Remark}
\numberwithin{assumption}{section}
\numberwithin{definition}{section}
\numberwithin{theorem}{section}
\numberwithin{remark}{section}

\usepackage{tikz}
\usetikzlibrary{decorations.pathreplacing,calc}
\newcommand{\tikzmark}[1]{\tikz[overlay,remember picture] \node (#1) {};}

\newcommand*{\AddNote}[4]{%
    \begin{tikzpicture}[overlay, remember picture]
        \draw [decoration={brace,amplitude=0.3em},decorate,thick,black]
            ($(#3.east)!(#1.north)!($(#3.east)-(0,1)$)$) --  
            ($(#3.east)!(#2.south)!($(#3.east)-(0,1)$)$)
                node [align=left, text width=4cm, pos=0.5, anchor=west] {#4};
    \end{tikzpicture}
}%

\usepackage{hyperref}
\hypersetup{colorlinks,citecolor=blue,linktocpage,breaklinks=true}

\newcommand{\blue}[1]{{\color{blue}{#1}}}
\renewcommand{\blue}[1]{#1}

\begin{document}

\title{Generalized Optimistic Methods for Convex-Concave Saddle Point Problems}

\author{Ruichen Jiang\thanks{Department of Electrical and Computer Engineering, The University of Texas at Austin, Austin, TX, USA \qquad\qquad\{rjiang@utexas.edu, mokhtari@austin.utexas.edu\}}         \and
        Aryan Mokhtari$^*$
}

\date{}

\maketitle

\begin{abstract} The optimistic gradient method has seen increasing popularity as an efficient first-order method for solving convex-concave saddle point problems. To analyze its iteration complexity, a recent work \cite{Mokhtari2020} proposed an interesting perspective that interprets the optimistic gradient method as an approximation to the proximal point method. In this paper, we follow this approach and distill the underlying idea of optimism to propose a generalized optimistic method, which encompasses the optimistic gradient method as a special case. Our general framework can handle constrained saddle point problems with composite objective functions and can work with arbitrary norms with compatible Bregman distances. Moreover, we also develop a backtracking line search scheme to select the step sizes without knowledge of the smoothness coefficients. We instantiate our method with first-order, second-order and higher-order oracles and give best-known global iteration complexity bounds. 
{For our first-order method, we show that the averaged iterates converge at a rate of $\bigO(\nicefrac{1}{N})$ when the objective function is convex-concave, and it achieves linear convergence when the objective is further strongly-convex-strongly-concave. For our second-order and higher-order methods, under the additional assumption that the distance-generating function has Lipschitz gradient, we prove a complexity bound of $\bigO(\nicefrac{1}{\epsilon}^{\frac{2}{p+1}})$ in the convex-concave setting and a complexity bound of $\bigO((\nicefrac{L_p D^{\frac{p-1}{2}}}{\mu})^{\frac{2}{p+1}}+\log\log \frac{1}{\epsilon})$ in the strongly-convex-strongly-concave setting,}
where $L_p$ ({$p\geq 2$}) is the Lipschitz constant of the $p$-th-order derivative, $\mu$ is the strong convexity parameter, and $D$ is the initial Bregman distance to the saddle point.
Moreover, our line search scheme provably only requires a constant number of calls to a subproblem solver per iteration on average, making our first-order and second-order methods particularly amenable to implementation.   
\end{abstract}

\newpage
\tableofcontents
\newpage

\input{intro.tex}

\input{related_work.tex}
\input{preliminaries.tex}

\input{proximal.tex}
\input{GOMD.tex}
\input{first_order.tex}

\input{second_order.tex}
\input{higher_order.tex}
\input{numerical.tex}

\section{Conclusions}\label{sec:conclusions}
In this paper, we proposed the generalized optimistic method for composite convex-concave saddle point problems by approximating the Bregman proximal point method. We also designed a novel line search scheme to select the step size in our method adaptively and without knowledge of the smoothness coefficients of the objective. Under our unified framework, our first-order optimistic method provably achieves a complexity bound of $\bigO(1/\epsilon)$ in terms of the restricted primal-dual gap in the convex-concave setting, and a complexity bound of $\bigO(\kappa_1 \log \frac{1}{\epsilon})$ in terms of the Bregman distance to the saddle point in the strongly-convex-strongly-concave setting. Furthermore, \blue{under the assumption that the distance-generating function has Lipschitz gradient,} our $p$-th-order optimistic method ($p\geq 2$) provably achieves a complexity bound of $\bigO(1/\epsilon^{\frac{2}{p+1}})$ in the convex-concave setting, and a complexity bound of $\bigO((\kappa_p(\vz_0))^{\frac{2}{p+1}}+\log\log\frac{1}{\epsilon})$ in the strongly-convex-strongly-concave setting. We also proved that our line search scheme only requires a constant number of calls to the optimistic subsolvers per iteration on average, which is supported by our numerical experiments.

\section*{Acknowledgements}
This research of R. Jiang and A. Mokhtari is supported in part by NSF Grants 2007668, 2019844, and  2112471,  ARO  Grant  W911NF2110226,  the  Machine  Learning  Lab  (MLL)  at  UT  Austin, and the Wireless Networking and Communications Group (WNCG) Industrial Affiliates Program.
\printbibliography

\appendix

\input{appendix.tex}

\end{document}

%% file: math_commands.tex
\newcommand\Vector[1]{\mathbf{#1}}

\newcommand\vb{{\Vector{b}}}

\newcommand\vh{{\Vector{h}}}

\newcommand\vu{{\Vector{u}}}
\newcommand\vv{{\Vector{v}}}
\newcommand\vw{{\Vector{w}}}
\newcommand\vx{{\Vector{x}}}
\newcommand\vy{{\Vector{y}}}
\newcommand\vz{{\Vector{z}}}

\newcommand\MATRIX[1]{\mathbf{#1}}

\newcommand\mA{{\MATRIX{A}}}

\newcommand\mI{{\MATRIX{I}}}

\newcommand\mX{{\MATRIX{X}}}

\newcommand\bigO{\mathcal{O}}

\DeclareMathOperator*{\E}{\mathbb{E}}
\newcommand\op{{\mathrm{op}}}

\DeclareMathOperator*{\argmin}{arg\,min}

\DeclareMathOperator{\dom}{dom\,}

%% file: intro.tex
\section{Introduction}

In this paper, we study convex-concave saddle point problems, also known as minimax optimization problems, where the objective function for both minimization and maximization has a composite structure. %
Specifically, we consider
\begin{equation}\label{eq:minimax}
    \min_{\vx\in\mathcal{X}} \max_{\vy\in\mathcal{Y}}\;\;\; \ell(\vx,\vy):=f(\vx,\vy)+h_1(\vx)-h_2(\vy),
\end{equation}
where $\mathcal{X}\subset \mathbb{R}^m$ and $\mathcal{Y} \subset \mathbb{R}^n$ are nonempty closed convex sets, $h_1:\mathbb{R}^m\rightarrow (-\infty,+\infty]$ and $h_2:\mathbb{R}^n\rightarrow (-\infty,+\infty]$ are proper closed convex functions, and $f$ is a smooth function defined on some open set $\dom f \subset \mathbb{R}^m \times \mathbb{R}^n$ containing $\mathcal{X}\times \mathcal{Y}$. Throughout the paper, we assume that $f$ is convex with respect to $\vx$ and concave with respect to $\vy$, i.e., $f(\cdot,\vy)$ is convex for each $\vy\in \mathcal{Y}$ and $f(\vx,\cdot)$ is concave for each $\vx\in \mathcal{X}$. As a result, the objective function $\ell$ is also convex-concave. 

Problem~\eqref{eq:minimax}, which we refer to as the {composite saddle point problem}, encompasses several important special cases. For instance, when $h_1\equiv 0$ on $\mathcal{X}$ and $h_2\equiv 0$ on $\mathcal{Y}$, it becomes the constrained smooth saddle point problem \cite{Korpelevich1976,Popov1980,Nemirovski2004}. If we further have $\mathcal{X}=\mathbb{R}^m$ and $\mathcal{Y}=\mathbb{R}^n$, it reduces to the unconstrained smooth saddle point problem \cite{Daskalakis2018,Mokhtari2020a,Mokhtari2020}. 
Problem \eqref{eq:minimax} is of central importance in the duality theory of constrained optimization {and appears as the primal-dual formulation of conic programming~\cite{Hamedani2021}, regularized empirical risk minimization~\cite{Zhang2015} and several imaging problems~\cite{Chambolle2011}.} 
It also arises in settings 
{such as zero-sum games~\cite{Basar1998} and distributionally robust optimization~\cite{Yu2022}.}

Various iterative methods---mostly first-order---have been proposed in the literature for solving saddle point problems. In particular, we will focus on optimistic methods, whose idea was first introduced by Popov \cite{Popov1980} and have gained much attention recently in the machine learning community \cite{Chiang2012,Rakhlin2013,Daskalakis2018,Liang2019,gidel2018a,Peng2020}. For constrained smooth strongly-convex strongly-concave problems (where $f$ is smooth and strongly-convex-strongly-concave), the result in \cite{gidel2018a} showed that Popov's method converges linearly and finds an $\epsilon$-accurate solution with an iteration complexity bound of $\bigO(\kappa_1\log(1/\epsilon))$, where $\kappa_1$ is the condition number we shall define in Section~\ref{subsec:related_work}. Moreover, a variant of Popov's method proposed in \cite{Malitsky2020} for monotone inclusion problems is proved to achieve the same complexity bound in the more general composite setting. For unconstrained smooth convex-concave problems, Popov's method, also more commonly known as the optimistic gradient descent-ascent (OGDA) method, is shown to achieve the complexity bound of $\bigO(1/\epsilon)$ in terms of the primal-dual gap \cite{Mokhtari2020}. 

In this paper, we follow and extend the approach in \cite{Mokhtari2020a,Mokhtari2020} by interpreting the optimistic method as an approximation of the proximal point method (PPM) \cite{Martinet1970,Rockafellar1976}. %
In particular, we propose a generalized optimistic method (see Algorithm~\ref{alg:GOMD}), where the future gradient required in the update of the PPM is replaced by the combination of a prediction term and a correction term. The prediction term serves as a local approximation of the future gradient, while the correction term is given by the prediction error at the previous iteration. 
In our framework, the existing first-order optimistic method corresponds to using the current gradient as the prediction term, whereas our theory allows general prediction terms in the setup of arbitrary norms and compatible Bregman distances.  
{Moreover, we develop a simple backtracking line search scheme (Algorithm~\ref{alg:ls}) to select the step sizes in our generalized optimistic method adaptively without knowledge of the smoothness coefficients of the objective.}        
By instantiating our method with different oracles, we obtain first-order, second-order and higher-order optimistic methods for the saddle point problem in \eqref{eq:minimax}. We give {best-known} complexity bounds of these methods in both the convex-concave and strongly-convex-strongly-concave settings, all matching the best existing upper bounds in their corresponding problem classes (see Section~\ref{subsec:related_work} for detailed comparisons).  The complexity of our line search scheme in terms of the number of calls to a subproblem solver is also presented. Our theoretical findings are summarized as follows:
\begin{enumerate}[label=(\alph*)]
\vspace{-2mm}
    \item When the smooth component $f$ of the objective function in \eqref{eq:minimax} has Lipschitz continuous gradient, our first-order optimistic method generalizes the OGDA method in \cite{Mokhtari2020a,Mokhtari2020} which only focuses on unconstrained smooth saddle point problems. We prove a complexity bound of $\bigO(1/\epsilon)$ in terms of the primal-dual gap in the convex-concave setting, and prove linear convergence with a complexity bound of $\bigO(\kappa_1 \log(1/\epsilon))$ in terms of the distance to the optimal solution in the strongly-convex-strongly-concave setting. In addition, by incorporating our line search scheme, we obtain an adaptive first-order optimistic method with the same convergence rates while only requiring a constant number of calls to a subproblem solver per iteration on average. 
    \vspace{2mm}
    \item When $f$ has Lipschitz continuous Hessian {and the distance-generating function has Lipschitz gradient}, we obtain an adaptive second-order optimistic method where the step sizes are chosen by our line search scheme. In the convex-concave setting, we show that it achieves a complexity of $\bigO(\epsilon^{-\frac{2}{3}})$ in terms of the primal-dual gap. In the strongly-convex-strongly-concave setting,  we prove global convergence and a local R-superlinear convergence rate of the order~$\frac{3}{2}$ for our method, leading to an overall complexity of $\bigO((\kappa_2(\vz_0))^{\frac{2}{3}} + \log\log\frac{1}{\epsilon})$ in terms of the distance to the optimal solution (here $\kappa_2(\vz_0)$ is {a problem-dependent constant}
    we shall define in \eqref{eq:condition_number}). Also, we prove that the line search procedure on average only requires a constant number of calls to a subproblem solver per iteration. %
    
    \vspace{2mm}
    \item When $f$ has Lipschitz $p$-th-order derivative with $p\geq 3$ {and the distance-generating function has Lipschitz gradient}, we further extend the results above and propose an adaptive $p$-th-order optimistic method. In the convex-concave setting, the complexity bound is improved to $\bigO(\epsilon^{-\frac{2}{p+1}})$. In the strongly-convex-strongly-concave setting, we prove global convergence and a local R-superlinear convergence rate of the order $\frac{p+1}{2}$, leading to an overall complexity of $\bigO((\kappa_p(\vz_0))^{\frac{2}{p+1}} + \log\log\frac{1}{\epsilon})$ (here $\kappa_p(\vz_0)$ is {defined in \eqref{eq:condition_number}}). Similarly, the line search procedure on average only requires a constant number of  calls to a subproblem solver per iteration.  
\end{enumerate}

%% file: related_work.tex
\subsection{Related work}\label{subsec:related_work}

In this section, we review iterative methods for convex-concave saddle point problems and their iteration complexity results.  
It is worth noting that most existing methods are developed by reformulating the saddle point problem in~\eqref{eq:minimax} as a monotone variational inequality and/or an inclusion problem (see Section \ref{subsec:saddle_point_problem}), which is also the approach we take to derive our generalized optimistic method. Therefore, we also include methods for solving this broader class of problems. %

For ease of exposition, in the strongly-convex-strongly-concave case, we define the condition number by $\kappa_1:=L_1/\mu$, where $L_1$ is the Lipschitz constant of the gradient of $f$ (cf. Assumption~\ref{assum:first_order}) and $\mu$ is the strong convexity parameter (cf. Assumption~\ref{assum:strongly_monotone}). More generally, for $p\geq 2$, 
{we define a problem-dependent constant} 
\begin{equation}\label{eq:condition_number}
	\kappa_p(\vz_0) = \frac{L_p (D_{\Phi}(\vz^*,\vz_0))^{\frac{p-1}{2}}}{\mu},
\end{equation}
where $L_p$ is the Lipschitz constant of the $p$-th-order derivative of $f$ (cf. Assumptions~\ref{assum:second_order} and~\ref{assum:high_order}) and $D_{\Phi}(\vz^*,\vz_0)$ is the Bregman distance between the optimal solution $\vz^*=(\vx^*,\vy^*)$ and the initial point $\vz_0=(\vx_0,\vy_0)$ (see Section~\ref{sec:preliminaries} for formal definitions). 
Intuitively, these quantities measure the hardness of solving the saddle point problems and they naturally appear in the complexity bounds of the first-order, second-order and higher-order methods, respectively. 

\noindent \textbf{First-order methods.} One of the earliest methods for solving constrained saddle point problems is the Arrow-Hurwicz method \cite{Arrow1958}, also known as gradient descent ascent\footnote{It corresponds to the projection method for variational inequalities, and the forward-backward method for monotone inclusion problems (see, e.g., \cite{Facchinei2003}).}. 
However, this method could fail to converge when the function $f$ in \eqref{eq:minimax} is smooth and only convex-concave, and only achieves a suboptimal iteration complexity of $\bigO(\kappa_1^2\log(1/\epsilon))$ when $f$ is smooth and strongly-convex-strongly-concave \cite{Nesterov2011,Fallah2020}. 
To remedy these issues, Korplevich \cite{Korpelevich1976} and Popov \cite{Popov1980} proposed two different modifications by introducing ``extrapolation'' and ``optimism'' into the Arrow-Hurwicz method, respectively. %

Korplevich's extragradient method \cite{Korpelevich1976} solves constrained saddle point problems by modifying the Arrow-Hurwicz method with an extrapolation step. For smooth and strongly-convex-strongly-concave functions, it converges linearly and finds an $\epsilon$-accurate saddle point within $\bigO(\kappa_1 \log(1/\epsilon))$ iterations \cite{Tseng1995,Azizian2020,Mokhtari2020a}. The mirror-prox method \cite{Nemirovski2004} generalizes the extragradient method and works with a general Bregman distance, achieving the ergodic iteration complexity of $\mathcal{O}(1/\epsilon)$ for smooth monotone {variational} inequalities and smooth convex-concave saddle point problems with compact feasible sets. 
Subsequently, it was extended to a class of {non-Euclidean distance-like functions} by \cite{Auslender2005}, to unbounded feasible sets by \cite{Monteiro2012}, and to composite objectives by \cite{Tseng2008,Monteiro2011,He2015}, all with the same rate of convergence. 
Along another line of research, the dual extrapolation method \cite{Nesterov2007a} was developed for variational inequalities and saddle point problems, where the extrapolation step is performed in the dual space. It is also shown to achieve the $\bigO(1/\epsilon)$ complexity bound for smooth monotone variational inequalities. Moreover, for smooth strongly monotone variational inequalities, it is proved to achieve the optimal iteration complexity of $\bigO(\kappa_1\log(1/\epsilon))$ when combined with the technique of estimate functions \cite{Nesterov2011}. 
Finally, Tseng \cite{Tseng2000} utilized the similar extrapolation idea to propose a splitting method for monotone inclusion problems and proved that the method converges linearly and achieves the iteration complexity of $\bigO(\kappa_1\log(1/\epsilon))$ for smooth strongly monotone inclusion problems. Later, a variant of Tseng's splitting method is shown to also achieve the iteration complexity of $\bigO(1/\epsilon)$ for smooth monotone inclusion problems and smooth convex-concave saddle point problems by using the hybrid proximal extragradient (HPE) framework \cite{Monteiro2011}. 

A defining characteristic of the above methods is that they need to maintain two intertwined sequences of updates at which the gradients are computed. In contrast, Popov's method \cite{Popov1980}, whose special case is also known as optimistic gradient descent ascent (OGDA) \cite{Daskalakis2018,Mokhtari2020a,Mokhtari2020}, requires one gradient computation per iteration.  It is shown to have a complexity of $\bigO(\kappa_1\log(1/\epsilon))$ for constrained saddle point problems when the objective is smooth and strongly-convex-strongly-concave \cite{gidel2018a,Azizian2020,Mokhtari2020a}; see also a similar result in \cite{Malitsky2020} for a variant of Popov's method on monotone inclusion problems. Moreover, by viewing OGDA as an approximation of the proximal point method, the authors in \cite{Mokhtari2020} proved an ergodic complexity of $\bigO(1/\epsilon)$ for unconstrained saddle point problems when the objective is smooth and convex-concave, matching that of the extragradient method.  
Several other variants of Popov's method have also appeared in the literature \cite{Malitsky2015,Malitsky2017}; see \cite{Hsieh2019} for more details. In a concurrent work \cite{kotsalis2022simple} on variational inequalities, the authors extended the method in \cite{Malitsky2020} to a general Bregman distance and proved the same convergence guarantees as above for both the convex-concave and strongly-convex-strongly-concave settings.
We note that their method is similar to our first-order optimistic method in Section~\ref{sec:first_order}, though proposed from a different perspective. In this paper, we recover and further extend their results to composite objectives via our unified framework.

\begin{table}[!t]
    \renewcommand{\arraystretch}{1.6}
    \centering
    \caption{Summary of first-order methods for saddle point problems. The third column indicates whether the method can handle general convex functions $h_1$ and $h_2$ in \eqref{eq:minimax} or requires $h_1\equiv 0$ and $h_2\equiv 0$. In the fourth column, ``C-C'' stands for ``convex-concave'' and ``SC-SC'' stands for ``strongly-convex-strongly-concave''. In the fifth column, ``$L_1$'' means $\nabla f$ is $L_1$-Lipschitz. The sixth column indicates whether the saddle point problem is constrained or not. The seventh column shows whether the method can work with a general Bregman distance or only the Euclidean norm. }\label{tab:first_order}
    \resizebox{\textwidth}{!}{%
        \begin{tabular}{|c|c|c|c|c|c|c|c|}
            \hline
			\multirow{2}{*}{Frameworks}			& \multirow{2}{*}{Methods} 		                                        & \multirow{2}{*}{Composite objective?}                                      & \multicolumn{2}{c|}{Assumptions}					    & \multirow{2}{*}{Constrained?} 	& \multirow{2}{*}{Distance} & \multirow{2}{*}{Complexity}           \\ \cline{4-5}
                                            	&                             	                                        &                                                         & Convexity                     & Smoothness            &									&                       &                                       \\ \hline
			\multirow{6}{*}{Extragradient}		& Mirror-prox \cite{Tseng2008}                                          & \tickmark                 		& C-C                           & $L_1$                 & \tickmark                 			& Bregman               & $\bigO(\epsilon^{-1})$             	\\ \cline{2-8}
											    & Extragradient \cite{Tseng1995}                                        & \xmark                            & SC-SC                         & $L_1$                 & \tickmark                   			& Euclidean             & $\bigO(\kappa_1\log\epsilon^{-1})$ 	\\ \cline{2-8}
                                            	& Dual extrapolation \cite{Nesterov2007a}                               & \xmark                            & C-C                           & $L_1$                 & \tickmark                 			& Bregman               & $\bigO(\epsilon^{-1})$ 				\\ \cline{2-8}
												& {Dual extrapolation \cite{Nesterov2011}}                              & \xmark                            & SC-SC                         & $L_1$                 & \tickmark                 			& Euclidean             & $\bigO(\kappa_1\log\epsilon^{-1})$	\\ \cline{2-8}
                                            	& {Tseng's splitting \cite{Monteiro2011}}                               & \tickmark                 & C-C                           & $L_1$                 & \tickmark                   			& Euclidean             & $\bigO(\epsilon^{-1})$                \\ \cline{2-8}
												& Tseng's splitting \cite{Tseng2000}                                    & \tickmark                 & SC-SC                         & $L_1$                 & \tickmark                   			& Euclidean             & $\bigO(\kappa_1\log\epsilon^{-1})$    \\ \hhline{*{8}{=}}
			{Primal-Dual}						& Accelerated primal-dual method \cite{Hamedani2021}																		& \tickmark					& C-C							& $L_1$					& \tickmark								& Bregman				& $\bigO(\epsilon^{-1})$				\\ \hhline{*{8}{=}}
			\multirow{8}{*}{Optimistic}			& OGDA \cite{Mokhtari2020}                                              & \xmark                            & C-C                           & $L_1$                 & \xmark                 			    & Euclidean             & $\bigO(\epsilon^{-1})$             	\\ \cline{2-8}
												& OGDA \cite{Mokhtari2020a}                                             & \xmark                            & SC-SC                         & $L_1$                 & \xmark                 			    & Euclidean             & $\bigO(\kappa_1\log\epsilon^{-1})$    \\ \cline{2-8}
												& Popov's method \cite{gidel2018a}                                      & \xmark                            & SC-SC                         & $L_1$                 & \tickmark                 			& Euclidean             & $\bigO(\kappa_1\log\epsilon^{-1})$    \\ \cline{2-8}
												& {Forward-reflected-backward method \cite{Malitsky2020}}               & \tickmark                 & SC-SC                         & $L_1$                 & \tickmark                 			& Euclidean             & $\bigO(\kappa_1\log\epsilon^{-1})$    \\ \cline{2-8}
												& {Operator extrapolation method \cite{kotsalis2022simple}}                   & \xmark                            & C-C                           & $L_1$                 & \tickmark                 			& Bregman               & $\bigO(\epsilon^{-1})$             	\\ \cline{2-8}
												& {Operator extrapolation method \cite{kotsalis2022simple}}                   & \xmark                            & SC-SC                         & $L_1$                 & \tickmark                 			& Bregman               & $\bigO(\kappa_1\log\epsilon^{-1})$    \\ \cline{2-8}
                                            	& \textbf{Theorems~\ref{thm:OMD_first} \&~\ref{thm:OMD_first_ls}}                                                     & \tickmark                 & C-C                           & $L_1$                 & \tickmark                 			& Bregman               & $\bigO(\epsilon^{-1})$ 				\\ \cline{2-8}
                                            	& \textbf{Theorems~\ref{thm:OMD_first} \&~\ref{thm:OMD_first_ls}}                                                     & \tickmark                 & SC-SC                         & $L_1$                 & \tickmark                 			& Bregman               & $\bigO(\kappa_1\log\epsilon^{-1})$    \\ \hline
        \end{tabular}%
    }
    \vspace{-1mm}
\end{table}

There is also a popular family of algorithms in the literature, known as the primal-dual method \cite{Chambolle2011,Condat2013,Chen2014,Chambolle2016, Malitsky2018,Hamedani2021}, that is more tailored to the structure of the saddle point problem by updating the $\vx$ and $\vy$ variables alternatively. They are mostly designed for \emph{bilinear} saddle point problems where the function $f$ in~\eqref{eq:minimax} is given by $f(\vx,\vy)=f_1(\vx)+\langle \mA \vx,\vy \rangle$, with $f_1$ being a smooth convex function and $\mA$ being an $n$-by-$m$ matrix. A notable exception is the recent work in \cite{Hamedani2021}, which generalized the original primal-dual method in \cite{Chambolle2016} for solving the general composite saddle point problem as in~\eqref{eq:minimax}. It was shown to achieve the complexity of $\bigO(1/\epsilon)$ in terms of the primal-dual gap when $f$ is smooth and convex-concave, and a better complexity of $\bigO(1/\sqrt{\epsilon})$ when $f$ is smooth and strongly-convex-linear\footnote{$f(\vx,\vy)$ is strongly-convex-linear if $f(\cdot,\vy)$ is strongly convex for each fixed $\vy$ and $f(\vx,\cdot)$ is linear for each fixed $\vx$. }. Some of the existing results as well as ours are summarized in Table~\ref{tab:first_order}.

\myalert{Second-order and higher-order methods.}
Unlike first-order methods, relatively few methods have been proposed to utilize second-order and higher-order information in saddle point problems. So far, there have been two separate approaches in the literature: one is to generalize the well-studied second-order methods for optimization problems, in particular the cubic regularized Newton's method \cite{Nesterov2006,Nesterov2008}; and the other is to generalize the first-order extragradient method \cite{Korpelevich1976,Nemirovski2004}. In this work, we offer a third alternative with comparable and sometimes better convergence properties. 

A celebrated technique of globalizing Newton's method for unconstrained optimization problems is the cubic regularization scheme proposed in \cite{Nesterov2006} with superior global complexity bounds \cite{Nesterov2008,Cartis2011}. In \cite{Nesterov2006a}, Nesterov further extended this methodology to solving variational inequalities. For monotone operators with Lipschitz Jacobian, the proposed second-order method achieves a complexity bound of $\bigO(1/\epsilon)$. For strongly-monotone operators with Lipschitz Jacobian, a local quadratic convergence is established for the regularized Newton method, and an overall complexity of $\bigO(\kappa_2+\log\log(\epsilon^{-1}))$ is achieved by combining several techniques. %
More recently, a new cubic regularization scheme was proposed in \cite{Huang2022cubic} for unconstrained smooth saddle point problems. When the objective is strongly-convex-strongly-concave and both its gradient and Hessian are Lipschitz continuous, the authors showed a global linear convergence and a local quadratic convergence with an iteration complexity of $\bigO(\kappa^2_1+\kappa_1\kappa_2+\log\log (\epsilon^{-1}))$. Note that these complexity bounds are worse than the ones presented in this paper. 

\begin{table}[!t]
    \renewcommand{\arraystretch}{1.6}
    \centering
    \caption{Summary of second-order and higher-order methods for saddle point problems. We adopt similar notations as in Table~\ref{tab:first_order}. In addition, ``$L_2$'' means the Hessian of $f$ is $L_2$-Lipschitz and more generally ``$L_p$'' means the $p$-th-order derivative of $f$ is $L_p$-Lipschitz for some $p$. We note that the hc-CRN-SPP method in \cite{Huang2022cubic} also requires an additional error bound assumption, which we discuss in the main text. {Also, for our results, we assume the distance-generating function is smooth.} }\label{tab:high_order}
    \resizebox{\textwidth}{!}{%
        \begin{tabular}{|c|c|c|c|c|c|c|c|}
            \hline
			\multirow{2}{*}{Frameworks}			    & \multirow{2}{*}{Methods} 		                                            & \multirow{2}{*}{Composite objective?}                                      & \multicolumn{2}{c|}{Assumptions}					    & \multirow{2}{*}{Constrained?} 	& \multirow{2}{*}{Distance} & \multirow{2}{*}{Complexity}                                                                   \\ \cline{4-5}
                                            	    &                             	                                            &                                                         & Convexity                     & Smoothness            &									&                       &                                                                                               \\ \hline
			\multirow{4}{*}{Cubic Regularization}   & {Dual Newton's method \cite{Nesterov2006a}}                               & \xmark                            & C-C                           & $L_2$                 & \xmark                              & Euclidean             & $\bigO(\epsilon^{-1})$                                                                        \\ \cline{2-8}
                                                    & {Restarted dual Newton's method \cite{Nesterov2006a}}                     & \xmark                            & SC-SC                         & $L_2$                 & \xmark                              & Euclidean             & $\bigO(\kappa_2+\log\log (\epsilon^{-1}))$                                                      \\ \cline{2-8}
                                                    & {hc-CRN-SPP \cite{Huang2022cubic}}                                             & \xmark                            & C-C                           & $L_1,L_2$             & \xmark                              & Euclidean             & {$\bigO(\log (\epsilon^{-1}))$ or $\bigO(\epsilon^{-\frac{1-\theta}{\theta^2}})$}    \\ \cline{2-8}
                                                    & {CRN-SPP} \cite{Huang2022cubic}                                                & \xmark                            & SC-SC                         & $L_1,L_2$             & \xmark                              & Euclidean             & $\bigO(\kappa^2_1+\kappa_1\kappa_2+\log\log (\epsilon^{-1}))$                                   \\ \hhline{*{8}{=}}
			\multirow{3}{*}{Extragradient}		    & {Newton proximal extragradient \cite{Monteiro2012}}                       & \tickmark                 & C-C                           & $L_2$                 & \tickmark                            & Euclidean             & $\bigO(\epsilon^{-\frac{2}{3}})$                                                              \\ \cline{2-8}
                                                    & {High-order Mirror-Prox \cite{Bullins2022higher}}                               & \xmark                            & C-C                           & $L_p$ ($p\geq 2$)     & \tickmark                            & Bregman               & $\bigO(\epsilon^{-\frac{2}{p+1}})$                                                            \\ \cline{2-8}
                                                    & {Restarted high-order Mirror-Prox \cite{Ostroukhov2020}}                  & \xmark                            & SC-SC                         & $L_p$  ($p\geq 2$)    & \xmark                              & Euclidean             & $\bigO(\kappa_p^{\frac{2}{p+1}}\log (\epsilon^{-1}))$                                                  \\ \hhline{*{8}{=}}
            \multirow{2}{*}{Optimistic}     		& \textbf{Theorems~\ref{thm:OMD_2nd_ls_convex} \&~\ref{thm:OMD_high_ls_convex}}                                                         & \tickmark                 & C-C                           & $L_p$  ($p\geq 2$)    & \tickmark                           & \makecell{Bregman,\\ \blue{smooth}}             & $\bigO(\epsilon^{-\frac{2}{p+1}})$                                                            \\ \cline{2-8}
                                                    & \textbf{Corollaries~\ref{coro:2nd_complexity_bound} \&~\ref{coro:high_complexity_bound}}                                                         & \tickmark                 & SC-SC                         & $L_p$ ($p\geq 2$)     & \tickmark                           & \makecell{Bregman,\\\blue{smooth}}            & $\bigO(\kappa_p^{\frac{2}{p+1}}+\log\log (\epsilon^{-1}))$                                      \\ \hline
        \end{tabular}%
    }
\end{table}

Along another line of research, a Newton-type extragradient method was analyzed in \cite{Monteiro2010,Monteiro2012} for solving monotone inclusion problems under the HPE framework \cite{Solodov1999}, and was shown to achieve an iteration complexity of $\bigO(\epsilon^{-\frac{2}{3}})$. %
More recently, the authors in \cite{Bullins2022higher} proposed second-order and higher-order methods by directly extending the analysis of the classical mirror-prox method \cite{Nemirovski2004}. They proved that the proposed $p$-th-order method enjoys a complexity bound of $\bigO(\epsilon^{-\frac{2}{p+1}})$ for smooth convex-concave saddle point problems with Lipschitz $p$-th-order derivative. 
Later, the authors in \cite{Ostroukhov2020} built upon this work and proposed a restarted high-order mirror-prox method in the strongly-convex-strongly-concave setting, achieving a complexity bound of $\bigO(\kappa_p^{\frac{2}{p+1}}\log (\epsilon^{-1}))$. 
Under the additional assumption that the gradient of $f$ is Lipschitz, they further combined the restarted scheme with the cubic regularized Newton method in \cite{Huang2022cubic}, leading to an overall complexity bound of $\bigO(\kappa_p^{\frac{2}{p+1}}\log (\kappa_1 \kappa_2)+\log\log(\epsilon^{-1}))$. While these works achieve 
comparable iteration complexities as ours, we remark that their overall complexity is worse due to the additional cost caused by their line search scheme.  
Specifically, the proposed two-stage line search subroutine in \cite{Monteiro2012} requires $O(\log\log(\epsilon^{-1}))$ calls to a subproblem solver per iteration. A simpler binary search subroutine is also considered in \cite{Bullins2022higher} with a higher per-iteration cost of $\mathcal{O}(\log(\epsilon^{-1}))$. In comparison, we prove that our line search scheme on average requires a constant number of calls to a subproblem solver per iteration, thus shaving the additional factor of $\bigO(\log\log(\epsilon^{-1}))$ from the computational complexity.
Some existing results as well as ours are summarized in Table~\ref{tab:high_order}.

{{{After the initial version of this paper was released}}, there have been new developments regarding second-order and higher-order methods for saddle point problems \cite{adil2022optimal,lin2022perseus,lin2022explicit, huang2022approximation,alves2023search}. Specifically, in order to eliminate the need of a line search scheme, the methods proposed in \cite{adil2022optimal,lin2022perseus,lin2022explicit,huang2022approximation} combine the idea of cubic regularization~\cite{Nesterov2006} with either mirror-prox~\cite{Nemirovski2004} or dual extrapolation~\cite{Nesterov2007a}.  
In the convex-concave setting, their $p$-th-order method all achieve a complexity of $\mathcal{O}(\epsilon^{-\frac{2}{p+1}})$, which matches the lower bound proved in \cite{adil2022optimal,lin2022perseus} under certain assumptions (see \cite{lin2022perseus} for details). In the strongly-convex-strongly-concave setting, the authors in \cite{lin2022perseus} further proposed a restart scheme to obtain a global complexity of $\bigO(\kappa_p^{\frac{2}{p+1}}\log(\epsilon^{-1}))$ as well as local superlinear convergence. Later, a similar restart scheme is also studied in \cite{huang2022approximation} with a more refined analysis, leading to an improved global complexity of $\bigO(\kappa_p^{\frac{2}{p+1}}\log\log(\epsilon^{-1}))$. 
However, while these methods do not require an explicit line search subroutine, in each iteration they need to solve one nontrivial $(p+1)$-th-order regularization subproblem, which can be computational costly. For instance, in the unconstrained setting, the authors in \cite{adil2022optimal} proposed a binary search scheme for their second-order method using $\bigO(\log(\epsilon^{-1}))$ calls to a linear system solver. 
In comparison, on average each iteration of our second-order method requires solving a constant number of subproblems, which are in the form of affine variational inequalities and further reduce to linear systems of equations in the unconstrained setting. Hence, the overall complexity of our second-order method can be advantageous at least in the unconstrained setting.

Along another line of research, the concurrent work in \cite{alves2023search} built on the large-step HPE framework \cite{Monteiro2012,svaiter2023complexity} and proposed a new line-search free second-order method. Using the same subproblem solver as ours, it is shown to achieve an overall complexity of $\bigO({\epsilon}^{-\frac{2}{3}}+\log(\epsilon^{-1}))$ for monotone variational inequalities. On the other hand, we note that their approach requires prior knowledge of the Lipschitz constant of the Hessian, which is not needed in our line search scheme. Moreover, our analysis covers both the convex-concave and strongly-convex-strongly-concave cases, while it remains unclear how to extend the method in \cite{alves2023search} to the latter setting. 
}

\noindent\textbf{Outline.} The rest of the paper is organized as follows. In Section~\ref{sec:preliminaries}, we present some definitions and preliminaries on monotone operator theory and saddle point problems. In Section~\ref{sec:BPP}, we review the classical Bregman PPM and show how some popular first-order methods can be viewed as its approximations. We propose our generalized optimistic method and the backtracking line search scheme in Section~\ref{sec:GOMD}. The first-order, second-order and higher-order optimistic method are discussed in Sections~\ref{sec:first_order}-\ref{sec:high_order}, respectively. We report our numerical results in Section~\ref{sec:numerical}. Finally, we conclude the paper with some additional remarks in Section~\ref{sec:conclusions}. 

%% file: preliminaries.tex
\vspace{-2mm}
\section{Preliminaries}\label{sec:preliminaries}
\vspace{-1mm}

\vspace{-1mm}
\subsection{Bregman distances and strong convexity}
\vspace{-1mm}

For an arbitrary norm $\|\cdot\|$, we denote its dual norm by $\|\cdot\|_{*}$. To deal with general non-Euclidean norms, we use the \emph{Bregman distance} to measure the proximity between two points, {defined as $D_{\Phi}(\vx',\vx) := \Phi(\vx')-\Phi(\vx)-\langle\nabla \Phi(\vx), \vx'-\vx\rangle$,
where the \emph{distance-generating function} $\Phi$ is differentiable and strictly convex.} {When} $\Phi$ is the squared Euclidean norm $\frac{1}{2}\|\cdot\|^2_2$, the Bregman distance is given by $D_{\Phi}(\vx',\vx)=\frac{1}{2}\|\vx'-\vx\|_2^2$. In the following, we refer to this important special case as the ``Euclidean setup''. 
{The main motivation of using non-Euclidean norms and Bregman distances is to achieve a better dependence on the dimension by adapting to the geometry of the problem. 
This technique is particularly useful when the variables are constrained to the simplex $\mathrm{\Delta}_n=\{\vx\in \mathbb{R}^n_+:\sum_{i=1}^n x_i =1\}$ or the spectrahedron $\mathcal{S}_n = \{\mX\in \mathbb{S}_+^n: \mathrm{Tr}(\mX)=1\}$. In the former case, we could use the $\ell_1$-norm and the negative entropy $\Phi(\vx)=\sum_{i} x_i \log x_i$, while in the latter case we could use the nuclear norm and the negative von Neumann entropy $\Phi(\mX)=\sum_{i=1}^n \lambda_i(\mX) \log \lambda_i(\mX)$, where $\lambda_i(\mX)$ denotes the $i$-th eigenvalue of the matrix $\mX$.
Such constraints can arise naturally in several typical applications of minimax optimization such as matrix game and semidefinite programming; see \cite{Nemirovski2004} for detailed discussions.} 
Using Bregman distance, one can define a more general notion of strong convexity (e.g., \cite{Bauschke2017,Lu2018}). 

\vspace{-1mm}
\begin{definition}\label{def:strong_convex_Bregman}
    A differentiable function $g$  is \emph{$\mu$-strongly convex} w.r.t.\@ a convex function $\Phi$ on a convex set $\mathcal{X}$ if 
    $
        g(\vx')\geq g(\vx)+\langle\nabla g(\vx),\vx'-\vx\rangle+\mu D_{\Phi}(\vx',\vx),
   $ for all $\vx,\vx'\in \mathcal{X}$.
    Furthermore,
    $g$ is \emph{$\mu$-strongly concave} if $-g$ is $\mu$-strongly convex. 
\end{definition}

In the Euclidean setup, the strong convexity defined in Definition~\ref{def:strong_convex_Bregman} coincides with the standard definition of strong convexity. %
Also note that by setting $\mu=0$ in the above inequality, we recover the definition for a convex function $g$. Similarly, we state that $g$ is concave if $-g$ is convex.

\vspace{-2mm}
\subsection{Monotone operators and Lipschitz and smooth operators}
\label{subsec:monotone}

\vspace{-1mm}
Since our convergence analysis relies heavily on the monotone operator theory, next we review some of its basic concepts. A \emph{set-valued operator} $T:\mathbb{R}^d \rightrightarrows\mathbb{R}^d$ maps a point $\vz\in \mathbb{R}^d$ to a 
(possibly empty) subset of $\mathbb{R}^d$, and its domain is given by $\dom T=\{\vz\in \mathbb{R}^d:\; T(\vz)\neq \emptyset\}$. Alternatively, we can identify $T$ with its graph defined as 
 $  \mathrm{Gr}(T) = \{(\vz,\vv):\;\vv\in T(\vz)\}\subset \mathbb{R}^d \times \mathbb{R}^d$.
With these notations we introduce the class of monotone operators, which is closely related to convex functions. 

\vspace{-1mm}
\begin{definition}
    A set-valued operator $T:\mathbb{R}^d \rightrightarrows \mathbb{R}^d$ is \emph{monotone} if {$\langle \vv-\vv', \vz-\vz' \rangle \geq 0$ for any $(\vz,\vv),(\vz',\vv')\in \mathrm{Gr}(T)$.} 
    Also, a monotone operator $T$ is \emph{maximal} if there is no other monotone operator $S$ that 
$\mathrm{Gr}(T) \subsetneq \mathrm{Gr}(S)$.
\end{definition}
Similar to Definition~\ref{def:strong_convex_Bregman}, we introduce a stronger notion of monotonicity using Bregman distance.
\begin{definition}
    A set-valued operator $T:\mathbb{R}^d \rightrightarrows \mathbb{R}^d$ is \emph{$\mu$-strongly monotone} w.r.t.\@ a convex function $\Phi$ if for any $ (\vz,\vv),(\vz',\vv')\in \mathrm{Gr}(T)$, we have {$\langle \vv-\vv', \vz-\vz' \rangle \geq \mu \langle \nabla\Phi(\vz)-\nabla\Phi(\vz'), \vz-\vz' \rangle$.}
\end{definition}

 For a smooth operator $F:\dom F\rightarrow \mathbb{R}^d$ defined on an open domain $\dom F\subset \mathbb{R}^d$, we use the notation $
    D^{p}F(\vz)[\vh_1,\vh_2,\ldots,\vh_p]\in \mathbb{R}^d
$ to denote the $p$-th-order directional derivative of $F$ at point $\vz$ along the directions $\vh_1,\ldots,\vh_p\in \mathbb{R}^d$. Moreover, when $\vh_1,\ldots,\vh_p$ are all identical, we abbreviate the directional derivative as $D^{p}F(\vz)[\vh]^p$. Note that the $p$-th-order derivative tensor $D^{p}F(\vz)[\cdot]$ is a multilinear map, and it coincides with the Jacobian of $F$ when $p=1$. Recalling that $\|\cdot\|_*$ is the dual norm of $\|\cdot\|$, we define the \emph{operator norm} of the tensor induced by a norm $\|\cdot\|$ on $\mathbb{R}^d$ by 
\begin{equation*}
    \|D^{p}F(\vz)\|_{\op}=\max_{\|\vh_i\|=1,i=1,\ldots,p} \|D^{p}F(\vz)[\vh_1,\ldots,\vh_p]\|_{*}.
\end{equation*}
In this paper, we focus on Lipschitz continuous and smooth operators, defined as follows. 
\begin{definition}\label{def:smooth}
    An operator $F:\dom F\rightarrow \mathbb{R}^d$ is \emph{$L_1$-Lipschitz} if {$\|F(\vz)-F(\vz')\|_{*} \leq L_1\|\vz-\vz'\|$ for any $\vz,\vz'\in \dom F$.}
    Moreover, for $p\geq 2$, $F$ is \emph{$p$-th-order $L_p$-smooth} if {$\|D^{p-1}F(\vz)-D^{p-1}F(\vz')\|_{\op} \leq L_p\|\vz-\vz'\|$ for any $\vz,\vz'\in \dom F$.} 
\end{definition}

We define the \emph{$p$-th-order Taylor expansion} of $F$ at point $\vz$ by   
\begin{equation}\label{eq:taylor_expansion}
    T^{(p)}(\vz';\vz) = F(\vz)+\sum_{i=1}^p \frac{1}{i!}D^{i}F(\vz)[\vz'-\vz]^i.
\end{equation}
For a $p$-th-order smooth operator $F$, its $(p-1)$-th-order Taylor expansion can be considered as its approximation. Specifically, if $F$ is $p$-th-order $L_p$-smooth, we have
\begin{align}
    \|F(\vz')-T^{(p-1)}(\vz';\vz)\|_{*} &\leq \frac{L_{p}}{p!}\|\vz-\vz'\|^p,\qquad \forall \vz,\vz'\in \dom F; \label{eq:residual_F}\\
    \|DF(\vz')-DT^{(p-1)}(\vz';\vz)\|_{\op} &\leq \frac{L_{p}}{(p-1)!}\|\vz-\vz'\|^{p-1},\qquad \forall \vz,\vz'\in \dom F.\label{eq:residual_DF}
\end{align}

\vspace{-3mm}
\subsection{Saddle point problems}\label{subsec:saddle_point_problem}
Recall the convex-concave minimax problem defined in \eqref{eq:minimax}. Its optimal solution $(\vx^*,\vy^*)$ is also called a \emph{saddle point} of the objective function $\ell$, as it satisfies $
    \ell(\vx^*,\vy)\leq \ell(\vx^*,\vy^*)\leq \ell(\vx,\vy^*)$ for any $ \vx\in \mathcal{X}$ and any $\vy\in \mathcal{Y}$.
Throughout the paper, we assume that Problem \eqref{eq:minimax} has at least one saddle point in $\mathcal{X} \times \mathcal{Y}$. 
There are several ways to reformulate the saddle point problem. To simplify the notation, let $\vz:=(\vx,\vy)\in \mathbb{R}^{m+n}$ and $\mathcal{Z} := \mathcal{X} \times \mathcal{Y}$. Define the operators 
$F: \mathcal{Z} \rightarrow \mathbb{R}^{m+n}$ and $H: \mathbb{R}^{m+n} \rightrightarrows \mathbb{R}^{m+n}$ by 
\begin{equation}\label{eq:def_of_F_H}
    F(\vz) := 
    \begin{bmatrix}
        \nabla_{\vx} f(\vx,\vy) \\
        -\nabla_{\vy} f(\vx,\vy)
    \end{bmatrix}
    \;\text{and}\;
    H(\vz) := 
    \begin{bmatrix}
        \partial h_1(\vx) \\
        \partial h_2(\vy)
    \end{bmatrix}+
    \begin{bmatrix}
        N_\mathcal{X}(\vx) \\
        N_\mathcal{Y}(\vy)
    \end{bmatrix},
\end{equation} 
{where $\partial h_1$ and $\partial h_2$ denote the subdifferentials of $h_1$ and $h_2$, respectively, and $N_\mathcal{X}$ and $N_\mathcal{Y}$ denote the normal cone operators of $\mathcal{X}$ and $\mathcal{Y}$, respectively.} 
Then by the first-order optimality condition, $\vz^*$ is a saddle point of \eqref{eq:minimax} if and only if it solves the following monotone inclusion problem: 
\begin{equation}\label{eq:monotone_inclusion}
    \text{find}\;\vz^*\in \mathbb{R}^{m+n}\;
    \text{such that}\;
    0\in F(\vz^*)+H(\vz^*).
\end{equation}
Moreover, if we define $h(\vz):=h_1(\vx)+h_2(\vy)$, \eqref{eq:minimax} is also equivalent to the variational inequality: 
\begin{equation}\label{eq:VI}
    \text{find}\;\vz^*\in \mathcal{Z}\;
    \text{such that}\;
    \langle F(\vz^*),\vz-\vz^* \rangle+h(\vz)-h(\vz^*)\geq 0,\;\forall \vz\in \mathcal{Z}.
\end{equation}

We assume that $\mathcal{X}$ and $\mathcal{Y}$ are equipped with the norms $\|\cdot\|_{\mathcal{X}}$ and $\|\cdot\|_{\mathcal{Y}}$, respectively. Moreover, we have two distance-generating functions --- a function 
$\Phi_{\mathcal{X}}:\mathcal{X}\rightarrow \mathbb{R}$ that is $1$-strongly convex w.r.t.\@ $\|\cdot\|_{\mathcal{X}}$ and a function $\Phi_{\mathcal{Y}}:\mathcal{Y}\rightarrow \mathbb{R}$ that is $1$-strongly convex w.r.t.\@ $\|\cdot\|_{\mathcal{Y}}$. We endow $\mathcal{Z}=\mathcal{X}\times \mathcal{Y}$ with the norm $\|\cdot\|_{\mathcal{Z}}$ defined by 
   $ \|\vz\|_{\mathcal{Z}} := \sqrt{\|\vx\|_{\mathcal{X}}^2+\|\vy\|_{\mathcal{Y}}^2}$ for all $\vz=(\vx,\vy)\in \mathcal{Z}$,
and with the distance-generating function $\Phi_{\mathcal{Z}}:\mathcal{Z} \rightarrow \mathbb{R}$ defined by  
$\Phi_{\mathcal{Z}}(\vz) := \Phi_{\mathcal{X}}(\vx)+\Phi_{\mathcal{Y}}(\vy)$. 
It can be verified that 
\begin{equation}\label{eq:mirror_map}
    D_{\Phi_{\mathcal{Z}}}(\vz',\vz) = D_{\Phi_{\mathcal{X}}}(\vx',\vx)+D_{\Phi_{\mathcal{Y}}}(\vy',\vy),
\end{equation}
and $\Phi_{\mathcal{Z}}$ is $1$-strongly convex w.r.t.\@ $\|\cdot\|_{\mathcal{Z}}$. To simplify the notation, we omit $\mathcal{Z}$ in the subscripts of $\|\cdot\|_{_{\mathcal{Z}}}$ and $\Phi_{\mathcal{Z}}$ when there is no ambiguity. 

In this paper, we consider two specific problem classes for our theoretical results: (i) convex-concave saddle point problems (Assumption~\ref{assum:monotone}) and (ii) strongly-convex-strongly-concave saddle point problems (Assumption~\ref{assum:strongly_monotone}). Note that Assumption~\ref{assum:strongly_monotone} reduces to Assumption~\ref{assum:monotone} when $\mu=0$, and sometimes this allows us to unify our results in both settings by considering $\mu\geq 0$. 
\begin{assumption}\label{assum:monotone}
The functions $h_1$ and $h_2$ are proper closed convex, and the function $f$ is convex in $\vx$ and concave in $\vy$ on $\mathcal{X} \times \mathcal{Y}$. %
\end{assumption}
\begin{assumption}\label{assum:strongly_monotone}
The functions $h_1$ and $h_2$ are proper closed convex, and
the function $f$ is $\mu$-strongly convex in $\vx$ w.r.t.\@ $\Phi_{\mathcal{X}}$ and 
$\mu$-strongly concave in $\vy$ w.r.t.\@ $\Phi_{\mathcal{Y}}$ on $\mathcal{X}\times \mathcal{Y}$. %
\end{assumption}

As mentioned in Section~\ref{subsec:monotone}, we will mainly work with the operator $F$ to develop our algorithms and convergence proofs in this paper. The following lemma serves the purpose of translating the properties of the function $f$ into the properties of the operator $F$ (see Appendix~\ref{appen:function_operator} for the proof). 

\vspace{-1mm}
\begin{lemma}\label{lem:function_operator}
    The operator $F$ defined in \eqref{eq:def_of_F_H} is monotone on $\mathcal{Z}$ under Assumption~\ref{assum:monotone}, and is $\mu$-strongly monotone w.r.t. $\Phi_{\mathcal{Z}}$ on $\mathcal{Z}$ under Assumption~\ref{assum:strongly_monotone}.
\end{lemma}

\vspace{-1mm}
In the convex-concave case, we measure the suboptimality of a solution $(\vx,\vy)$ by
    $\Delta(\vx,\vy) := 
    \max_{\vy'\in \mathcal{Y}} \ell(\vx,\vy')-\min_{\vx'\in \mathcal{X}} \ell(\vx',\vy)$,
which is known as the \textit{primal-dual gap}. 
We note that $\Delta(\vx,\vy)\geq 0$ for any $(\vx,\vy)\in \mathcal{X}\times \mathcal{Y}$ and the equality holds if and only if $(\vx,\vy)$ is a saddle point. However, in some cases where the feasible sets are unbounded, the primal-dual gap can be always infinite except at the saddle points\footnote{For example, consider the unconstrained saddle point problem $\min_{\vx\in \mathbb{R}^m} \max_{\vy \in \mathbb{R}^m} \langle \vx,\vy \rangle$. Then the primal-dual gap is infinite except at the unique saddle point $(0,0)$.}, rendering it useless. One remedy is to use the restricted primal-dual gap function \cite{Nesterov2007a,Chambolle2011}. Given two bounded sets $\mathcal{B}_1\subset \mathcal{X}$ and $\mathcal{B}_2\subset \mathcal{Y}$,  
we define the \emph{restricted primal-dual gap} as
\begin{equation}\label{eq:restricted_gap_function}
    \Delta_{\mathcal{B}_1 \times \mathcal{B}_2}(\vx,\vy) := 
    \max_{\vy'\in \mathcal{B}_2 } \ell(\vx,\vy')-\min_{\vx'\in \mathcal{B}_1} \ell(\vx',\vy).
\end{equation}
As discussed in \cite{Chambolle2011}, the gap function has two properties: (i) If $(\vx^*,\vy^*)\in \mathcal{B}_1\times \mathcal{B}_2$, then we have $\Delta_{\mathcal{B}_1 \times \mathcal{B}_2}(\vx,\vy)\geq 0$ for any $(\vx,\vy)\in \mathcal{X}\times \mathcal{Y}$; (ii) If $(\vx,\vy)$ is in the interior of $\mathcal{B}_1\times \mathcal{B}_2$,
    then $\Delta_{\mathcal{B}_1 \times \mathcal{B}_2}(\vx,\vy)= 0$ if and only if $(\vx,\vy)$ is a saddle point.  %
As we shall show later, the iterates of our algorithm stay in a bounded set centered at $(\vx^*,\vy^*)$. Hence, the restricted primal-dual gap function will serve as a good measure of suboptimality if we choose $\mathcal{B}_1 \times \mathcal{B}_2$ large enough such that it contains a saddle point $(\vx^*,\vy^*)$.

The following classical result is the first step of our analysis, providing an upper bound on the primal-dual gap at the averaged iterate. The proof is in Appendix~\ref{appen:function_operator}.
\vspace{-1mm}
\begin{lemma}\label{lem:duality_gap}
    Suppose that Assumption~\ref{assum:monotone} holds. Then for any 
    $\vz,\vz_1,\ldots,\vz_N\in \mathcal{Z}$ and $\theta_1,\ldots,\theta_N\geq 0$ with 
    $\sum_{k=1}^N \theta_k=1$, we have {$\ell(\bar{\vx}_N,\vy)-\ell(\vx,\bar{\vy}_N) \leq \sum_{k=1}^N \theta_k
    ( \langle F(\vz_k), \vz_k-\vz \rangle
    + h(\vz_k)-h(\vz))$},
    where %
    $\bar{\vx}_N = \sum_{k=1}^N \theta_k \vx_k$ and 
    $\bar{\vy}_N = \sum_{k=1}^N \theta_k \vy_k$. 
\end{lemma}
\vspace{-1mm}
\begin{remark}
    For brevity, we report our convergence results in the form of an upper bound on $\ell(\bar{\vx}_N,\vy)-\ell(\vx,\bar{\vy}_N)$ for \emph{any} $(\vx,\vy)\in \mathcal{X}\times \mathcal{Y}$ (e.g., see Proposition~\ref{prop:BPP}(a)). Then taking the supremum of both sides over $\vx\in \mathcal{B}_1$ and $\vy\in \mathcal{B}_2$ results in an upper bound on the restricted primal-dual gap in \eqref{eq:restricted_gap_function}. 
\end{remark}
\vspace{-1mm}
In the strongly-convex-strongly-concave setting, Problem \eqref{eq:minimax} has a unique saddle point $\vz^*\in \mathcal{Z}$, and we measure the suboptimality of $\vz$ by the Bregman distance $D_{\Phi}(\vz^*,\vz)$. The following lemma presents the key property we shall use in our convergence analysis. The proof is available in Appendix~\ref{sec:proof_of_lem:saddle_point}.
\vspace{-1mm}
\begin{lemma}\label{lem:saddle_point}
    If Assumption~\ref{assum:strongly_monotone} holds with $\mu \geq 0$ and $\vz^*$ is a saddle point of Problem~\eqref{eq:minimax}, then for any $\vz\in \mathcal{Z}$, we have $\langle F(\vz), \vz-\vz^* \rangle +h(\vz)-h(\vz^*) \geq \mu D_{\Phi}(\vz^*,\vz)$.
\end{lemma}

\begin{remark}\label{rem:symmetry} 
    {The \emph{symmetry coefficient}  of $\Phi$   is defined  as 
    $%
        s(\Phi) := \inf_{\vx,\vx'\in \mathrm{int}\dom \Phi} \frac{D_{\Phi}(\vx,\vx')}{D_{\Phi}(\vx',\vx)}
        \in [0,1]
    $ in \cite{Bauschke2017},
    where $\mathrm{int}\dom \Phi$ is the interior of the domain of $\Phi$. In particular, we have $s(\Phi)=1$ when $\Phi=\frac{1}{2}\|\cdot\|_2^2$.}
     It is not hard to see from the proof of Lemma~\ref{lem:saddle_point} that {the right-hand side of the inequality} can be improved to $(1+s(\Phi))\mu D_{\Phi}(\vz^*,\vz)$.
    For simplicity, we omit the constant $1+s(\Phi)$ when reporting our main convergence results, but we will take it into consideration in Section~\ref{sec:numerical} for a more accurate comparison between our theoretical analysis and numerical results. Moreover, our convergence results still hold if we replace Assumption~\ref{assum:strongly_monotone} with the weaker assumption that the inequality in {Lemma~\ref{lem:saddle_point}} is always satisfied, which is referred to as the generalized monotonicity condition in \cite{kotsalis2022simple}.
\end{remark}

%% file: proximal.tex
\section{Bregman proximal point method and its approximations}\label{sec:BPP}
As mentioned in the introduction, %
the generalized optimistic method we propose in this paper can be interpreted as a systematic approach to approximating the classical PPM. Hence, to better motivate our method, we first review the basics of the Bregman PPM \cite{Chen1993,Eckstein1993}. In the $k$-th iteration of the Bregman PPM for solving Problem~\eqref{eq:minimax}, the new point $\vz_{k+1}$ is given by the unique solution of the monotone inclusion subproblem
\begin{equation}\label{eq:BPP_inclusion}
    0 \in \eta_kF(\vz)+\eta_k H(\vz) + \nabla \Phi (\vz)- \nabla \Phi (\vz_k),
\end{equation}
where $\eta_k>0$ is the ``step size''. Equivalently, we can also write
\begin{equation}\label{eq:BPP}
    \vz_{k+1} = \argmin_{\vz\in \mathcal{Z}} \big\{\eta_{k}\langle  F(\vz_{k+1}),\vz-\vz_k \rangle +\eta_k h(\vz) +D_{\Phi}(\vz,\vz_k)\big\}.
\end{equation}
Note, however, that \eqref{eq:BPP} is not an explicit update rule since its right-hand side also depends on $\vz_{k+1}$. 
Regarding its convergence property, we have the following results. The proofs are presented in Appendix~\ref{sec:proof_of_BPP}.
While the result for the convex-concave setting is well-known in the literature \cite{Nemirovski2004}, we note that the result for the strongly-convex-strongly-concave setting appears to be new\footnote{Rockafellar proved a similar result for the proximal point method in \cite{Rockafellar1976} in the {Euclidean} setup, while our result also applies to a general Bergman distance.}.
\begin{proposition}\label{prop:BPP}
    Let $\{\vz_k\}_{k\geq 0}$ be the iterates generated by the Bregman PPM in \eqref{eq:BPP_inclusion}. 
    \begin{enumerate}[label=(\alph*)]
        \item \label{item:BPP_a} Under Assumption~\ref{assum:monotone}, we have $D_{\Phi}(\vz^*,\vz_{k+1})\leq D_{\Phi}(\vz^*,\vz_k)$ for any $k \geq 0$.
        Moreover,  we have  
       $
            \ell(\bar{\vx}_N,\vy)-\ell(\vx,\bar{\vy}_N)
            \leq D_{\Phi}(\vz,\vz_0)\left(\sum_{k=0}^{N-1} \eta_k\right)^{-1} 
        $ for any $\vz = (\vx,\vy)\in \mathcal{X}\times \mathcal{Y}$, 
        where $\bar{\vz}_N = (\bar{\vx}_N,\bar{\vy}_N)$ is given by $\bar{\vz}_N=\frac{1}{\sum_{k=0}^{N-1} \eta_k}\sum_{k=0}^{N-1} \eta_k\vz_{k+1}$. 
        \item \label{item:BPP_b} Under Assumption~\ref{assum:strongly_monotone}, we have 
        $
            D_{\Phi}(\vz^*,\vz_N)\leq D_{\Phi}(\vz^*,\vz_0) \prod_{k=0}^{N-1} (1+\eta_k\mu)^{-1}.
        $
    \end{enumerate}
\end{proposition}
Proposition~\ref{prop:BPP} characterizes the convergence rates of the Bregman PPM in terms of the step sizes $\{\eta_k\}_{k\geq 0}$. 
We note that it requires no condition on the step sizes $\eta_k$, and hence in theory Bregman PPM can converge arbitrarily fast with sufficiently large step sizes. 
On the other hand, each step of Bregman PPM involves solving a highly nontrivial monotone inclusion problem in \eqref{eq:BPP_inclusion} and can be computationally intractable. Hence, our goal is to introduce a general class of optimistic methods that approximates the Bregman PPM by replacing \eqref{eq:BPP_inclusion} with more tractable subproblems. As we shall see, {our generalized optimistic method in Algorithm~\ref{alg:GOMD}} can achieve the same convergence guarantees (up to constant) as in Proposition~\ref{prop:BPP}, but {unlike Bregman PPM} it introduces additional restrictions on the {allowable step sizes}.  

To leverage the superior convergence performance of the Bregman PPM while maintaining computational tractability,  several iterative algorithms have been developed that aim at approximating Bregman PPM with explicit and efficient update rules. In particular, we describe two first-order methods, which represent the two different frameworks we reviewed in Section~\ref{subsec:related_work}: (i) The mirror-prox method and (ii) The optimistic mirror descent method. 

\vspace{2mm}
\noindent\textbf{Mirror-Prox method.} The key insight of the mirror-prox method %
\cite{Nemirovski2004} 
is that when $F$ is Lipschitz continuous,
we can approximately solve \eqref{eq:BPP} in two steps per iteration, with a properly chosen step size. Specifically, the mirror-prox update is given by 
\begin{equation}\label{eq:mirror_prox}
    \begin{aligned}
        \vz_{k+1/2} &= \argmin_{\vz\in \mathcal{Z}}\big\{\eta_k\langle  F(\vz_{k}),\vz-\vz_k \rangle + \eta_k h(\vz)+D_{\Phi}(\vz,\vz_k)\big\},\\
        \vz_{k+1} &= \argmin_{\vz\in \mathcal{Z}}\big\{\eta_k\langle  F(\vz_{k+1/2}),\vz-\vz_k \rangle +\eta_k h(\vz)+D_{\Phi}(\vz,\vz_k)\big\}.
    \end{aligned}
\end{equation}
In words, we first take a step of mirror descent to the middle point 
$\vz_{k+1/2}$, and then use $F(\vz_{k+1/2})$ as a surrogate for $F(\vz_{k+1})$ (c.f.~\eqref{eq:BPP}). In the Euclidean case, it boils down to the extragradient method \cite{Korpelevich1976}. %

\vspace{2mm}
\noindent\textbf{Optimistic mirror descent method.} Another approach for properly approximating proximal point updates is the optimistic mirror descent method, which dates back to \cite{Popov1980}. Later on, it also gained attention in the context of online learning \cite{Chiang2012,Rakhlin2013,Rakhlin13Optimization,joulani2020modular} and Generative Adversarial Networks (GANs) \cite{Daskalakis2018,gidel2018a,mertikopoulos2018optimistic,Liang2019,Peng2020}. Specifically, the update rule of this method for constrained saddle point problems is
\begin{equation}\label{eq:OMD_original}
    \begin{aligned}
        \vz_{k} &= \argmin_{\vz\in \mathcal{Z}}\big\{\eta_{k}\langle  F(\vz_{k-1}),\vz-\vz_{k-1/2} \rangle +D_{\Phi}(\vz,\vz_{k-1/2})\big\},\\
        \vz_{k+1/2} &= \argmin_{\vz\in \mathcal{Z}}\big\{\eta_{k}\langle  F(\vz_{k}),\vz-\vz_{k-1/2} \rangle +D_{\Phi}(\vz,\vz_{k-1/2})\big\}.
    \end{aligned}
\end{equation}
Comparing \eqref{eq:OMD_original} with \eqref{eq:mirror_prox}, we observe that optimistic mirror descent operates in a similar way as the mirror-prox method, but instead of computing $F(\vz_{k-1/2})$, we reuse the gradient $F(\vz_{k-1})$ from the previous iteration to obtain $\vz_k$. Hence, each iteration of this algorithm requires one gradient computation, unlike the mirror-prox update which requires two gradient evaluations per iteration.

The discussions above view the ``midpoints'' $\{\vz_{k+{1}/{2}}\}_{k\geq 0}$ as the approximate iterates of the Bregman PPM. Alternatively, we can also interpret $\{\vz_k\}_{k\geq 0}$ as the approximate iterates in the special setting of unconstrained saddle point problems in the Euclidean setup. Such viewpoint was first discussed in \cite{Mokhtari2020a}. 
In this case, the Bregman PPM in \eqref{eq:BPP_inclusion} can be written as the implicit update 
\begin{equation}\label{eq:BPP_simple}
    \vz_{k+1} = \vz_k-\eta_k F(\vz_{k+1}),
\end{equation}
while \eqref{eq:OMD_original} leads to the simple update rule 
\begin{equation}\label{eq:ogda_update}
    \vz_{k+1} = \vz_{k}-2\eta_{k} F(\vz_k)+\eta_{k} F(\vz_{k-1}) = \vz_{k}-\eta_{k} F(\vz_k)-\eta_{k}(F(\vz_{k})-F(\vz_{k-1})), 
\end{equation}
also known as OGDA %
\cite{Mokhtari2020a,Mokhtari2020}. 
To better illustrate the main idea, in \eqref{eq:ogda_update} we write the OGDA update as a combination of two terms: a ``prediction term'' $F(\vz_k)$ that serves as a surrogate for $F(\vz_{k+1})$, and a ``correction term'' $F(\vz_k)-F(\vz_{k-1})$ given by the deviation between $F(\vz_k)$ and its prediction in the previous iteration. If we make the optimistic assumption that the prediction error between two consecutive iterations does not change significantly, 
i.e., $F(\vz_{k+1})-F(\vz_{k})\approx F(\vz_{k})-F(\vz_{k-1})$, then it holds that $F(\vz_{k+1})\approx 2F(\vz_k)-F(\vz_{k-1})$ and we can expect \eqref{eq:ogda_update} to be a good approximation for \eqref{eq:BPP_simple}.  

Finally, we note that mirror-prox and optimistic mirror descent methods only use first-order information to approximate the Bregman PPM. It is natural to ask whether one can exploit higher-order information to improve the accuracy of approximation and hence achieve faster convergence. While previous works \cite{Monteiro2010,Monteiro2012,Bullins2022higher} have explored such possibility for the mirror-prox method, to the best of our knowledge, no similar effort has been made for the optimistic mirror descent method. Moreover, it remains unclear whether the current convergence analysis for the optimistic mirror decent method can be extended to more general composite saddle point problem. %
In this paper, we close these gaps and present a general framework for optimistic methods, following the viewpoint in \cite{Mokhtari2020a}. 

%% file: GOMD.tex
\section{Main algorithm: the generalized optimistic method}
\label{sec:GOMD}
In this section, we present the generalized optimistic method, which extends the ``optimistic'' idea described above in the OGDA method for approximating the Bregman PPM. Our proposed scheme is able to handle the composite structure in \eqref{eq:minimax} and can utilize second-order and higher-order information as we shall discuss in the upcoming sections.
To better explain the idea behind our generalized optimistic method, let us first focus on the unconstrained smooth saddle point problem. In the OGDA method defined in \eqref{eq:ogda_update}, the term $F(\vz_{k+1})$ is approximated by a linear combination of the prediction term $F(\vz_{k})$ and the correction term $F(\vz_k)- F(\vz_{k-1})$. 
Now instead of using $F(\vz_{k})$ as the prediction for $F(\vz_{k+1})$, suppose that we have access to a general approximation function denoted by
$$
\text{prediction term:} \qquad P(\vz;\mathcal{I}_k),
$$
which is constructed from the available information $\mathcal{I}_k$ up to the $k$-th iteration. %
For instance, %
with a second-order oracle we have $\mathcal{I}_k=\{F(\vz_0),DF(\vz_0),\dots,F(\vz_k),DF(\vz_k)\}$, and a possible candidate for $P$ is the first-order Taylor expansion $T^{(1)}(\vz;\vz_k)= F(\vz_k)+ DF(\vz_k)(\vz-\vz_k)$. 
Similar to the logic in OGDA, the correction term is then given %
by the difference between $F(\vz_{k})$ and our prediction in the previous step, i.e., 
$$
\text{correction term:} \qquad F(\vz_k)-P(\vz_k;\mathcal{I}_{k-1}).
$$
Indeed, if we are ``optimistic'' that the change in the prediction error between consecutive iterations is small, i.e., $ F(\vz_{k+1})-P(\vz_{k+1};\mathcal{I}_k) \approx F(\vz_k)-P(\vz_k;\mathcal{I}_{k-1})$, %
then this implies $F(\vz_{k+1})\approx P(\vz_{k+1};\mathcal{I}_k)+(F(\vz_k)-P(\vz_k;\mathcal{I}_{k-1}))$ and adding the correction term can help reducing the approximation error. 
Considering these generalizations, %
we propose to compute the new iterate by following the update 
$$
\vz_{k+1}= \vz_{k}-\left(\eta_k  P(\vz_{k+1};\mathcal{I}_k) + \hat{\eta}_k(F(\vz_k)-P(\vz_k;\mathcal{I}_{k-1}))\right).
$$
Unlike the OGDA method, this is not an explicit update in general and requires us to solve the equation with respect to $\vz_{k+1}$. Also, we use different step size parameters for the approximation term and correction term to make our algorithm as general as possible. 
In fact, our analysis suggests that the best choice of $\hat{\eta}_k$ is not always equal to $\eta_k$. As we shall see, we will set $\hat{\eta}_k=\eta_{k-1}$ in the convex-concave setting while set $\hat{\eta}_k=\eta_{k-1}/(1+\eta_{k-1}\mu)$ in the strongly-convex-strongly-concave setting. 

The same methodology can be readily generalized for the composite setting with a non-Euclidean norm.
Specifically, in the most general form of our proposed generalized optimistic method for solving Problem \eqref{eq:minimax}, the new point $\vz_{k+1}$ is given by the unique solution of the following inclusion problem
\begin{equation}\label{eq:GOM_inclusion}
    \begin{aligned}
        0\in \eta_k  P(\vz;\mathcal{I}_k) + \hat{\eta}_k(F(\vz_k)-P(\vz_k;\mathcal{I}_{k-1})) 
         +\eta_k H(\vz) + \nabla \Phi(\vz) - \nabla \Phi(\vz_k),
    \end{aligned}
\end{equation}
which can also be equivalently written as  
\begin{equation}\label{eq:GOM_argmin}
        \vz_{k+1}=\argmin_{\vz\in \mathcal{Z}} \{  \langle \eta_k  P(\vz_{k+1};\mathcal{I}_k) + \hat{\eta}_k(F(\vz_k)-P(\vz_k;\mathcal{I}_{k-1})),\vz-\vz_k \rangle
         +\eta_k h(\vz) +D_{\Phi}(\vz,\vz_k) \}.
\end{equation}
Comparing the update rules \eqref{eq:GOM_inclusion} and \eqref{eq:GOM_argmin} with the ones for the Bregman PPM in \eqref{eq:BPP_inclusion} and \eqref{eq:BPP}, we can see that the only modification is replacing the term $\eta_k F(\vz_{k+1})$ with its optimistic approximation $\eta_k P(\vz_{k+1};\mathcal{I}_k) + \hat{\eta}_k(F(\vz_k)-P(\vz_k;\mathcal{I}_{k-1}))$. Moreover, to ensure that $P(\vz_{k+1};\mathcal{I}_k)$ remains a valid approximation of $F(\vz_{k+1})$, 
we impose the following condition
\begin{equation}\label{eq:stepsize_upper_bound}
    \eta_k \|F(\vz_{k+1})-P(\vz_{k+1};\mathcal{I}_k)\|_{*}\leq \frac{\alpha}{2}\|\vz_{k+1}-\vz_k\|,
\end{equation}
where $\alpha\in(0,1]$ is a user-specified constant. The steps of the generalized optimistic method are summarized in Algorithm~\ref{alg:GOMD}.
\begin{algorithm}[!t]\small
    \caption{Generalized optimistic method for monotone inclusion problems}\label{alg:GOMD}
    \begin{algorithmic}[1]
        \STATE \textbf{Input:} initial point $\vz_0\in \mathcal{Z}$, approximation function $P$, strong convexity parameter $\mu\geq 0$, and $0<\alpha\leq 1\!$  
        \STATE \textbf{Initialize:} set $\vz_{-1}\leftarrow \vz_0$ and $P(\vz_0;\mathcal{I}_{-1})\leftarrow F(\vz_0)$
        \FOR{iteration $k=0,\ldots,N-1$}
        \STATE Choose $\eta_k,\hat{\eta}_k>0$ and compute $\vz_{k+1}$ such that
            \begin{gather*}
                0\in \eta_k  P(\vz_{k+1};\mathcal{I}_k) + \hat{\eta}_k(F(\vz_k)-P(\vz_k;\mathcal{I}_{k-1})) 
                +\eta_k H(\vz_{k+1}) + \nabla \Phi(\vz_{k+1}) - \nabla \Phi(\vz_k), \\
            \eta_k \|F(\vz_{k+1})-P(\vz_{k+1};\mathcal{I}_k)\|_{*}\leq \frac{\alpha}{2}\|\vz_{k+1}-\vz_k\|. 
            \end{gather*}
        \ENDFOR
    \end{algorithmic}
\end{algorithm}

We remark that our generalized optimistic method should be regarded as a general framework rather than a directly applicable algorithm. At this point, we do not specify how to choose $\eta_k$ and $\vz_{k+1}$ to satisfy the condition in \eqref{eq:stepsize_upper_bound}, which will depend on the properties of $F$ and the particular choice of $P$. A generic line search scheme will be described in Section~\ref{subsec:line_search}, and we will further devote Sections~\ref{sec:first_order}, \ref{sec:second_order}, and \ref{sec:high_order} to discuss the particular implementations in different settings.  

The following proposition forms the basis for all the convergence results derived later in this paper. Essentially, it shows that Algorithm~\ref{alg:GOMD} enjoys similar convergence guarantees to the Bregman PPM under the specified conditions, even if it only solves an approximated version of the subproblem in~\eqref{eq:BPP_inclusion}.  

\begin{proposition}\label{prop:GOMD}
    Let $\{\vz_k\}_{k\geq 0}$ be the iterates generated by Algorithm~\ref{alg:GOMD}.
     \vspace{-1mm}
    \begin{enumerate}[label=(\alph*)]
        \item \label{item:GOMD_monotone}
        Under Assumption~\ref{assum:monotone}, if we  set $\hat{\eta}_k = \eta_{k-1}$, we have
        \vspace{-2mm}
        \begin{equation}\label{eq:bounded_iterates}
            D_{\Phi}(\vz^*,\vz_k)\leq \frac{2}{2-\alpha}D_{\Phi}(\vz^*,\vz_0), \qquad \forall k \geq 0. 
        \end{equation}
        Moreover, 
        for any $\vz = (\vx,\vy)\in \mathcal{X}\times \mathcal{Y}$, we have  
        \vspace{-4mm}
        \begin{equation}\label{eq:duality_gap}
          \ell(\bar{\vx}_N,\vy)-\ell(\vx,\bar{\vy}_N) \leq D_{\Phi}(\vz,\vz_0)\left(\sum_{k=0}^{N-1} \eta_k\right)^{-1}, 
          \vspace{-2mm}
        \end{equation}
        where $\bar{\vz}_N = (\bar{\vx}_N,\bar{\vy}_N)$ is given by $\bar{\vz}_N=\frac{1}{\sum_{k=0}^{N-1} \eta_k}\sum_{k=0}^{N-1} \eta_k\vz_{k+1}$.
        \item \label{item:GOMD_strongly_monotone}
        \vspace{2mm}
        Under Assumption~\ref{assum:strongly_monotone}, if we set $\hat{\eta}_k = {\eta_{k-1}}/{(1+\mu\eta_{k-1})}$,  we have 
           \vspace{-2mm}
        \begin{equation*}
         D_{\Phi}(\vz^*,\vz_N)\leq \frac{2}{2-\alpha}D_{\Phi}(\vz^*,\vz_0) \prod_{k=0}^{N-1} (1+\eta_k\mu)^{-1}. 
        \end{equation*}
    \end{enumerate}
\end{proposition}
    \vspace{-5mm}
\begin{proof}
    See Appendix~\ref{appen:proof_of_GOMD}. 
\end{proof}

As in Proposition~\ref{prop:BPP} for the Bregman PPM, the convergence rate in Proposition~\ref{prop:GOMD} for the generalized optimistic method depends on the choice of $\{\eta_k\}_{k\geq 0}$. On the other hand, the step sizes here are not arbitrary but constrained by \eqref{eq:stepsize_upper_bound}. Depending on our choice of the approximation function~$P$, the condition on $\eta_k$ varies, and as a result, we obtain different rates. 
In fact, the step size $\eta_k$  is closely related to the displacement $\|\vz_{k+1}-\vz_k\|$, and more specifically it is approximately on the order of $\frac{1}{\|\vz_{k+1}-\vz_k\|^{p-1}}$ in our $p$-th-order method ($p\geq 2$).
Thus, next we provide an upper bound on a (weighted) sum of $\|\vz_{k+1}-\vz_k\|^2$, which will later translate into bounds on $\sum_{k=0}^{N-1}\eta_k$ and $\prod_{k=0}^{N-1}(1+\eta_k \mu)$. 
\begin{lemma}\label{lem:sum_of_square}
    Let $\{\vz_k\}_{k\geq 0}$ be the iterates generated by Algorithm~\ref{alg:GOMD} with $\hat{\eta}_k$ chosen as in Proposition~\ref{prop:GOMD}. Then under Assumption~\ref{assum:monotone}, we have $\sum_{k=0}^{N-1} \|\vz_{k+1}-\vz_{k}\|^2 \leq \frac{2}{1-\alpha}D_{\Phi}(\vz^*,\vz_0)$. Moreover,
    under Assumption~\ref{assum:strongly_monotone}, we have $\sum_{k=0}^{N-1} \left(\|\vz_{k+1}-\vz_{k}\|^2\prod_{l=0}^{k-1} \left(1+{\eta_l}\mu\right)\right) \leq \frac{2}{1-\alpha}D_{\Phi}(\vz^*,\vz_0)$.
\end{lemma}
\begin{proof}
    See Appendix~\ref{appen:proof_of_SOS}. 
\end{proof}

\vspace{-5mm}
\subsection{Backtracking line search}\label{subsec:line_search}
    \vspace{-2mm}

As mentioned earlier, our general scheme in Algorithm~\ref{alg:GOMD} does not directly lead to a practical algorithm, as it requires a proper policy to ensure that the condition in \eqref{eq:stepsize_upper_bound} holds. 
As the first step towards practical implementations, in this section, we propose a backtracking line search scheme to select the step size $\eta_k$ adaptively in Algorithm~\ref{alg:GOMD}. To simplify the notation, let $\vz^{-}=\vz_k$, $\eta=\eta_k$, $P(\vz)=P(\vz;\mathcal{I}_k)$, and $\vv^{-} = \vv_k := \hat{\eta}_k(F(\vz_k)-P(\vz_k;\mathcal{I}_{k-1}))$. In light of \eqref{eq:GOM_inclusion} and \eqref{eq:stepsize_upper_bound}, at the $k$-th iteration our goal is to find a step size $\eta>0$ and a point $\vz\in \mathcal{Z}$ such that 
\begin{gather}
    0 \in  \eta P(\vz)+\vv^{-}+\eta H(\vz)+\nabla \Phi(\vz)- \nabla \Phi(\vz^{-}), \label{eq:simplified_subproblem}\\
    \eta \|F(\vz) - P(\vz)\|_* \leq \frac{\alpha}{2}\|\vz-\vz^{-}\|.\label{eq:simplified_step_bound}
\end{gather}
Note that the conditions in \eqref{eq:simplified_subproblem} and \eqref{eq:simplified_step_bound} depend on both $\eta$ and $\vz$. Hence, in general we need to choose $\eta$ and $\vz$ simultaneously, except in the first-order setting (see Section~\ref{sec:first_order}).

We assume that the approximation function $P$ is maximal monotone and continuous on $\mathcal{Z}$.
As a corollary, for any fixed $\eta>0$, the monotone inclusion subproblem with respect to $\vz$ in \eqref{eq:simplified_subproblem} has a unique solution (see, e.g., \cite{Eckstein1993}), which we denote by $\vz(\eta;\vz^{-})$. Also, we assume access to a black-box solver of the subproblem in \eqref{eq:simplified_subproblem}, which we formally define as follows for later reference. 
\begin{definition}[Optimistic Subsolver]\label{def:subsolver}
    For any given $\eta$, $\vv^{-}$ and $\vz^{-}$, the \emph{optimistic subsolver} returns the solution of \eqref{eq:simplified_subproblem}. 
\end{definition}

Now we are ready to describe our line search scheme, which takes an initial trial step size $\sigma$ and a parameter $\beta \in (0,1)$ as inputs. 
At a high level, it works by repeatedly solving \eqref{eq:simplified_subproblem} for different step sizes in $\{\sigma \beta^i: i \geq 0\}$ and then checking the condition in \eqref{eq:simplified_step_bound}, each of which involves one call to the optimistic subsolver. 
Specifically, we first assign the value of $\sigma$ to the step size $\eta$ and compute $\vz$ as the solution of \eqref{eq:simplified_subproblem}.  
If the condition in \eqref{eq:simplified_step_bound} is satisfied with respect to $\eta$ and $\vz$, we 
accept $\eta$ as the final step size and choose $\vz$ as the next iterate. Otherwise, we decrease the step size by $\eta \leftarrow \beta \eta$ and recompute $\vz$ from \eqref{eq:simplified_subproblem}, 
and this backtracking process is repeated until $(\eta,\vz)$ satisfies the condition in 
\eqref{eq:simplified_step_bound}. The whole procedure is formally described in Algorithm~\ref{alg:ls}.
\begin{remark}
    The major computational cost of our line search scheme comes from solving Subproblem~\eqref{eq:simplified_subproblem} with different step sizes, namely, multiple calls to the optimistic subsolver. We note that the computational cost of the subsolver depends on the specific saddle point problem and our choice of the approximation function $P$. Hence, we use one call to the optimistic subsolver as the unit and measure the complexity of our line search scheme in terms of the number of calls to the solver.   
\end{remark}

\begin{algorithm}[!t]\small
    \caption{{Backtracking line search scheme}}\label{alg:ls}
    \begin{algorithmic}[1]
        \STATE \textbf{Input:} current iterate $\vz^{-}\in \mathcal{Z}$, $\vv^{-}\in \mathbb{R}^d$, initial trial step size $\sigma>0$, approximation function $P$, $\alpha\in (0,1]$, $\beta\in (0,1)$
        \STATE \textbf{Output:} step size $\eta$ and the next iterate $\vz\in \mathcal{Z}$
        \STATE Set $\eta \leftarrow \sigma$ and let $\vz$ be the solution of $0 \in  \eta P(\vz)+\vv^{-}+\eta H(\vz)+\nabla \Phi(\vz)- \nabla \Phi(\vz^{-})$ 
        \WHILE{$\eta \|F(\vz) - P(\vz)\|_* > \frac{\alpha}{2}\|\vz-\vz^{-}\|$}
        \STATE  Set $\eta \leftarrow \eta \beta$ and let $\vz$ be the solution of $0 \in  \eta P(\vz)+\vv^{-}+\eta H(\vz)+\nabla \Phi(\vz)- \nabla \Phi(\vz^{-})$ 
        \ENDWHILE
        \STATE \textbf{Return} $\eta$ and $\vz$

    \end{algorithmic}
\end{algorithm}

Under suitable conditions, the following lemma ensures that the line search scheme will terminate properly and find an admissible step size.

\begin{lemma}\label{lem:backtrack_terminate}
    Assume that there exist $\delta>0$ and $L>0$ such that $\|F(\vz)-P(\vz)\|_* \leq L \|\vz-\vz^{-}\|$ for any $\vz$ satisfying $\|\vz-\vz^{-}\|\leq \delta$. Then Algorithm~\ref{alg:ls} will always terminate in finite steps. 
\end{lemma}
\begin{proof}
    Our proof is adapted from \cite[Lemma 3.2]{Malitsky2020}. First, we show that $\vz(\eta;\vz^{-})$ lies in a compact set $\mathcal{C}_{\sigma}\subset \mathcal{Z}$ for any $\eta\in [0,\sigma]$. For simplicity, in the following we denote $\vz(\eta;\vz^{-})$ by $\vz$. By \eqref{eq:simplified_subproblem}, there exists $\vw\in H(\vz)$ such that $        \eta P(\vz) + \vv^{-}+\eta \vw+ \nabla\Phi(\vz)-\nabla\Phi(\vz^{-})=0$, 
    which implies that 
    \begin{equation}\label{eq:zero_inner_product}
        -\langle \vv^{-},\vz-\vz^{-} \rangle = \langle\nabla\Phi(\vz)-\nabla\Phi(\vz^{-}),\vz-\vz^{-} \rangle+\eta \langle P(\vz)+\vw,\vz-\vz^{-}\rangle. 
    \end{equation} 
    Since $\vz^{-}$ is feasible, we can find $\vw^{-}\in \mathbb{R}^d$ such that $\vw^{-}\in H(\vz^{-})$.
    By the strong convexity of $\Phi$ and the monotonicity of $P$ and $H$, it holds that $\langle\nabla\Phi(\vz)-\nabla\Phi(\vz^{-}), \vz-\vz^{-} \rangle \geq \|\vz-\vz^{-}\|^2$ and $\langle P(\vz)+\vw - (P(\vz^-)+\vw^-), \vz-\vz^- \rangle \geq 0$. Hence, from \eqref{eq:zero_inner_product} we have 
    $-\langle \vv^{-},\vz-\vz^{-} \rangle 
        \geq \|\vz-\vz^{-}\|^2+\eta \langle P(\vz^{-})+\vw^{-}, \vz-\vz^{-} \rangle$. 
    Together with Cauchy-Schwarz inequality and the fact that $\eta\leq \sigma$, this leads to
    \begin{equation}\label{eq:uniform_bounded}
        \|\vz-\vz^{-}\| \leq \|\vv^{-}\|_{*}+\eta\|P(\vz^{-})+\vw^{-}\|_{*}\leq \|\vv^{-}\|_{*}+\sigma\|P(\vz^{-})+\vw^{-}\|_{*}.
    \end{equation}
    Since the upper bound in \eqref{eq:uniform_bounded} is independent of $\eta$, this proves that $\vz(\eta;\vz^{-})$ belongs to the compact set $\mathcal{C}_{\sigma}:=\{\vz\in \mathbb{R}^d: \|\vz-\vz^{-}\|\leq \|\vv^{-}\|_{*}+\sigma\|P(\vz^{-})+\vw^{-}\|_{*}\}\cap \mathcal{Z}$ for all $0\leq \eta\leq \sigma$. 

    Now we prove the lemma by contradiction. Denote the trial step size after $i$ steps in Algorithm~\ref{alg:ls} by $\eta^{(i)}:=\sigma\beta^{i}$. If Algorithm~\ref{alg:ls} does not terminate in finite steps, then for any integer $i\geq 0$, the trial step size $\eta^{(i)}$ does not satisfy condition \eqref{eq:simplified_step_bound}, i.e., 
    \begin{equation}\label{eq:condition_for_all_i}
        \eta^{(i)} \|F(\vz(\eta^{(i)};\vz^{-}))-P(\vz(\eta^{(i)};\vz^{-}))\|_{*} > \frac{1}{2}\alpha\|\vz(\eta^{(i)};\vz^{-})-\vz^{-}\|, \quad \forall i \geq 0. 
    \end{equation}
    Since both $F$ and $P$ are continuous on $\mathcal{Z}$ and $\vz(\eta^{(i)};\vz^{-})$ remains in a compact set $\mathcal{C}_{\sigma}$ for all $i\geq 0$, we can upper bound $\|F(\vz(\eta^{(i)};\vz^{-}))-P(\vz(\eta^{(i)};\vz^{-}))\|_{*}$ by a constant independent of $i$. Thus, we obtain $\vz(\eta^{(i)};\vz^{-}) \rightarrow \vz^{-}$ by taking $i$ to infinity in \eqref{eq:condition_for_all_i}. Hence, for $i$ large enough, we have $\|\vz(\eta^{(i)};\vz^{-})-\vz^{-}\|\leq \delta$, which further implies that {$\|F(\vz(\eta^{(i)};\vz^{-}))-P(\vz(\eta^{(i)};\vz^{-}))\|_{*} \leq L\|\vz(\eta^{(i)};\vz^{-})-\vz^{-}\|$} by our assumption.
    Combining this with \eqref{eq:condition_for_all_i}, we conclude that $\eta^{(i)}L>\frac{1}{2}\alpha$ for any large enough $i$. Since $\eta^{(i)}\rightarrow 0$ as $i\rightarrow \infty$, this leads to a contradiction.    
\end{proof}
Finally, while Lemma~\ref{lem:backtrack_terminate} %
ensures that the number of steps taken in Algorithm~\ref{alg:ls} must be finite,  we still need an explicit upper bound on the total number of subsolver calls to characterize the complexity of our line search scheme. %
The following proposition partially addresses this issue by giving such an upper bound in terms of the returned step size $\eta$ and the initial trial step size $\sigma$.

\begin{proposition}\label{prop:ls_total}
    Let $\eta$ be the step size returned by Algorithm~\ref{alg:ls}. Then the procedure takes $\log_{\frac{1}{\beta}}\frac{\sigma}{\eta}+1$ calls to the optimistic subsolver in total. 
\end{proposition}
\begin{proof}
    The line search scheme iteratively tests the step size from the set $\{\sigma \beta^i: i\geq 0\}$ in a decreasing order, each of which corresponds to one call to the optimistic subsolver. Hence, if $N_{\mathrm{ls}}$ denotes the total number of subsolver calls in Algorithm~\ref{alg:ls}, the returned step size satisfies $\eta = \sigma \beta^{N_{\mathrm{ls}}-1}$, which implies $N_{\mathrm{ls}} = \log_{\frac{1}{\beta}}\frac{\sigma}{\eta}+1$.  
\end{proof}

 \begin{remark}
Note that the above result does not directly provide a complexity bound for our line search scheme, as it depends on the value of the returned step size $\eta$ and our choice of the initial trial step size $\sigma$. In the following sections, for each specific algorithm that we develop, we will establish lower bounds on $\eta$ and use them to derive an explicit upper bound on the number of calls to the optimistic subsolver during the whole process. 
\end{remark}

%% file: first_order.tex
\section{First-order generalized optimistic method}\label{sec:first_order}

In this section, we focus on the case where only first-order information of the smooth component $f$ of the objective function in \eqref{eq:minimax} is available. We also make the following assumption on $f$. 
\begin{assumption}\label{assum:first_order}
    The operator $F$ defined in \eqref{eq:def_of_F_H} is $L_1$-Lipschitz on $\mathcal{Z}$. Also, we have access to an oracle that returns $F(\vz)$ for any given $\vz$.
\end{assumption}
Under the above assumption, we have from Definition~\ref{def:smooth} that 
\begin{equation}\label{eq:Lipschitz}
    \|F(\vz)-F(\vz_{k})\|_{*}\leq L_1\|\vz-\vz_k\|, \quad \forall \vz \in \mathcal{Z}.
\end{equation}
In this case, a natural choice for the approximation function is $P(\vz;\mathcal{I}_k):=F(\vz_k)$, and accordingly at iteration $k$, our proposed optimistic method in Algorithm~\ref{alg:GOMD} aims to find $\vz\in \mathcal{Z}$ and $\eta>0$ such that 
\begin{gather}
    \vz = \argmin_{\vw\in \mathcal{Z}} \left\{\left\langle \eta F(\vz_k)+\hat{\eta}_k(F(\vz_k)-F(\vz_{k-1})),\vw-\vz_k\right\rangle+\eta h(\vw)+D_{\Phi}(\vw,\vz_k)\right\}, \label{eq:OMD_first}\\
    \eta \|F(\vz)-F(\vz_{k})\|_{*}\leq \frac{\alpha}{2}\|\vz-\vz_k\|. \label{first_order_step_size_condition}
\end{gather}
 In particular, the optimistic subsolver defined in Definition~\ref{def:subsolver} solves the subproblem in \eqref{eq:OMD_first}, which is a standard update in first-order methods.

As described in Algorithm~\ref{alg:first_optimistic}, we will study two different approaches for selecting the step size $\eta_k$ in the first-order optimistic method: %
(i) fixed step size scheme; (ii) backtracking line search scheme. As the name suggests, in the first approach a fixed step size $\eta_k \equiv \eta$ is properly selected to ensure that the condition in \eqref{first_order_step_size_condition} is always satisfied. In the second approach, we instead select the step size according to the line search scheme described in Algorithm~\ref{alg:ls}. One major advantage of the second approach is that it does not require any prior knowledge of the Lipschitz constant of $F$, while in the former we at least need to know an upper bound on the Lipschitz constant. 

\begin{algorithm}[!t]\small
    \caption{First-order optimistic method}\label{alg:first_optimistic}
    \begin{algorithmic}[1]
        \STATE \textbf{Input:} initial point $\vz_0\in \mathcal{Z}$, strong convexity parameter $\mu\geq 0$%
        \STATE \textbf{Initialize:} set $\vz_{-1}\leftarrow \vz_0$ and $P(\vz_0;\mathcal{I}_{-1})\leftarrow F(\vz_0)$
        \STATE \textbf{Option I (fixed step size scheme):} \tikzmark{top}
        \begin{ALC@g}
            \STATE Choose $M>0$
            \FOR{iteration $k=0,\ldots,N-1$}
            \STATE Compute $
                \vz_{k+1} = \argmin_{\vz\in \mathcal{Z}} \left\{\left\langle \frac{1}{M}F(\vz_k)+\frac{1}{M+\mu}(F(\vz_k)-F(\vz_{k-1})),\vz\right\rangle+\frac{1}{M}h(\vz)+D_{\Phi}(\vz,\vz_k)\right\} \tikzmark{right}
            $
            \ENDFOR
            \tikzmark{bottom} 
        \end{ALC@g}
        \STATE \textbf{Option II (line search scheme):}\tikzmark{top2}
        \begin{ALC@g}
            \STATE Choose initial trial step size $\sigma_0>0$, line search parameters $\alpha\in(0,1]$ and $\beta\in (0,1)$
            \FOR{iteration $k=0,\ldots,N-1$}
            \STATE Set $\hat{\eta}_k=\eta_{k-1}/(1+\eta_{k-1}\mu)$
            \STATE Select $\eta_{k}$ by Algorithm~\ref{alg:ls} with $\sigma=\sigma_k$, $\vz^-=\vz_k$, $P(\vz)=F(\vz_k)$, and $\vv^{-}=\hat{\eta}_k(F(\vz_k)-F(\vz_{k-1}))$ \tikzmark{right2}
            \STATE Compute $
                \vz_{k+1} = \argmin_{\vz\in \mathcal{Z}} \left\{\left\langle \eta_k F(\vz_k)+\hat{\eta}_k(F(\vz_k)-F(\vz_{k-1})),\vz\right\rangle+\eta_k h(\vz)+D_{\Phi}(\vz,\vz_k)\right\} 
            $
            \STATE Set $\sigma_{k+1} \leftarrow \eta_k/\beta$
            \ENDFOR
            \vspace{-3mm}
            \tikzmark{bottom2}
        \end{ALC@g}
    \end{algorithmic}
    \AddNote{top}{bottom}{right2}{}
    \AddNote{top2}{bottom2}{right2}{}
\end{algorithm}
\subsection{Fixed step size scheme}

In light of the inequality in \eqref{eq:Lipschitz}, if we select the step size such that $\eta_k\leq \frac{1}{2L_1}$, then the condition in \eqref{first_order_step_size_condition} required for the convergence of the optimistic method is always satisfied with $\alpha=1$. Hence, we can simply select a fixed step size of  $\eta_k \equiv 1/M$ with $M\geq 2L_1$ for all $k\geq 0$. In this case, the update of the first-order optimistic method becomes 
\begin{equation}\label{eq:OMD_first_fixed}
    \vz_{k+1} = \argmin_{\vz\in \mathcal{Z}} \left(\left\langle \frac{1}{M}F(\vz_k)+\hat{\eta}_k(F(\vz_k)-F(\vz_{k-1})),\vz-\vz_k\right\rangle+\frac{1}{M}h(\vz)+D_{\Phi}(\vz,\vz_k)\right),
\end{equation}
where the correction coefficient $\hat{\eta}_k$ will be chosen as in Proposition~\ref{prop:GOMD} based on the strong convexity parameter $\mu$. %
By viewing \eqref{eq:OMD_first_fixed} as an instance of our generalized optimistic method, the convergence guarantees immediately follow from Proposition~\ref{prop:GOMD}, which are summarized in the theorem below.

\begin{theorem}\label{thm:OMD_first}
    Assume that the operator $F$ is $L_1$-Lipschitz on $\mathcal{Z}$. Let $\{\vz_k\}_{k\geq 0}$ be the iterates generated by \eqref{eq:OMD_first_fixed}. Given $M\geq 2L_1$, the following statements hold:   
    \begin{enumerate}[label=(\alph*)]
    \vspace{-2mm}
        \item Under Assumption~\ref{assum:monotone}, set $\hat{\eta}_k = 1/M$. Then we have $D_{\Phi}(\vz^*,\vz_k)\leq 2D_{\Phi}(\vz^*,\vz_0)$ for any $k\geq 0$. Moreover, if we define $\bar{\vz}_N = (\bar{\vx}_N,\bar{\vy}_N)$ where $\bar{\vz}_N=\frac{1}{N}\sum_{k=0}^{N-1} \vz_{k+1}$, then for any $\vz = (\vx,\vy)\in \mathcal{X}\times \mathcal{Y}$ 
        \begin{equation}\label{eq:gap_bound_1st}
            \ell(\bar{\vx}_N,\vy)-\ell(\vx,\bar{\vy}_N) \leq \frac{MD_{\Phi}(\vz,\vz_0)}{N}, 
        \end{equation}
        \item Under Assumption~\ref{assum:strongly_monotone}, set $\hat{\eta}_k = 1/(M+\mu)$. Then we have 
        \begin{equation*}%
            D_{\Phi}(\vz^*,\vz_N)\leq 2D_{\Phi}(\vz^*,\vz_0) \left(\frac{M}{\mu+M}\right)^{N}.
        \end{equation*}
    \end{enumerate}
\end{theorem}
    
The result in Theorem~\ref{thm:OMD_first} shows that the proposed first-order generalized optimistic method converges at a sublinear rate of $\mathcal{O}(1/N)$ for convex-concave saddle point problems, and converges at a linear rate of $\mathcal{O}((M/(\mu +M))^N)$ for strongly-convex-strongly-concave saddle point problems. 
Asymptotically, these bounds are the same as those derived in \cite{Mokhtari2020a,Mokhtari2020,gidel2018a} for the OGDA method. However, these prior works are restricted to the unconstrained saddle point problems in the Euclidean setup.
In this regard, our results can be viewed as a natural extension of the OGDA method and its analysis to the more general composite saddle point problems equipped with an arbitrary Bregman distance.

The forward-reflected-backward splitting method in \cite{Malitsky2020} for monotone inclusion problems and the operator extrapolation method in \cite{kotsalis2022simple} for constrained variational inequalities share similar update rule and convergence analysis as the above first-order optimistic method. 
Moreover, all differ from the classical Popov's method in \eqref{eq:OMD_original} except in the unconstrained setting. %
{In particular,  
a {notable difference is that}  our method and those in \cite{Malitsky2020,kotsalis2022simple} require one projection onto the feasible set per iteration, while  \eqref{eq:OMD_original} requires two projections per iteration.}
On the other hand, our setting is more general than the mentioned works in the sense that 
{the authors in \cite{Malitsky2020} only considered the Euclidean norm (i.e., $\Phi=\frac{1}{2}\|\cdot\|_2^2$), while the authors in \cite{kotsalis2022simple} only considered smooth objectives (i.e., $h_1=h_2\equiv 0$).}

\subsection{Backtracking line search scheme}
While simple in concept, the fixed step size scheme requires
knowledge of the Lipschitz constant, which might be impractical. Also, choosing the same step size for all iterations could fail to exploit the local geometry and be overly conservative when it is possible to take larger steps. Next, we overcome these issues by choosing the step sizes adaptively with the line search scheme proposed in Algorithm~\ref{alg:ls}. 
Our proposed method is shown as Option II in Algorithm~\ref{alg:first_optimistic}. 
In addition, at the $k$-th iteration ($k\geq 1$), our line search procedure starts from $\sigma_{k}=\eta_{k-1}/\beta$, where $\eta_{k-1}$ is the step size chosen at the previous iteration. Similar initialization strategy is used in \cite{Malitsky2017,Malitsky2020}. 
Finally, note that Assumption~\ref{assum:first_order} ensures that the condition in Lemma~\ref{lem:backtrack_terminate} is satisfied, and hence the line search scheme is guaranteed to terminate in finite steps and hence the step sizes $\{\eta_k\}_{k\geq 0}$ are well defined. 

To obtain convergence rates and bound the number of subsolver calls required for the line search scheme, we need to bound the step size $\eta_k$ away from zero.
{To this end, define $ \mathcal{B} =\{k: \eta_k < \sigma_k\}$, i.e., the set of iteration indices where the line search scheme backtracks the step size.}
The following lemma serves this purpose and gives a lower bound on $\eta_k$ when $k \in \mathcal{B}$.
\begin{lemma}\label{lem:beta_optimal_point}
    Suppose that $F$ is $L_1$-Lipschitz on $\mathcal{Z}$. 
    If $k \in \mathcal{B}$, then $\eta_k>{\alpha\beta}/(2L_1)$. 
\end{lemma}

 \vspace{-2mm}
\begin{proof}
    To simplify the notation, we drop the subscript $k$ and denote $\vz_{k}$ by $\vz^{-}$. Since $k \in \mathcal{B}$, by Algorithm~\ref{alg:ls}, the condition in~\eqref{first_order_step_size_condition} is violated for $\eta' = \eta/\beta$ and $\vz'=\vz(\eta';\vz^-)$ (since the line search scheme would choose $\eta'$ instead of $\eta$ otherwise). 
    Hence,
    \begin{equation*}
        \frac{\alpha}{2}\|\vz'-\vz^{-}\| < \eta'\|F(\vz')-F(\vz^{-})\|_{*}
                                                                \leq \eta'L_1\|\vz'-\vz^{-}\|,
    \end{equation*}
    where we used the fact that $F$ is $L_1$-Lipschitz in the last inequality. This implies that $\eta'>\alpha/(2L_1)$, and we immediately obtain that  $\eta \geq \beta\eta'>{\alpha\beta}/(2L_1)$.  
\end{proof}

    \vspace{-1mm}
Combining all pieces above, we arrive at the following theorem by specializing  
Propositions~\ref{prop:GOMD} and~\ref{prop:ls_total}.

    \vspace{-1mm}

\begin{theorem}\label{thm:OMD_first_ls}
    Assume that the operator $F$ is $L_1$-Lipschitz on $\mathcal{Z}$. Let $\{\vz_k\}_{k\geq 0}$ be the iterates generated by \eqref{eq:OMD_first}, where the step sizes are determined by Algorithm~\ref{alg:ls} with parameters $\alpha\in (0,1]$, $\beta\in (0,1)$, and an initial trial step size $\sigma_0$. %
    Then the following statements hold:   
    \begin{enumerate}[label=(\alph*)]
    \vspace{-2mm}
        \item\label{item:first_ls_cc} Under Assumption~\ref{assum:monotone}, set $\hat{\eta}_k = \eta_{k-1}$. 
        Then we have $D_{\Phi}(\vz^*,\vz_k)\leq \frac{2}{2-\alpha} D_{\Phi}(\vz^*,\vz_0)$ for any $k\geq 0$. Moreover, for any $\vz = (\vx,\vy)\in \mathcal{X}\times \mathcal{Y}$,
        \begin{equation*}
            \ell(\bar{\vx}_N,\vy)-\ell(\vx,\bar{\vy}_N)\leq \frac{2L_1 D_{\Phi}(\vz,\vz_0)}{\alpha\beta N}+\frac{D_{\Phi}(\vz,\vz_0)}{(1-\beta)\sigma_0N^2}, 
        \end{equation*}
        where 
         $\bar{\vz}_N = (\bar{\vx}_N,\bar{\vy}_N)$ is given by $\bar{\vz}_N=\frac{1}{\sum_{k=0}^{N-1} \eta_k}\sum_{k=0}^{N-1} \eta_k\vz_{k+1}$.
        \item\label{item:first_ls_scsc} Under Assumption~\ref{assum:strongly_monotone}, set $\hat{\eta}_k = {\eta_{k-1}}/{(1+\mu\eta_{k-1})}$. Then we have
        \begin{equation*}
            D_{\Phi}(\vz^*,\vz_N)\leq \frac{2C}{2-\alpha}D_{\Phi}(\vz^*,\vz_0)  \left(1+\frac{\mu\alpha\beta}{2L_1}\right)^{-N},
        \end{equation*}
        where $C=\exp\left(\frac{\alpha\beta}{2(1-\beta)\sigma_0L_1}\frac{\alpha\beta\mu/(2L_1)}{1+\alpha\beta\mu/(2L_1)}\right)$ is a constant that only depends on $\alpha,\beta,\sigma_0,L_1$ and $\mu$. 
         \vspace{2mm}
        \item\label{item:first_ls_complexity} In both cases of (a) and (b) above, the total number of calls to the optimistic subsolver after $N$ iterations can be bounded by {$ \max\{N, 2N-1+\log_{\frac{1}{\beta}}\bigl(\frac{2\sigma_0 L_1}{\alpha \beta}\bigr)\}$.}
    \end{enumerate}    
\end{theorem}
\begin{proof}
    Recall that $\mathcal{B} =\{k: \eta_k < \sigma_k\}$. 
    Since $\{\eta_k\}_{k\geq 0}$ is determined by Algorithm~\ref{alg:ls}, we observe that $\eta_k=\sigma_k $ when $k \notin \mathcal{B}$, and $\eta_k>(\alpha\beta)/(2L_1)$ when $k\in \mathcal{B}$ by Lemma~\ref{lem:beta_optimal_point}. Moreover, in Algorithm~\ref{alg:2nd_optimistic}, we set $\sigma_k = \eta_{k-1}/\beta$. 
From these observations, we shall prove by induction that for all $k\geq 0$,  
\begin{equation}\label{eq:1st_step_lb}
    \eta_k \geq \min\left\{\frac{\sigma_0}{\beta^{k}}, \frac{\alpha\beta}{2L_1}\right\}.
\end{equation}
It is easy to verify that \eqref{eq:1st_step_lb} holds for $k=0$. 
Now assume that \eqref{eq:1st_step_lb} holds for all $0\leq k \leq l-1$. \blue{If $l\in \mathcal{B}$, we have $\eta_l\geq (\alpha\beta)/(2L_1)$ and hence \eqref{eq:1st_step_lb} holds. Otherwise, we have}
\begin{equation*}
    \eta_l = \sigma_l = \frac{\eta_{l-1}}{\beta} \geq \frac{1}{\beta}\min\left\{\frac{\sigma_0}{\beta^{l-1}}, \frac{\alpha\beta}{2L_1}\right\}=\min\left\{\frac{\sigma_0}{\beta^{l}}, \frac{\alpha}{2L_1}\right\} \blue{\geq \min\left\{\frac{\sigma_0}{\beta^{l}}, \frac{\alpha\beta}{2L_1}\right\}}.
\end{equation*}
In both cases, we can see that \eqref{eq:1st_step_lb} is satisfied for $k=l$, and hence the claim is proved by induction. 

\myalert{Proof of Part~\ref{item:first_ls_cc}.} From \eqref{eq:1st_step_lb}, we have
\begin{equation}\label{eq:1st_step_inverse}
    \frac{1}{\eta_k} \leq \max\left\{\frac{\beta^{k}}{\sigma_0}, \frac{2L_1}{\alpha\beta}\right\} \leq \frac{\beta^{k}}{\sigma_0}+\frac{2L_1}{\alpha\beta}.
\end{equation}
By summing \eqref{eq:1st_step_inverse} over $k=0,\ldots,N-1$ and using the fact that $\beta<1$, we obtain
\begin{equation}\label{eq:1st_step_inverse_sum}
    \sum_{k=0}^{N-1} \frac{1}{\eta_k} \leq \frac{1}{(1-\beta)\sigma_0}+ \frac{2L_1}{\alpha\beta}N.
\end{equation}
Now combining \eqref{eq:1st_step_inverse_sum} with Cauchy-Schwarz inequality leads to 
\(\sum_{k=0}^{N-1} \eta_k \geq N^2 \biggl({\sum_{k=0}^{N-1} \frac{1}{\eta_k}}\biggr)^{-1} \geq  \left(\frac{1}{(1-\beta)\sigma_0N^2}+ \frac{2L_1}{\alpha\beta N}\right)^{-1}\).
The rest follows from Proposition~\ref{prop:GOMD}.  

\myalert{Proof of Part~\ref{item:first_ls_scsc}.} Note that
$
    \log \Bigl(\prod_{k=0}^{N-1} (1+\eta_k\mu)\Bigr) = \sum_{k=0}^{N-1} \log(1+\eta_k\mu) = \sum_{k=0}^{N-1}\log\Bigl(1+\frac{1}{1/(\eta_k\mu)}\Bigr). 
$ Moreover, $\log(1+\frac{1}{t})$ is  convex and monotonically decreasing on $\mathbb{R}_{++}$. Hence, by Jensen's inequality and \eqref{eq:1st_step_inverse_sum} we have 
\begin{equation}\label{eq:1st_ls_log}
    \begin{aligned}
        \log \biggl(\prod_{k=0}^{N-1} (1+\eta_k\mu)\biggr) &\geq N \log\biggl(1+\mu\biggl(\frac{1}{N}\sum_{k=0}^{N-1}\frac{1}{\eta_k}\biggr)^{-1}\biggr) \\
        &\geq N \log \left(1+\frac{\alpha \beta \mu}{2L_1}\right)+N\log\left(1-\frac{c_1}{1+c_1}\frac{1}{1+c_2N}\right),
    \end{aligned}
\end{equation}
where $c_1:=\nicefrac{\alpha \beta \mu}{2L_1}$ and $c_2:=\nicefrac{2(1-\beta)\sigma_0L_1}{\alpha\beta}$. 
Moreover, by using the elementary inequality $\log(1+x) \geq x/(1+x)$ for all $x\geq -1$, we have 
\begin{equation}\label{eq:1st_ls_constant}
    N\log\Bigl(1-\frac{c_1}{1+c_1}\frac{1}{1+c_2N}\Bigr) \geq -\frac{c_1N}{1+(1+c_1)c_2N} \geq -\frac{c_1}{(1+c_1)c_2} = -\frac{\alpha\beta}{2(1-\beta)\sigma_0L_1}\frac{\alpha\beta\mu/(2L_1)}{1+\alpha\beta\mu/(2L_1)}.
\end{equation}
Combining \eqref{eq:1st_ls_log} and \eqref{eq:1st_ls_constant}, we arrive at 
\begin{equation*}
    \prod_{k=0}^{N-1} (1+\eta_k\mu) \geq \left(1+\frac{\alpha \beta \mu}{2L_1}\right)^N \exp\left(-\frac{\alpha\beta}{2(1-\beta)\sigma_0L_1}\frac{\alpha\beta\mu/(2L_1)}{1+\alpha\beta\mu/(2L_1)}\right).
\end{equation*}
The rest follows from Proposition~\ref{prop:GOMD}. 

\myalert{Proof of Part~\ref{item:first_ls_complexity}.} By Proposition~\ref{prop:ls_total}, the number of calls to the optimistic subsolver at the $k$-th iteration is given by $\log_{1/\beta}\frac{\sigma_k}{\eta_k}+1$. Also note that $\sigma_k=\eta_{k-1}/\beta$ for $k\geq 1$. Hence, the total number of calls after $N$ iterations can be bounded by 
\begin{align*}
    N_{\mathrm{total}} = \sum_{k=0}^{N-1} \left(\log_{\frac{1}{\beta}}\frac{\sigma_k}{\eta_k}+1\right) &= N + \log_{\frac{1}{\beta}}\frac{\sigma_0}{\eta_0} + \sum_{k=1}^{N-1} \log_{\frac{1}{\beta}}\frac{\eta_{k-1}}{\beta\eta_k} \\
    &= 2N-1 +  \log_{\frac{1}{\beta}}{\sigma_0} - \log_{\frac{1}{\beta}}{\eta_0} + \sum_{k=1}^{N-1} \left(\log_{\frac{1}{\beta}}{\eta_{k-1}} - \log_{\frac{1}{\beta}}{\eta_{k}}\right) \\
    &= 2N-1 + \log_{\frac{1}{\beta}} \frac{\sigma_0}{\eta_{N-1}}. 
\end{align*}
Since $\eta_{N-1}\geq \min\left\{\frac{\sigma_0}{\beta^{N-1}}, \frac{\alpha\beta}{2L_1}\right\}$, we can further upper bound $N_{\mathrm{total}} \leq 2N-1+\max\{-(N-1),\log_{\frac{1}{\beta}}\bigl(\frac{2\sigma_0 L_1}{\alpha \beta}\bigr)\} = \max\{N, 2N-1+\log_{\frac{1}{\beta}}\bigl(\frac{2\sigma_0 L_1}{\alpha \beta}\bigr)\}$. This completes the proof. 
\end{proof}

Asymptotically, Theorem~\ref{thm:OMD_first_ls} shows that the line search scheme achieves the same convergence rates as the fixed step size scheme where the step size in all iterations are chosen as $(\alpha\beta)/(2L_1)$, although it does not require any prior information of the Lipschitz constant $L_1$. Moreover, it shows the total number of calls to a subsolver of \eqref{eq:OMD_first} required by the line search scheme is $2N+\mathcal{O}(\log(\sigma_0L_1))$. Hence, the average number of calls per iteration can be bounded by a constant close to 2 for large~$N$.  

%% file: second_order.tex
\section{Second-order generalized optimistic method}\label{sec:second_order}

In this section, we instantiate Algorithm~\ref{alg:GOMD} with a second-order oracle to derive a novel \textit{second-order optimistic} method for solving the saddle point problem in \eqref{eq:minimax}. For technical reasons, throughout the section, we restrict ourselves to the case where the distance-generating function $\Phi$ is $L_{\Phi}$-smooth on $\mathcal{Z}$, i.e., $\nabla \Phi$ is $L_{\Phi}$-Lipschitz continuous. We also require the following assumption on the smooth component of the objective in \eqref{lem:saddle_point} denoted by $f$. 
\begin{assumption}\label{assum:second_order}
    The operator $F$ defined in \eqref{eq:def_of_F_H} is second-order $L_2$-smooth on $\mathcal{Z}$. Also, we have access to an oracle that returns $F(\vz)$ and $DF(\vz)$ for any given $\vz$. 
\end{assumption}
Under Assumption~\ref{assum:second_order}, it follows from \eqref{eq:residual_F} that
\begin{equation}\label{eq:2nd_smooth}
    \|F(\vz)-F(\vz_{k})-DF(\vz_k)(\vz-\vz_k)\|_{*}\leq \frac{L_2}{2}\|\vz-\vz_k\|^2, \quad \forall \vz \in \mathcal{Z},
\end{equation}
which suggests we can choose the approximation function as $P(\vz;\mathcal{I}_k):=T^{(1)}(\vz;\vz_k)=F(\vz_k)+DF(\vz_k)(\vz-\vz_k)$. Note that computing this approximation function requires access to the operator $DF(\cdot)$, which involves the second-order derivative of the function $f$. Therefore, we refer to the resulting algorithm as the second-order generalized optimistic method. According to the update rule in \eqref{eq:GOM_inclusion}, the update of the proposed second-order optimistic method can be written as
\begin{equation}\label{eq:OMD_2nd_ls}
    0 \in   \eta_k T^{(1)}(\vz_{k+1};\vz_k)
    +\hat{\eta}_k(F(\vz_k)-T^{(1)}(\vz_k;\vz_{k-1}))+\eta_k H(\vz_{k+1})+\nabla \Phi(\vz_{k+1})-\nabla \Phi(\vz_k),
\end{equation}
where the condition in \eqref{eq:stepsize_upper_bound} on $\eta_k$ becomes 
\begin{equation}\label{eq:stepsize_2nd}
    \eta_k \|F(\vz_{k+1})-T^{(1)}(\vz_{k+1};\vz_{k}))\|_{*}\leq \frac{\alpha}{2}\|\vz_{k+1}-\vz_k\|.
\end{equation}
\blue{The inclusion problem in \eqref{eq:OMD_2nd_ls} is maximal monotone and we defer the proof to Appendix~\ref{appen:maximal}.}

To begin with, we explain how the second-order information could help accelerate convergence. Intuitively, in light of \eqref{eq:stepsize_2nd}, choosing a more accurate approximation $T^{(1)}(\cdot;\vz_k)$ allows us to take a larger step size $\eta_k$ than the first-order method in \eqref{eq:OMD_first}. Moreover, as Proposition~\ref{prop:GOMD} shows, the speed of convergence 
depends on the choice of step sizes $\{\eta_k\}_{k\geq 0}$, and larger step sizes lead to faster convergence.  
More precisely, by~\eqref{eq:2nd_smooth}, the left-hand side of \eqref{eq:stepsize_2nd} can be bounded above by $\frac{1}{2}\eta_k L_2\|\vz_{k+1}-\vz_k\|^2$, suggesting that we can pick our step size $\eta_k$ as large as 
\begin{equation}\label{eq:approx_step_2nd}
    \eta_k  \approx \frac{\alpha}{L_2\|\vz_{k+1}-\vz_k\|}.
\end{equation}
Hence, as $\{\vz_k\}_{k\geq 0}$ approach the optimal solution and the displacement $\|\vz_{k+1}-\vz_k\|$ becomes smaller, we can afford to select a larger step size and accelerate convergence. 

The above argument for faster convergence relies crucially on the premise that the step size $\eta_k$ is adaptive to the displacement $\|\vz_{k+1}-\vz_k\|$, i.e., at least on the order of $\frac{1}{\|\vz_{k+1}-\vz_k\|}$. As mentioned in Section~\ref{sec:GOMD}, a major challenge here is the interdependence between $\eta_k$ and $\vz_{k+1}$: the iterate $\vz_{k+1}$ depends on $\eta_k$ via \eqref{eq:OMD_2nd_ls}, while the step size $\eta_k$ depends on $\vz_{k+1}$ as illustrated in \eqref{eq:stepsize_2nd} and \eqref{eq:approx_step_2nd}. Hence, we need to simultaneously select $\eta_k$ and $\vz_{k+1}$ to satisfy the above requirements. This goal can be achieved by using the line search scheme discussed in Section~\ref{subsec:line_search}. 
 In addition, similar to the line search scheme for the first-order optimistic method, 
 we use a warm-start strategy and choose the initial trial step size $\sigma_{k+1}$ based on the previously accepted step size $\eta_k$. Specifically, we set $\sigma_{k+1} = \eta_k/\beta$ in the convex-concave setting and set $\sigma_{k+1}=\eta_{k}\sqrt{1+\eta_{k}\mu}/\beta$ in the strongly-convex-strongly-concave setting.   
As shown in Section~\ref{subsec:complexity_LS}, our choice of $\eta_{k+1}$ is crucial for ensuring that the average number of line search steps per iteration is bounded by a constant. The second-order optimistic method is formally described in Algorithm~\ref{alg:2nd_optimistic}.

Two remarks on the line search scheme in Algorithm~\ref{alg:2nd_optimistic} follow. 
First, by using the results in Lemma~\ref{lem:backtrack_terminate}, we can guarantee that the line search scheme always terminates in finite steps, and hence the step sizes $\{\eta_k\}_{k\geq 0}$ are well-defined. To see why Lemma~\ref{lem:backtrack_terminate} applies, note that for any $\|\vz-\vz_k\|\leq \delta$, by \eqref{eq:2nd_smooth} we have $   \|F(\vz)-T^{(1)}(\vz;\vz_k)\|_* \leq (\nicefrac{L_2}{2})\|\vz-\vz_k\|^2 \leq (\nicefrac{L_2\delta}{2})\|\vz-\vz_k\|$. Hence, the condition in Lemma~\ref{lem:backtrack_terminate} is indeed satisfied.
Second, at the $k$-th iteration,  the optimistic subsolver in this case is required to solve a subproblem of $\vz$ in the form of 
\begin{equation}\label{eq:2nd_order_subproblem}
    0 \in \eta T^{(1)}(\vz;\vz_{k})+\vv_k+\eta H(\vz) + \nabla \Phi(\vz)-\nabla \Phi(\vz_k)
\end{equation}
for given $\vv_k=\hat{\eta}_k(F(\vz_k)-T^{(1)}(\vz_k;\vz_{k-1}))$ and $\eta>0$ (cf. Definition~\ref{def:subsolver}). 
Note that $T^{(1)}(\vz;\vz_k)$ is an affine operator, and hence this subproblem can be solved with respect to $\vz$ efficiently by standard solvers in some special cases. For instance, in the Euclidean setup where $\Phi(\vz)=\frac{1}{2}\|\vz\|^2_2$, the subproblem \eqref{eq:2nd_order_subproblem} is equivalent to an affine variational inequality \cite{Facchinei2003} for constrained saddle point problems (where $h_1\equiv 0$ on $\mathcal{X}$ and $h_2\equiv 0$ on $\mathcal{Y}$), and further reduces to a system of linear equations for unconstrained saddle point problems (where $\mathcal{X}=\mathbb{R}^m$ and $\mathcal{Y}=\mathbb{R}^n$). 
\begin{algorithm}[!t]\small
    \caption{Second-order optimistic method}\label{alg:2nd_optimistic}
    \begin{algorithmic}[1]
        \STATE \textbf{Input:} initial point $\vz_0\in \mathcal{Z}$, initial trial step size $\sigma_0$, strong convexity parameter $\mu\geq 0$, line search parameters $\alpha,\beta \in(0,1)$  
        \STATE \textbf{Initialize:} set $\vz_{-1}\leftarrow \vz_0$ and $P(\vz_0;\mathcal{I}_{-1})\leftarrow F(\vz_0)$
        \FOR{iteration $k=0,\ldots,N-1$}
            \STATE Set $\hat{\eta}_k=\eta_{k-1}/(1+\eta_{k-1}\mu)$
            \STATE Select $\eta_{k}$ by Algorithm~\ref{alg:ls} with $\sigma=\sigma_k$, $\vz^-=\vz_k$, $P(\vz)=T^{(1)}(\vz;\vz_k)$, $\vv^{-}=\hat{\eta}_k(F(\vz_k)-T^{(1)}(\vz_k;\vz_{k-1}))$
            \STATE Compute $\vz_{k+1}$ by solving 
                $
                0\in \eta_k T^{(1)}(\vz;\vz_k)+\hat{\eta}_k(F(\vz_k)-T^{(1)}(\vz_k;\vz_{k-1}))+\eta_k H(\vz) + \nabla \Phi(\vz)-\nabla \Phi(\vz_k)
            $
                \STATE Set $\sigma_{k+1} \leftarrow \eta_k\sqrt{1+\eta_k\mu}/\beta$
        \ENDFOR
    \end{algorithmic}
\end{algorithm}

\subsection{Intermediate results}

To establish the convergence properties of the proposed second-order optimistic method, we first state a few intermediate results that will be used in the following sections.

Throughout the rest of the analysis, we define $ \mathcal{B} =\{k: \eta_k < \sigma_k\}$, i.e., the set of iteration indices where the line search scheme backtracks the step size. To begin with, in the following lemma we establish a lower bound on $\eta_k$ when $k \in \mathcal{B}$. 
\begin{lemma}\label{lem:beta_optimal_point_2nd}
    Let $\{\eta_k\}_{k \geq 0}$ be the step sizes in \eqref{eq:OMD_2nd_ls} generated by Algorithm~\ref{alg:2nd_optimistic}. Suppose that $F$ is second-order $L_2$-smooth and $\Phi$ is $L_{\Phi}$-smooth on $\mathcal{Z}$. Further define $\vv_k:=\hat{\eta}_k(F(\vz_k)-T^{(1)}(\vz_k;\vz_{k-1}))$. If $k\in \mathcal{B}$,
    then we have 
    \begin{equation*}
        \eta_k \geq \frac{\alpha\beta^2/L_2}{L_{\Phi}^{3/2}\|\vz_{k+1}-\vz_k\|+(\beta+L_{\Phi}^{1/2})\|\vv_k\|_*}.
    \end{equation*}
\end{lemma}
\begin{proof}
    See Appendix~\ref{appen:beta_optimal_2nd}. 
\end{proof}

By combining Lemma~\ref{lem:beta_optimal_point_2nd} and Lemma~\ref{lem:sum_of_square}, we further present the following result.

\begin{lemma}\label{lem:bound_on_sum_square_inverse}
    We have 
    $\sum_{k\in \mathcal{B}} \frac{\prod_{l=0}^{k-1} (1+\eta_l\mu)}{\eta_k^2}  \leq \gamma_2^2 L_2^2 D_{\Phi}(\vz^*,\vz_0)$, 
    where $\gamma_2$ is an absolute constant defined as
    \begin{equation}\label{eq:def_gamma}
        \gamma_2= \sqrt{\frac{2}{1-\alpha}}\left(\frac{L_{\Phi}^{3/2}}{\alpha \beta^2} + \frac{\beta+L_{\Phi}^{1/2}}{2\beta^2}\right).
    \end{equation} 
\end{lemma}

\begin{proof}
   We first apply Lemma~\ref{lem:beta_optimal_point_2nd} for $k\in \mathcal{B}$. 
    Note that by the choice of $\hat{\eta}_k$ and \eqref{eq:stepsize_2nd}, we can upper bound 
    \(\|\vv_k\|_* = \frac{\eta_{k-1}}{1+\eta_{k-1}\mu}\|F(\vz_k)-T^{(1)}(\vz_k;\vz_{k-1})\|_* \leq \frac{\alpha \|\vz_k-\vz_{k-1}\|}{2(1+\eta_{k-1}\mu)}.\)
    Together with Lemma~\ref{lem:beta_optimal_point_2nd}, this leads to 
    \begin{align*}
        \frac{1}{\eta_k} \leq \frac{L_2}{\alpha\beta^2}\left(L_{\Phi}^{3/2}\|\vz_{k+1}-\vz_k\|+(\beta+L_{\Phi}^{1/2})\|\vv_k\|_*\right) \leq c_1L_2\|\vz_{k+1}-\vz_k\| + c_2L_2\frac{\|\vz_{k}-\vz_{k-1}\|}{1+\eta_{k-1}\mu},
    \end{align*}
    where we define $c_1:=L_{\Phi}^{3/2}/(\alpha\beta^2)$ and $c_2:=(\beta+L_{\Phi}^{1/2})/(2\beta^2)$. 
    Furthermore, we square both sides of the above inequality and use Young's inequality to get 
    \begin{equation}\label{eq:after_Young}
        \frac{1}{\eta_k^2} \leq (c_1^2+c_1c_2) L_2^2\|\vz_{k+1}-\vz_k\|^2+{(c_2^2+c_1c_2)L_2^2} \frac{\|\vz_{k}-\vz_{k-1}\|^2}{(1+\eta_{k-1}\mu)^2}.  
\end{equation}
Multiplying both sides of \eqref{eq:after_Young} by $\prod_{l=0}^{k-1} (1+\eta_l\mu)$, 
    we have $
    \frac{1}{\eta_k^2}{\prod_{l=0}^{k-1} (1+\eta_l\mu)} 
        \leq (c_1^2+c_1c_2) L_2^2 \|\vz_{k+1}-\vz_k\|^2\prod_{l=0}^{k-1} (1+\eta_l\mu)
        +(c_2^2+c_1c_2) L_2^2\|\vz_{k}-\vz_{k-1}\|^2\prod_{l=0}^{k-2} (1+\eta_l\mu)$. 
    By summing both sides of the inequality above over 
    $k \in \mathcal{B}$ and applying Lemma~\ref{lem:sum_of_square}, we get the desired result in Lemma~\ref{lem:bound_on_sum_square_inverse}.  
\end{proof}

Finally, we build on Lemma~\ref{lem:bound_on_sum_square_inverse} to control step sizes $\eta_k$ for all $k \geq 0$.  
\begin{lemma}\label{lem:induction_bound}
    For any $k \geq 0$, we have 
    \begin{equation}\label{eq:induction_bound}
        \frac{1}{\eta_k^2} \prod_{l=0}^{k-1} (1+\eta_l\mu) \leq \max\left\{\gamma_2^2 L_2^2 D_{\Phi}(\vz^*,\vz_0), \frac{\beta^{2k}}{\sigma^2_0} \right\}. 
    \end{equation}
\end{lemma}
\begin{proof}
    We shall prove Lemma~\ref{lem:induction_bound} by induction. For the base case where $k=0$, we consider two subcases. If $0 \notin \mathcal{B}$, then we have $\eta_0 = \sigma_0$, which directly implies \eqref{eq:induction_bound}. Otherwise, since $0\in \mathcal{B}$, we can use Lemma~\ref{lem:bound_on_sum_square_inverse} to obtain that $\frac{1}{\eta_0^2} \leq \sum_{k\in \mathcal{B}} \left(\frac{1}{\eta_k^2} \prod_{l=0}^{k-1} (1+\eta_l\mu) \right) \leq \gamma_2^2 L_2^2 D_{\Phi}(\vz^*,\vz_0)$. This proves \eqref{eq:induction_bound} for $k=0$. 

    Now assume that \eqref{eq:induction_bound} holds for $k = s$ where $s\geq 0$. If $s+1\in \mathcal{B}$,  then similarly we have $\frac{1}{\eta_{s+1}^2} \prod_{l=0}^{s} (1+\eta_l\mu) \leq \sum_{k\in \mathcal{B}} \left(\frac{1}{\eta_k^2} \prod_{l=0}^{k-1} (1+\eta_l\mu) \right) \leq \gamma_2^2 L_2^2 D_{\Phi}(\vz^*,\vz_0)$ by Lemma~\ref{lem:bound_on_sum_square_inverse}, which implies~\eqref{eq:induction_bound}. Otherwise, if $s+1\notin \mathcal{B}$, then by definition we have $\eta_{s+1} = \sigma_{s+1} = \eta_{s} \sqrt{1+2\eta_{s}\mu}/\beta$, and furthermore
\begin{equation*}
    \frac{1}{\eta_{s+1}^2} \prod_{l=0}^{s} (1+\eta_l\mu) = \frac{\beta^2}{\eta^2_{s}}\prod_{l=0}^{s-1} (1+\eta_l\mu) \leq \beta^2\max\left\{\gamma_2^2 L_2^2 D_{\Phi}(\vz^*,\vz_0), \frac{\beta^{2s}}{\sigma^2_0} \right\}, 
\end{equation*}
where we used the induction hypothesis in the last inequality. 
In both cases, we prove that \eqref{eq:induction_bound} is satisfied for $k=s+1$, and hence the claim follows by induction.
\end{proof}

\subsection{Convergence analysis: convex-concave case}
Now we turn to the convergence analysis of Algorithm~\ref{alg:2nd_optimistic} and start with the case where the smooth component $f$ of the objective in \eqref{eq:minimax} is merely convex-concave. 

\begin{lemma}\label{lem:bound_on_sum_square_inverse_convex}
    Let $\{\eta_k\}_{k \geq 0}$ be the step sizes in \eqref{eq:OMD_2nd_ls} generated by Algorithm~\ref{alg:2nd_optimistic}. Then, 
    we have %
    \(
        {\sum_{k=0}^{N-1} \frac{1}{\eta_k} \leq \frac{1}{1-\beta } \left(\frac{1}{\sigma_0}+\gamma_2 L_2 \sqrt{D_{\Phi}(\vz^*,\vz_0)} \sqrt{N} \right)}
    \), where  $\gamma_2$  is defined in \eqref{eq:def_gamma}. 
\end{lemma}
\begin{proof}
Recall that $ \mathcal{B} =\{k: \eta_k < \sigma_k\}$, i.e., the set of iteration indices where the line search backtracks the step size.  %
Moreover, in Algorithm~\ref{alg:ls}, we have  $\eta_k=\sigma_k$ when $k \notin \mathcal{B}$ and $\sigma_k = \eta_{k-1}/\beta$ for $k \geq 1$. Hence, we can obtain that  
\begin{equation*}
    \sum_{k=0}^{N-1} \frac{1}{\eta_k} =  \sum_{k \in \mathcal{B}} \frac{1}{\eta_k} +  \sum_{k\notin \mathcal{B}} \frac{1}{\eta_k} \leq \frac{1}{\sigma_0} + \sum_{k \in \mathcal{B}} \frac{1}{\eta_k} +  \sum_{k \geq 1, k\notin \mathcal{B}} \frac{\beta}{\eta_{k-1}} \leq \frac{1}{\sigma_0} + \sum_{k \in \mathcal{B}} \frac{1}{\eta_k} +  \sum_{k=0}^{N-1} \frac{\beta}{\eta_{k}}. 
\end{equation*}
Combining the left-hand side with the last term in the right-hand side, we obtain 
$\sum_{k=0}^{N-1} \frac{1}{\eta_k} \leq \frac{1}{1-\beta}\frac{1}{\sigma_0} + \frac{1}{1-\beta} \sum_{k \in \mathcal{B}} \frac{1}{\eta_k} \leq  \frac{1}{1-\beta}\frac{1}{\sigma_0} + \frac{1}{1-\beta} \sqrt{N \sum_{k \in \mathcal{B}} \frac{1}{\eta^2_k}}$, 
where we use Cauchy-Schwarz inequality in the last inequality. Finally, Lemma~\ref{lem:bound_on_sum_square_inverse_convex} follows from the above inequality and Lemma~\ref{lem:bound_on_sum_square_inverse}. 
\end{proof}

By instantiating Proposition~\ref{prop:GOMD}, we obtain the following convergence result.

\begin{theorem}\label{thm:OMD_2nd_ls_convex}
    Suppose Assumptions~\ref{assum:monotone} and~\ref{assum:second_order} hold and the distance-generating function $\Phi$ is $L_{\Phi}$-smooth.
    Let $\{\vz_k\}_{k\geq 0}$ be the iterates generated by Algorithm~\ref{alg:2nd_optimistic}, where 
    $\hat{\eta}_k = \eta_{k-1}$, $\sigma_0>0$,  $\alpha\in (0,1)$, and $\beta\in (0,1)$.  
    Then we have  
    $D_{\Phi}(\vz^*,\vz_k)\leq \frac{2}{2-\alpha} D_{\Phi}(\vz^*,\vz_0)$ for $k\geq 0$.
    Moreover, for any $\vz = (\vx,\vy)\in \mathcal{X}\times \mathcal{Y}$,
        \begin{equation}\label{eq:gap_bound_2nd}
            \ell(\bar{\vx}_N,\vy)-\ell(\vx,\bar{\vy}_N) \leq
            \frac{\gamma_2 L_2\sqrt{D_{\Phi}(\vz^*,\vz_0)} D_{\Phi}(\vz,\vz_0) }{(1-\beta) N^{\frac{3}{2}}}+ \frac{D_{\Phi}(\vz,\vz_0)}{(1-\beta)\sigma_0 N^2},
        \end{equation}
         where $\bar{\vz}_N = (\bar{\vx}_N,\bar{\vy}_N)$ is given by  $\bar{\vz}_N=\frac{1}{\sum_{k=0}^{N-1} \eta_k}\sum_{k=0}^{N-1} \eta_k\vz_{k+1}$ and $\gamma_2$ is an absolute constant defined in \eqref{eq:def_gamma}.
\end{theorem}
\begin{proof}
     By Cauchy-Schwarz inequality, we have  
$\sum_{k=0}^{N-1} \eta_k
   \sum_{k=0}^{N-1} \frac{1}{\eta_k} 
   \geq N^2$.
Together with Lemma~\ref{lem:bound_on_sum_square_inverse_convex}, this implies that 
\(\left(\sum_{k=0}^{N-1} \eta_k\right)^{-1} \! \leq \frac{1}{N^2} \sum_{k=0}^{N-1} \frac{1}{\eta_k}\leq  \frac{1}{(1-\beta) \sigma_0 N^2 } + \frac{\gamma_2 L_2 \sqrt{D_{\Phi}(\vz^*,\vz_0)}}{(1-\beta)N^{\frac{3}{2}}}.\) 
The rest follows from Proposition~\ref{prop:GOMD}. 
\end{proof}

Theorem~\ref{thm:OMD_2nd_ls_convex} shows that the second-order optimistic method converges at a rate of $\bigO(N^{-3/2})$ in terms of the primal-dual gap, which is faster than the rate of $\bigO(N^{-1})$ for first-order methods. As a corollary, 
to obtain a solution with a primal-dual gap of $\epsilon$, 
the proposed second-order optimistic method requires at most $\mathcal{O}(1/\epsilon^{2/3})$ iterations. Note that the total number of subsolver calls in the line search procedure after $N$ iterations can also be explicitly controlled, as we discuss later in Theorem~\ref{thm:steps_upper_2nd}.

\noindent\textbf{Comparison with \cite{Monteiro2012} and \cite{Bullins2022higher}}.
Similar iteration complexity bounds have been reported in \cite{Monteiro2012,Bullins2022higher} for extragradient-type second-order methods. 
In comparison, our method is based on a different algorithmic framework.
Moreover, \cite{Monteiro2012} only considers the Euclidean setup, while \cite{Bullins2022higher} only provides implementation details for the special case of unconstrained problems under stronger assumptions. 
Finally, as we shall discuss in Section~\ref{subsec:complexity_LS}, our line search scheme requires fewer calls to the subproblem solver in total compared to the approaches in \cite{Monteiro2012,Bullins2022higher}.

\subsection{Convergence analysis: strongly-convex-strongly-concave case}
\label{subsec:2nd_sc_sc}
We proceed to the case where the smooth component of the objective in \eqref{eq:minimax} is $\mu$-strongly-convex-strongly-concave.
We first define a positive decreasing sequence $\{\zeta_k\}_{k\geq 0}$ by 
    \begin{equation}\label{eq:def_zeta}
        \zeta_k = 
        \begin{cases}
            1, & \text{if } k=0;\\
            \prod_{l=0}^{k-1} (1+\eta_l\mu)^{-1}, & \text{if } k\geq 1.
        \end{cases}
\end{equation}
According to the result in Part (b) of Proposition \ref{prop:GOMD}, the Bregman distance to the optimal solution $ {D_{\Phi}(\vz^*,\vz_k)}$ can be bounded above by 
\begin{equation}\label{eq:bound_with_zeta}
    {D_{\Phi}(\vz^*,\vz_k)} \leq \frac{2{D_{\Phi}(\vz^*,\vz_0)}}{2-\alpha}\ \zeta_k,
\end{equation}
where $\alpha\in(0,1)$ is a line search parameter. As a result, 
${D_{\Phi}(\vz^*,\vz_k)}\leq \epsilon$ when we have $\zeta_k \leq  \frac{(2-\alpha)\epsilon}{2{D_{\Phi}(\vz^*,\vz_0)}}$. 
Hence, in the following we will characterize the convergence behavior of the sequence $\{\zeta_k\}_{k\geq 0}$, which immediately implies that of ${D_{\Phi}(\vz^*,\vz_k)}$. To begin with, we present the following lemma, which is a direct corollary of Lemma~\ref{lem:bound_on_sum_square_inverse}.

\begin{lemma}\label{lem:bound_on_sum_square_inverse_strongly_convex}
    Let $\{\eta_k\}_{k \geq 0}$ be the step sizes in \eqref{eq:OMD_2nd_ls} generated by Algorithm~\ref{alg:2nd_optimistic}. 
    Then, we have 
    $\sum_{k=0}^{N-1} \frac{1}{\zeta_k \eta_k^2}  \leq \frac{1}{(1-\beta^2) \sigma_0^2} + \frac{\gamma_2^2 L_2^2}{1-\beta^2} D_{\Phi}(\vz^*,\vz_0)$. 
\end{lemma}
\begin{proof}
    Recall $ \mathcal{B} =\{k: \eta_k < \sigma_k\}$, i.e., the set of iteration indices where the line search backtracks. 
    Since $\{\eta_k\}_{k\geq 0}$ is determined by Algorithm~\ref{alg:ls}, we have $\eta_k=\sigma_k$ when $k \notin \mathcal{B}$, where $\sigma_k = \eta_{k-1}\sqrt{1+\eta_{k-1}\mu}/\beta$ for $k \geq 1$. Hence, we can bound 
    \begin{equation*}
        \sum_{k\notin \mathcal{B}} \frac{1}{\zeta_k \eta_k^2} =\sum_{k\notin \mathcal{B}} \frac{1}{\zeta_k \sigma_k^2} \leq \frac{1}{\sigma_0^2} + \sum_{k \geq 1, k\notin \mathcal{B}} \frac{\beta^2}{\zeta_k \eta^2_{k-1}(1+\eta_{k-1}\mu)} =  \frac{1}{\sigma_0^2} + \sum_{k \geq 1, k\notin \mathcal{B}} \frac{\beta^2}{\zeta_{k-1} \eta^2_{k-1}},
    \end{equation*}
    where we used $\zeta_{k-1} = \zeta_k(1+\eta_{k-1}\mu)$ in the last equality. This further leads to  
    \begin{align*}
    \sum_{k=0}^{N-1} \frac{1}{\zeta_k \eta^2_k} =  \sum_{k \in \mathcal{B}} \frac{1}{\zeta_k \eta_k^2} +  \sum_{k\notin \mathcal{B}} \frac{1}{\zeta_k \eta_k^2}
    &\leq \frac{1}{\sigma_0^2} + \sum_{k \in \mathcal{B}} \frac{1}{\zeta_k \eta^2_k} +  \sum_{k \geq 1, k\notin \mathcal{B}} \frac{\beta^2}{\zeta_{k-1} \eta^2_{k-1}} \\
    &\leq \frac{1}{\sigma_0^2} + \sum_{k \in \mathcal{B}} \frac{1}{\zeta_k \eta^2_k} +  \sum_{k=0}^{N-1} \frac{\beta^2}{\zeta_{k} \eta^2_{k}}. 
\end{align*}
Combining the left-hand side with the last term in the right-hand side, we obtain 
$\sum_{k=0}^{N-1} \frac{1}{\zeta_k \eta_k^2} \leq \frac{1}{1-\beta^2}\frac{1}{\sigma_0^2} + \frac{1}{1-\beta^2} \sum_{k \in \mathcal{B}} \frac{1}{\zeta_k\eta_k^2}$. 
    By applying Lemma~\ref{lem:bound_on_sum_square_inverse}, we get the desired result in Lemma~\ref{lem:bound_on_sum_square_inverse_strongly_convex}. 
\end{proof}
By using Lemma~\ref{lem:bound_on_sum_square_inverse_strongly_convex}, we characterize the dynamic of the sequence $\{\zeta_k\}_{k\geq 0}$ in the following lemma. 
\begin{lemma}\label{lem:multiplying_factor}
    Let $\{\eta_k\}_{k \geq 0}$ be the step sizes in \eqref{eq:OMD_2nd_ls} generated by Algorithm~\ref{alg:2nd_optimistic} and $\{\zeta_k\}_{k\geq 0}$ be defined in \eqref{eq:def_zeta}. Further, recall the definition of  $\gamma_2$ in \eqref{eq:def_gamma}. 
    Then for any $0 \leq k_1<k_2\leq N$, we have 
    \blue{$\frac{1}{\zeta_{k_2}} \geq \frac{1}{\zeta_{k_1}}+C_2^{-\frac{1}{2}}\left(\sum_{k=k_1}^{k_2-1} \frac{1}{\zeta_k}\right)^{\frac{3}{2}}$},
    where ${\kappa}_2(\vz_0)$ is defined in \eqref{eq:condition_number} and 
    \begin{equation}\label{eq:def_C2}
        C_2 := \frac{1}{1-\beta^2 } \left(\frac{1}{\sigma_0^2 \mu^2}+\gamma_2^2 \kappa_2^2(\vz_0)\right).
    \end{equation}
    \end{lemma}
    \begin{proof}
    By the definition of $\zeta_k$ in \eqref{eq:def_zeta}, we have $\eta_k=(\zeta_{k}/\zeta_{k+1}-1)/\mu$. 
    We apply Lemma~\ref{lem:bound_on_sum_square_inverse_strongly_convex}
    and rewrite the inequality %
    in terms of $\{\zeta_k\}_{k=0}^{N-1}$ as
    \begin{align}
        \sum_{k=0}^{N-1}  \left(\frac{\mu\zeta_{k+1}}{\zeta_k-\zeta_{k+1}}\right)^2\!\!\frac{1}{\zeta_k} &\leq \frac{1}{(1-\beta^2) \sigma_0^2} + \frac{\gamma_2^2 L_2^2}{1-\beta^2} D_{\Phi}(\vz^*,\vz_0) \nonumber\\
        \; \Leftrightarrow \;    \sum_{k=0}^{N-1}  \left(\frac{1}{\zeta^{-1}_{k+1}-\zeta^{-1}_{k}}\right)^2\!\!\frac{1}{\zeta_k^3} &\leq \frac{1}{(1-\beta^2) \sigma_0^2\mu^2} +\frac{\gamma_2^2 \kappa_2^2(\vz_0)}{1-\beta^2} = C_2, \label{eq:bound_sum_square_zeta}
    \end{align}
    where in \eqref{eq:bound_sum_square_zeta} we used the definition of ${\kappa}_2(\vz_0)$ and $C_2$. 
    Since each summand in \eqref{eq:bound_sum_square_zeta} is nonnegative, it follows that 
    $\sum_{k=k_1}^{k_2-1}  \left(\frac{1}{1/\zeta_{k+1}-1/\zeta_{k}}\right)^2\frac{1}{\zeta_k^3} \leq C_2$
    for any $k_2>k_1\geq 0$. 
    Furthermore, by applying H{\"o}lder's inequality we get
    \begin{gather*}
        \left[\sum_{k=k_1}^{k_2-1} \left(\frac{1}{\zeta_{k+1}}-\frac{1}{\zeta_k}\right)\right]^{\frac{2}{3}}
        \left[\sum_{k=k_1}^{k_2-1}  \left(\frac{1}{1/\zeta_{k+1}-1/\zeta_{k}}\right)^2\frac{1}{\zeta_k^3}\right]^{\frac{1}{3}}
        \geq \sum_{k=k_1}^{k_2-1} \frac{1}{\zeta_k}, %
    \end{gather*}
    and Lemma~\ref{lem:multiplying_factor} follows from the above two inequalities.  
\end{proof}

 By leveraging Lemma~\ref{lem:multiplying_factor}, we establish a global complexity bound for our proposed second-order method. Specifically, we show that the sequence $\{\zeta_k\}_{k\geq 0}$ halves its value after a geometrically decreasing number of iterations. Our argument is inspired by the restarting strategy in \cite{Nesterov2008,Nesterov2006a}, but we do not need to restart our algorithm here.

\begin{theorem}[Global convergence]\label{thm:2nd_strongly_convex_global}
   Suppose that Assumptions~\ref{assum:strongly_monotone} and~\ref{assum:second_order} hold and the distance-generating function $\Phi$ is $L_{\Phi}$-smooth. Let $\{\vz_k\}_{k\geq 0}$ be the iterates generated by Algorithm~\ref{alg:2nd_optimistic}, where $\hat{\eta}_k = {\eta_{k-1}}/{(1+\mu\eta_{k-1})}$, $\alpha\in (0,1)$, $\beta\in (0,1)$, and $\sigma_0>0$.  Then for any $\epsilon>0$, we have $\zeta_N \leq \frac{(2-\alpha)\epsilon}{2{D_{\Phi}(\vz^*,\vz_0)}}$ after at most the following number of iterations 
   \begin{equation}\label{eq:zeta_2nd_global_bound}
    N \leq \max\left\{\frac{1}{1-2^{-\frac{1}{3}}}C_2^{\frac{1}{3}}+ \log_2\left(\frac{2{D_{\Phi}(\vz^*,\vz_0)}}{(2-\alpha)\epsilon}\right)+1, 1\right\},   
   \end{equation}
   where $C_2$ is defined in \eqref{eq:def_C2}. 
\end{theorem}
 \begin{proof}
     Without loss of generality, we assume that $ \frac{(2-\alpha)\epsilon}{2{D_{\Phi}(\vz^*,\vz_0)}}\leq 1$; otherwise the result becomes trivial as $\zeta_0 = 1$. 
    Based on \eqref{eq:def_zeta},  $\{\zeta_k\}_{k \geq 0}$ is non-increasing in $k$. Hence, from Lemma~\ref{lem:multiplying_factor} we have
    \begin{equation*}
        \frac{1}{\zeta_{k_2}} 
        \geq \frac{1}{\zeta_{k_1}}+\frac{1}{\sqrt{C_2}}\left(\sum_{k=k_1}^{k_2-1} \frac{1}{\zeta_k} \right)^{\frac{3}{2}} 
        \geq \frac{1}{\zeta_{k_1}}+\frac{1}{\sqrt{C_2}}\left(\frac{k_2-k_1}{\zeta_{k_1}}\right)^{\frac{3}{2}}. 
    \end{equation*}
    In particular, this implies 
    $\zeta_{k_2}\leq \frac{1}{2}\zeta_{k_1}$ when we have $k_2-k_1\geq C_2^{\frac{1}{3}} \zeta_{k_1}^{\frac{1}{3}}$.  Hence, for any integer $l\geq 0$, we can prove by induction that the number of iterations $N$ required to achieve 
    $\zeta_N \leq 2^{-l}$ does not exceed 
    \begin{equation*}
       \sum_{k=0}^{l-1} \left \lceil C_2^{\frac{1}{3}} 2^{-\frac{1}{3}k}\right\rceil 
       \leq \sum_{k=0}^{l-1} \left(C_2^{\frac{1}{3}} 2^{-\frac{1}{3}k}+1\right) 
       \leq \frac{1}{1-2^{-\frac{1}{3}}}C_2^{\frac{1}{3}}+l.
    \end{equation*}
    The bound in \eqref{eq:zeta_2nd_global_bound} immediately follows by setting $l=\lceil \log_2(\frac{2{D_{\Phi}(\vz^*,\vz_0)}}{(2-\alpha)\epsilon})\rceil \leq \log_2\left(\frac{2{D_{\Phi}(\vz^*,\vz_0)}}{(2-\alpha)\epsilon}\right)+1$.  
\end{proof}

Theorem~\ref{thm:2nd_strongly_convex_global} guarantees the global convergence of our proposed second-order optimistic method, since it shows that we can achieve an arbitrary accuracy $\epsilon$ after running the number of iterations given in \eqref{eq:zeta_2nd_global_bound}. Better yet, we proceed to show that our method eventually achieves a fast local R-superlinear convergence rate. This result is formally stated in the following theorem.

\begin{theorem}[Local convergence]\label{thm:2nd_strongly_convex_local}
    {Under the same conditions as in Theorem~\ref{thm:2nd_strongly_convex_global},} for any $k\geq 0$, we have 
    ${\zeta_{k+1}}   \leq \max\left\{{\gamma_2}\kappa_2(\vz_0), \frac{\beta^{k}}{\sigma_0 \mu} \right\}
     \zeta_k^{\frac{3}{2}}$, 
        where ${\kappa}_2(\vz_0)$ and $\gamma_2$ are defined in \eqref{eq:condition_number} and \eqref{eq:def_gamma}, respectively.
\end{theorem} 
\begin{proof}
    By the definition of $\zeta_k$ in \eqref{eq:def_zeta}, we have 
    $\zeta_{k+1} = \frac{\zeta_k}{1+\eta_{k}\mu} \leq \frac{\zeta_k}{\eta_{k}\mu}$. Moreover, Lemma~\ref{lem:induction_bound} implies that $\frac{1}{\eta_k} \leq \max\left\{\gamma_2 L_2 \sqrt{D_{\Phi}(\vz^*,\vz_0)}, \frac{\beta^{k}}{\sigma_0} \right\} \prod_{l=0}^{k-1} (1+\eta_l\mu)^{-\frac{1}{2}} = \max\left\{\gamma_2 L_2 \sqrt{D_{\Phi}(\vz^*,\vz_0)}, \frac{\beta^{k}}{\sigma_0} \right\} \zeta_k^{-\frac{1}{2}}$.
    Theorem~\ref{thm:2nd_strongly_convex_local} follows from the two inequalities and the definition of $\kappa_2(\vz_0)$ in \eqref{eq:condition_number}. 
\end{proof}
Theorem~\ref{thm:2nd_strongly_convex_local} implies that the sequence $\{\zeta_k\}_{k \geq 0}$ converges to 0 at a superlinear convergence rate of order $\frac{3}{2}$ once we have $\zeta_k < \min\left\{\frac{1}{\gamma_2^2 \kappa^2_2(\vz_0)},\frac{\sigma_0^2 \mu^2}{\beta^{2k}}\right\}$. Due to the result in Theorem~\ref{thm:2nd_strongly_convex_global}, the sequence indeed eventually falls below this threshold after a finite number of iterations. Hence, in light of the bound in \eqref{eq:bound_with_zeta}, this in turn implies the local R-superlinear convergence of $D_{\Phi}(\vz^*,\vz_k)$.

By combining the global convergence result in Theorem~\ref{thm:2nd_strongly_convex_global} and the local convergence result in Theorem~\ref{thm:2nd_strongly_convex_local}, we characterize the overall iteration complexity of the second-order optimistic method in Algorithm~\ref{alg:2nd_optimistic}. Specifically, if the required accuracy $\epsilon$ is moderate, we simply use the global complexity result in Theorem~\ref{thm:2nd_strongly_convex_global}. On the other hand, if $\epsilon$ is small, we upper bound the number of iterations needed to reach the local convergence neighborhood using Theorem~\ref{thm:2nd_strongly_convex_global}. Subsequently, we require at most $\bigO(\log\log(1/\epsilon))$ additional iterations to achieve the desired accuracy. We summarize the corresponding complexity bounds in Corollary~\ref{coro:2nd_complexity_bound}. %

\begin{corollary}\label{coro:2nd_complexity_bound}
    Suppose $\epsilon$ is the required accuracy that we aim for, i.e., $D_{\Phi}(\vz^*,\vz)\leq \epsilon$. {Under the same conditions as in Theorem~\ref{thm:2nd_strongly_convex_global}}, we have the following complexity bound for Algorithm~\ref{alg:2nd_optimistic}.
    
    \begin{itemize}
        \item If $\epsilon\geq\frac{\mu^2}{(2-\alpha)\gamma^2_2 L^2_2} $, then $D_{\Phi}(\vz^*,\vz_N) \leq \epsilon$ when the number of iterations $N$ exceeds
        \begin{equation*}
            \max \left\{\frac{1}{1-2^{-\frac{1}{3}}} C_2^{\frac{1}{3}}
            + \log_2 \left(\frac{{2D_{\Phi}(\vz^*,\vz_0)}}{(2-\alpha)\epsilon}\right)+1,1\right\},
        \end{equation*}
        where $C_2$ is defined in \eqref{eq:def_C2}. 
        \item If $\epsilon<\frac{\mu^2}{(2-\alpha)\gamma^2_2 L^2_2} $, then $D_{\Phi}(\vz^*,\vz_N) \leq \epsilon$ when the number of iterations $N$ exceeds  
        \begin{equation} \label{eq:2nd_complexity_after_local}
            \begin{aligned}
                &\max \left\{\frac{1}{1-2^{-\frac{1}{3}}}C_2^{\frac{1}{3}}
                + 2\log_2 \left(\gamma_2{\kappa}_2(\vz_0)\right)+2,-\log_{\frac{1}{\beta}}(\gamma_2 \sigma_0 L_2 \sqrt{D_{\Phi}(\vz^*,\vz_0)})+1,1\right\} \\
                &+ \log_{\frac{3}{2}}\log_{2} \left(\frac{2\mu^2}{(2-\alpha)\gamma^2_2 L^2_2\epsilon}\right)+1. 
            \end{aligned}
        \end{equation}
    \end{itemize}
\end{corollary}
\begin{proof}
    Recall that $D_{\Phi}(\vz^*,\vz)\leq \epsilon$ if $\zeta_k \leq \frac{(2-\alpha)\epsilon}{2{D_{\Phi}(\vz^*,\vz_0)}}$ by \eqref{eq:bound_with_zeta}. Hence, it suffices to upper bound the number of iterations required such that the latter condition holds. Also, we only need to prove the second case where $\epsilon < \frac{\mu^2}{(2-\alpha)\gamma^2_2 L^2_2}$, as the first case directly follows from Theorem~\ref{thm:2nd_strongly_convex_global}. 
    Let $N_1$ be the smallest integer such that $\zeta_{N_1} \leq 1/(2\gamma_2^2{\kappa}_2^2(\vz_0))$ and $\frac{\beta^{N_1}}{\sigma_0} \leq \gamma_2 L_2 \sqrt{D_{\Phi}(\vz^*,\vz_0)}$. By setting $\epsilon = \frac{D_{\Phi}(\vz^*,\vz_0)}{(2-\alpha)\gamma_2^2 \kappa_2^2(\vz_0)}=\frac{\mu^2}{(2-\alpha)\gamma^2_2 L^2_2}$ in Theorem~\ref{thm:2nd_strongly_convex_global}, we obtain that the first condition is satisfied when  
    \begin{align}
        N_1 &\geq \max\left\{\frac{1}{1-2^{-\frac{1}{3}}}C_2^{\frac{1}{3}}+ \log_2\left(\frac{2\gamma_2^2 L_2^2 {D_{\Phi}(\vz^*,\vz_0)}}{\mu^2}\right)+1, 1\right\} \nonumber \\
        & = \max\left\{\frac{1}{1-2^{-\frac{1}{3}}}C_2^{\frac{1}{3}}+ 2\log_2\left(\gamma_2 \kappa_2(\vz_0)\right)+2, 1\right\}. \label{eq:bound_on_N1}
    \end{align}
    Moreover, the second condition is satisfied when $N_1 \geq -\log_{1/\beta}(\gamma_2 \sigma_0 L_2 \sqrt{D_{\Phi}(\vz^*,\vz_0)})$. Since $N_1$ is the smallest integer to satisfy both conditions, we have 
    \begin{equation*}
        N_1 \leq \max\left\{\frac{1}{1-2^{-\frac{1}{3}}}C_2^{\frac{1}{3}}+ 2\log_2\left(\gamma_2 \kappa_2(\vz_0)\right)+2, -\log_{\frac{1}{\beta}}(\gamma_2 \sigma_0 L_2 \sqrt{D_{\Phi}(\vz^*,\vz_0)})+1,1\right\}.
    \end{equation*}
    Furthermore, since $\frac{\beta^{N_1}}{\sigma_0} \leq \gamma_2 L_2 \sqrt{D_{\Phi}(\vz^*,\vz_0)}$, we have ${\zeta_{k+1}}\leq \gamma_2 \kappa_2(\vz_0)\zeta_k^{\frac{3}{2}}$ by Theorem~\ref{thm:2nd_strongly_convex_local} for all $k \geq N_1$, which implies that $\gamma_2^2{\kappa}_2^2(\vz_0) \zeta_{k+1} \leq \left(\gamma_2^2{\kappa}_2^2(\vz_0) \zeta_k\right)^{\frac{3}{2}}$.
    By induction, we can prove that 
    $\gamma_2^2{\kappa}_2^2(\vz_0) \zeta_{k}\leq 2^{-(\frac{3}{2})^{k-N_1}}$ for all $k\geq N_1$. Hence, the additional number of iterations required to achieve $\zeta_k \leq \frac{(2-\alpha)\epsilon}{2{D_{\Phi}(\vz^*,\vz_0)}}$ does not exceed
    \begin{equation}\label{eq:iterations_after_local}
        \left\lceil \log_{{3}/{2}}\log_2 \left(\frac{2D_{\Phi}(\vz^*,\vz_0)}{(2-\alpha)\gamma^2_2 \kappa_2^2(\vz_0)\epsilon}\right) \right\rceil  \leq \log_{3/2}\log_{2} \left(\frac{2\mu^2}{(2-\alpha)\gamma^2_2 L^2_2\epsilon}\right)+1,
    \end{equation}
    where we used the definition of $\kappa_2(\vz_0)$ in \eqref{eq:condition_number}. 
    The result in \eqref{eq:2nd_complexity_after_local} now follows from \eqref{eq:bound_on_N1} and \eqref{eq:iterations_after_local}. 
\end{proof}

 Corollary~\ref{coro:2nd_complexity_bound} characterizes the overall iteration complexity of our proposed method. To summarize, 
if the required accuracy $\epsilon$ is larger than $\frac{\mu^2}{(2-\alpha)\gamma^2_2 L^2_2}$, then the iteration complexity is bounded by 
\begin{equation}\label{eq:global_complexity}
    \mathcal{O}\biggl( \biggl(\frac{L_2\sqrt{D_{\Phi}(\vz^*,\vz_0)}}{\mu}\biggr)^{\frac{2}{3}} +\log \left(\frac{{D_{\Phi}(\vz^*,\vz_0)}}{\epsilon}\right)\biggr).
\end{equation} 
Otherwise, for $\epsilon$ smaller than $\frac{\mu^2}{(2-\alpha)\gamma^2_2 L^2_2}$, the overall iteration complexity is bounded by 
\begin{equation}\label{eq:global_and_local_complexity}
    \mathcal{O}\biggl( \biggl(\frac{L_2\sqrt{D_{\Phi}(\vz^*,\vz_0)}}{\mu}\biggr)^{\frac{2}{3}}
   +\log\log \left(\frac{\mu^2}{L_2^2\epsilon}\right)
   \biggr).
\end{equation}
In \eqref{eq:global_and_local_complexity}, the first term captures the number of required iterations to reach the local neighborhood and the second term corresponds to the number of iterations in the local neighborhood to achieve accuracy~$\epsilon$.

\noindent\textbf{Comparison with  \cite{Ostroukhov2020}.}
We close the discussion on our iteration complexity by mentioning that similar complexity bounds to \eqref{eq:global_complexity} and \eqref{eq:global_and_local_complexity} were also established in \cite{Ostroukhov2020} but via a very different approach. Specifically, the authors of that paper proposed to apply a restarting technique on the extragradient-type method in \cite{Bullins2022higher}. Under Assumptions~\ref{assum:strongly_monotone} and~\ref{assum:second_order}, they proved a complexity bound of $\bigO((\frac{L_2R}{\mu})^{\frac{2}{3}}\log \frac{\mu R^2}{\epsilon})$ in terms of the primal-dual gap, where $R$ is an upper bound on $\|\vz_0-\vz^*\|$. Moreover, with the additional assumptions that $F$ is $L_1$-Lipschitz and $L_2$-smooth, they further combined it with the cubic-regularization-based method in \cite{Huang2022cubic} and improved the complexity bound to $\bigO((\frac{L_2R}{\mu})^{\frac{2}{3}}\log \frac{L_1L_2R}{\mu^2}+\log\log\frac{L_1^3}{\mu^2 \epsilon})$. 
Compared with ours, their method is restricted to the unconstrained problems in the Euclidean setup, and can only achieve local superlinear convergence under stronger assumptions.
Also, our method appears to be simpler as we do not need to switch between different update rules. Most importantly, they did not provide any detail on solving the subproblems nor identify the procedure and cost of finding an admissible step size, while we formally propose a line search scheme and identify the number of subsolver calls that our proposed second-order method requires. This is the topic of the next subsection.

\subsection{Complexity bound of the line search scheme}\label{subsec:complexity_LS}

So far, we have characterized the number of iterations required to solve the saddle point problem in \eqref{eq:minimax} using the second-order method in Algorithm~\ref{alg:2nd_optimistic} for both convex-concave and strongly-convex strongly-concave settings. The only missing piece of our analysis is characterizing the total number of calls to the optimistic subsolver during the whole process of our proposed method,  where each call solves an instance of \eqref{eq:2nd_order_subproblem}. 

\begin{theorem}\label{thm:steps_upper_2nd}
    Consider Algorithm~\ref{alg:2nd_optimistic} with parameters  $\alpha\in (0,1)$, $\beta\in (0,1)$,  and $\sigma_0>0$. For both convex-concave (Theorem~\ref{thm:OMD_2nd_ls_convex}) and strongly-convex-strongly-concave (Theorem~\ref{thm:2nd_strongly_convex_global}) settings, after $N$ iterations, the total number of calls to the optimistic subsolver does not exceed 
    $
    \max\left\{2N-1 + \log_{\frac{1}{\beta}}(\sigma_0\gamma_2 L_2 \sqrt{D_{\Phi}(\vz^*,\vz_0)}), N\right\}$,
    where $\gamma_2$  is a constant defined in \eqref{eq:def_gamma}.
\end{theorem}
\begin{proof}
By Proposition~\ref{prop:ls_total}, the number of calls to the optimistic subsolver at the $k$-th iteration is given by $\log_{\frac{1}{\beta}}\frac{\sigma_k}{\eta_k}+1$. Since $\sigma_k=\eta_{k-1}\sqrt{1+\eta_{k-1}\mu}/\beta$ for $k\geq 1$,  the total number of calls after $N$ iterations is
\begin{align}
    N_{\mathrm{total}} = \sum_{k=0}^{N-1} \left(\log_{\frac{1}{\beta}}\frac{\sigma_k}{\eta_k}+1\right) &= N + \log_{\frac{1}{\beta}}\frac{\sigma_0}{\eta_0} + \sum_{k=1}^{N-1} \log_{\frac{1}{\beta}}\frac{\eta_{k-1}\sqrt{1+\eta_{k-1}\mu}}{\beta\eta_k} \\
    &= 2N-1 + \log_{\frac{1}{\beta}}\frac{\sigma_0}{\eta_0} + \sum_{k=1}^{N-1} \log_{\frac{1}{\beta}}\frac{\eta_{k-1}\sqrt{1+\eta_{k-1}\mu}}{\eta_k} \\
    &= 2N -1 + \log_{\frac{1}{\beta}} \left(\frac{\sigma_0}{\eta_{N-1}}\prod_{k=0}^{N-2} \sqrt{1+\eta_{k}\mu} \right). \label{eq:line_search_bound} %
\end{align}
In addition, we have $\frac{1}{\eta_{N-1}^2} \prod_{k=0}^{N-2} (1+\eta_k\mu) \leq \max\left\{\gamma_2^2 L_2^2 D_{\Phi}(\vz^*,\vz_0), \frac{\beta^{2(N-1)}}{\sigma^2_0} \right\}$ by Lemma~\ref{lem:induction_bound}. Hence, together with \eqref{eq:line_search_bound}, we can upper bound $N_{\mathrm{total}}\leq 2N-1 + \max\{ \log_{\frac{1}{\beta}}(\sigma_0\gamma_2 L_2 \sqrt{D_{\Phi}(\vz^*,\vz_0)}), -(N-1)\}=\max\{2N-1 + \log_{\frac{1}{\beta}}(\sigma_0\gamma_2 L_2 \sqrt{D_{\Phi}(\vz^*,\vz_0)}), N\}$. This completes the proof. 
\end{proof}

This result shows that if we run our proposed second-order optimistic method in Algorithm~\ref{alg:2nd_optimistic} for $N$ iterations, the total number of calls to a subsolver of \eqref{eq:2nd_order_subproblem} is of the order $2N + \bigO(\log(\sigma_0 L_2 \sqrt{D_{\Phi}(\vz^*,\vz_0)}))$. Hence, the average number of calls per iteration can be bounded by a constant close to 2 for large~$N$. In comparison, the line search subroutines presented in \cite{Monteiro2012,Bullins2022higher} require a total number of subsolver calls of $\bigO(N \log\log(\frac{1}{\epsilon}))$ and $\bigO(N \log(\frac{1}{\epsilon}))$, respectively.

%% file: higher_order.tex
\section{Higher-order generalized optimistic method}\label{sec:high_order}
In this section, we explore the possibility of designing methods for saddle point problems using higher-order derivatives. Similar to the case in Section~\ref{sec:second_order}, \blue{for technical reasons we require that the distance-generating function $\Phi$ is $L_{\Phi}$-smooth w.r.t. $\|\cdot\|$ on $\mathcal{Z}$ and 1-strongly convex  on $\mathbb{R}^{m+n}$.} Also, we make the following assumption on $f$.  
\begin{assumption}\label{assum:high_order}
    The operator $F$ defined in \eqref{eq:def_of_F_H} is $p$-th-order $L_p$-smooth ($p \geq 3$) on $\mathcal{Z}$. Also, we have access to an oracle that returns $F(\vz),DF(\vz),\dots,DF^{(p-1)}(\vz)$ for any given $\vz$.   
\end{assumption}
Under Assumption~\ref{assum:high_order}, it follows from \eqref{eq:residual_F} that 
\begin{equation}\label{eq:higher_smooth}
    \|F(\vz)-T^{(p-1)}(\vz;\vz_k)\|_{*}\leq \frac{L_p}{p!}\|\vz-\vz_k\|^p, \quad \forall \vz \in \mathcal{Z},
\end{equation}
where $T^{(p-1)}(\vz;\vz_k)$ is the $(p-1)$-th order Taylor expansion of $F$ at point $\vz_k$ (cf. \eqref{eq:taylor_expansion}).  
A natural generalization of the second-order optimistic method in Section~\ref{sec:second_order} is to choose the approximation function as $P(\vz;\mathcal{I}_k):=T^{(p-1)}(\vz;\vz_k)$. Alas, a subtle technical issue arises on closer inspection of the resulting subproblem. Specifically, with this choice of $P(\vz;\mathcal{I}_k)$ we need to solve the inclusion problem: 
\begin{equation}\label{eq:high_order_subproblem_nonmonotone}
    0\in \eta_k T^{(p-1)}(\vz;\vz_k)+\hat{\eta}_k\left(F(\vz_k)-T^{(p-1)}(\vz_k;\vz_{k-1})\right)+\eta_k H(\vz) + \nabla \Phi(\vz)-\nabla \Phi(\vz_k).
\end{equation}
We observe that this is a non-monotone inclusion problem, since the operator $T^{(p-1)}(\vz;\vz_k)$, being the $(p-1)$-th Taylor approximation of a monotone operator $F$, is not monotone for $p\geq 3$ in general\footnote{For a counterexample, note that any multivariate polynomial of odd degree in $\mathbb{R}^d$ is not convex. Hence, when $p$ is odd, the $p$-th-order Taylor expansion of a convex function $f$ is not convex, which implies that the $(p-1)$-th-order Taylor expansion of the monotone gradient operator $\nabla f$ is not monotone.}. The lack of monotonicity not only makes solving the problem in \eqref{eq:high_order_subproblem_nonmonotone} potentially hard, but also impedes the analysis of our line search procedure. 

Inspired by the approach in~\cite{Nesterov2019}, we tackle this issue by using a regularized Taylor expansion in place of $T^{(p-1)}(\vz;\vz_k)$. Consider the operator $T^{(p-1)}_{\lambda}(\cdot;\vz):\mathbb{R}^d\rightarrow \mathbb{R}^d$ defined as
\begin{equation*}
    T^{(p-1)}_\lambda (\vz';\vz) := T^{(p-1)}(\vz';\vz)+ \frac{\lambda}{(p-1)!} (2D_{\Phi}(\vz',\vz))^{\frac{p-1}{2}}(\nabla \Phi(\vz')-\nabla \Phi(\vz)),
\end{equation*}
where $\lambda>0$ is the regularization parameter. 
Note that the regularization term is the gradient of the convex function $\frac{p\lambda}{(p+1)!}(2D_{\Phi}(\vz',\vz))^{\frac{p+1}{2}}$, and its norm is on the order of $\bigO(\|\vz'-\vz\|^p)$ since $\Phi$ is $L_{\Phi}$-smooth. 
Hence, for sufficiently large $\lambda$, we expect  $T^{(p-1)}_\lambda (\vz;\vz_k)$ to be a monotone operator of $\vz$, while maintaining a similar approximation error as in \eqref{eq:higher_smooth}. The following lemma confirms this intuition.  
\begin{lemma}\label{lem:property_high}
    \blue{Assume that $F$ is $p$-th-order $L_p$-smooth}. The following statements hold: 
    \begin{enumerate}[label=(\alph*)]
        \item \label{item:approx_high} For any $\lambda \geq 0$ we have \blue{$\| F(\vz')-T^{(p-1)}_{\lambda}(\vz';\vz)\|_* \leq \frac{L_p(\Phi,\lambda)}{p!}\|\vz'-\vz\|^p$ for any $\vz' \in \dom F$},
        where we denote
        \begin{equation}\label{eq:def_L_p}
            L_p(\Phi,\lambda) := L_p+p L_{\Phi}^{\frac{p+1}{2}} \lambda.
        \end{equation}
        \item \label{item:maximal_monotone_high} If $F$ is monotone, for $\lambda \geq L_p$, the operator $T^{(p-1)}_{\lambda}(\cdot;\vz)$ is maximal monotone with domain $\mathbb{R}^d$. 
    \end{enumerate}
\end{lemma}
\begin{proof}
    See Appendix~\ref{appen:property_high}. 
\end{proof}
Therefore, to instantiate the generalized optimistic method (Algorithm~\ref{alg:GOMD}) with a $p$-th-order oracle, we choose the approximation function as the regularized Taylor expansion $T_{\lambda}^{(p-1)}(\vz;\vz_k)$. Accordingly, the update rule for the proposed $p$-th-order optimistic method can be written as
\begin{equation}\label{eq:OMD_high_ls}
    0 \in   \eta_k T^{(p-1)}_{\lambda}(\vz_{k+1};\vz_k)
    +\hat{\eta}_k\left(F(\vz_k)-T^{(p-1)}_{\lambda}(\vz_k;\vz_{k-1})\right)+\eta_k H(\vz_{k+1})+\nabla \Phi(\vz_{k+1})-\nabla \Phi(\vz_k),
\end{equation}  
where the condition in \eqref{eq:stepsize_upper_bound} on $\eta_k$ can be written as 
\begin{equation}\label{eq:stepsize_higher}
    \eta_k \|F(\vz_{k+1})-T^{(p-1)}_{\lambda}(\vz_{k+1};\vz_{k}))\|_{*}\leq \frac{\alpha}{2}\|\vz_{k+1}-\vz_k\|.
\end{equation}

Apart from the regularization modification, the intuition and analysis for the $p$-th-order optimistic method are very similar to those for the second-order algorithm in Section~\ref{sec:second_order}. 
To see how the $p$-th-order information can lead to further speedup, note that the left hand side of \eqref{eq:stepsize_higher} is upper bounded by $\frac{\eta_kL_p(\Phi,\lambda)}{p!}\|\vz_{k+1}-\vz_k\|^p$ by Lemma~\ref{lem:property_high}(a). This suggests that we can pick our step size $\eta_k$ as 
\begin{equation}\label{eq:approx_step_high}
    \eta_k \approx \frac{ p! \alpha }{2  L_p(\Phi,\lambda)\|\vz_{k+1}-\vz_k\|^{p-1}}.
\end{equation}   
Compared with the step size in \eqref{eq:approx_step_2nd} for the second-order optimistic method, we can see that the step size in \eqref{eq:approx_step_high} will eventually be larger as the iterates approach the optimal solution and the displacement $\|\vz_{k+1}-\vz_k\|$ tends to zero. Hence, according to Proposition~\ref{prop:GOMD}, we shall expect better convergence rates for the $p$-th-order optimistic method. 
Also, to address the interdependence between the step size $\eta_k$ and the iterate $\vz_{k+1}$, we take the same approach as in Section~\ref{sec:second_order}: we use Algorithm~\ref{alg:ls} to select the step size in each iteration with a warm-start strategy for the choice of the initial step sizes $\{\sigma_k\}_{k \geq 1}$. Specifically, we set $\sigma_{k+1} = \eta_k/\beta$ in the convex-concave setting and set $\sigma_{k+1}=\eta_{k}(1+\eta_{k}\mu)^{\frac{p-1}{2}}/\beta$ in the strongly-convex-strongly-concave setting.  The resulting $p$-th-order optimistic method is formally described in Algorithm~\ref{alg:higher_optimistic}. 

Two remarks on the line search scheme in Algorithm~\ref{alg:higher_optimistic} follow. First, similar to the arguments in Section~\ref{sec:second_order}, we can guarantee that the line search scheme always terminates in finite steps by using the result in Lemma~\ref{lem:backtrack_terminate}. Hence, the step sizes $\{\eta_k\}_{k\geq 0}$ are well-defined. Second, at iteration $k$,  the optimistic subsolver in this case is required to solve a subproblem of $\vz$ in the form of 
\begin{equation}\label{eq:high_order_subproblem}
    0 \in \eta T^{(p-1)}_{\lambda}(\vz;\vz_{k})+\vv_k+\eta H(\vz) + \nabla \Phi(\vz)-\nabla \Phi(\vz_k)
\end{equation}
for given $\vv_k = \hat{\eta}_k(F(\vz_k)-T^{(p-1)}_{\lambda}(\vz_k;\vz_{k-1}))$ and $\eta>0$ (cf. Definition \ref{def:subsolver}). As shown in Lemma~\ref{lem:property_high}(b), the operator $T^{(p-1)}_{\lambda}(\vz;\vz_{k})$ is monotone in $\vz$ and hence the problem in \eqref{eq:high_order_subproblem} can be solved in general by many methods of monotone inclusion problems. %

\begin{algorithm}[!t]\small
    \caption{$p$-th-order optimistic method}\label{alg:higher_optimistic}
    \begin{algorithmic}[1]
        \STATE \textbf{Input:} initial point $\vz_0\in \mathcal{Z}$, initial trial step size $\sigma_0$, strong convexity parameter $\mu\geq 0$, regularization parameter $\lambda \geq L_p$, line search parameters $\alpha,\beta \in(0,1)$ 
        \STATE \textbf{Initialize:} set $\vz_{-1}\leftarrow \vz_0$ and $P(\vz_0;\mathcal{I}_{-1})\leftarrow F(\vz_0)$
        \FOR{iteration $k=0,\ldots,N-1$}
            \STATE Set $\hat{\eta}_k=\eta_{k-1}/(1+\eta_{k-1}\mu)$
            \STATE Select $\eta_{k}$ by Algorithm~\ref{alg:ls} with $\sigma=\sigma_k$, $\vz^-=\vz_k$, $P(\vz)=T^{(p-1)}_{\lambda}(\vz;\vz_k)$, and $\vv^{-}=\hat{\eta}_k(F(\vz_k)-T^{(p-1)}_{\lambda}(\vz_k;\vz_{k-1}))$
            \STATE Compute $\vz_{k+1}$ by solving  
             $
                0\in \eta_k T^{(p-1)}_{\lambda}(\vz;\vz_k)+\hat{\eta}_k\left(F(\vz_k)-T^{(p-1)}_{\lambda}(\vz_k;\vz_{k-1})\right)+\eta_k H(\vz) + \nabla \Phi(\vz)-\nabla \Phi(\vz_k)
            $
            \STATE Set $\sigma_{k+1} \leftarrow \eta_{k}(1+\eta_{k}\mu)^{\frac{p-1}{2}}/\beta$
        \ENDFOR
    \end{algorithmic}
\end{algorithm}

In the following sections, we derive convergence rates of the higher-order optimistic method and upper bound its total number of calls to the optimistic subsolver. Since the analysis largely mirrors the one for the second-order case in Section~\ref{sec:second_order}, we relegate the proofs to Appendix~\ref{appen:higher_order}.

\vspace{-1mm}
\subsection{Convergence analysis: convex-concave case}
\vspace{-1mm}

We first consider the case where $f$ is merely convex-concave. The ergodic convergence rate result below follows from Proposition~\ref{prop:GOMD}. 
\begin{theorem}\label{thm:OMD_high_ls_convex}
    Suppose Assumptions~\ref{assum:monotone} and~\ref{assum:high_order} hold and the distance-generating function $\Phi$ is $L_{\Phi}$-smooth.
    Let $\{\vz_k\}_{k\geq 0}$ be the iterates generated by Algorithm~\ref{alg:higher_optimistic}, where 
    $\hat{\eta}_k = \eta_{k-1}$, $\sigma_0>0$, $\lambda \geq L_p$, $\alpha\in (0,1)$, and $\beta\in (0,1)$. Then we have 
    $D_{\Phi}(\vz^*,\vz_k)\leq \frac{2}{2-\alpha} D_{\Phi}(\vz^*,\vz_0)$ for any $k\geq 0$.
    Moreover, for any $\vz = (\vx,\vy)\in \mathcal{X}\times \mathcal{Y}$,    
        \begin{equation*}
            \ell(\bar{\vx}_N,\vy)-\ell(\vx,\bar{\vy}_N) \leq 
            \left(1-\beta^{\frac{2}{p-1}}\right)^{-\frac{p-1}{2}} D_{\Phi}(\vz,\vz_0)\left(\sigma_0^{-\frac{2}{p-1}} +  \gamma_p^2 (L_p(\Phi,\lambda))^\frac{2}{p-1}D_{\Phi}(\vz^*,\vz_0) \right)^{\frac{p-1}{2}} N^{-\frac{p+1}{2}}, 
        \end{equation*}
         where  $\bar{\vz}_N=(\bar{\vx}_N,\bar{\vy}_N)$ is given by $\bar{\vz}_N=\frac{1}{\sum_{k=0}^{N-1} \eta_k}\sum_{k=0}^{N-1} \eta_k\vz_{k+1}$ and $\gamma_p$
        is a constant defined by 
        \begin{equation}\label{eq:def_gamma_p}
            \gamma_p = \sqrt{\frac{2}{1-\alpha}}L_{\Phi}^{3/2}\left(\frac{2}{p!\alpha\beta^p}\right)^{\frac{1}{p-1}}+\frac{1}{\sqrt{2(1-\alpha)}}\left(\frac{2}{p!\alpha\beta^p}\right)^{\frac{1}{p-1}}\alpha(\beta+L_{\Phi}^{1/2}).
        \end{equation}    
\end{theorem}
\begin{proof}
    See Appendix~\ref{appen:OMD_high_ls_convex}.
\end{proof}

Theorem~\ref{thm:OMD_high_ls_convex} shows that the $p$-th-order optimistic method converges at a rate of $\bigO(N^{-(p+1)/2})$ in terms of the primal-dual gap, faster than the rate of $\bigO(N^{-1})$ for first-order methods and the rate of $\bigO(N^{-3/2})$ for our second-order optimistic methods. As a corollary, to obtain a solution with a primal-dual gap of $\epsilon$, the proposed $p$-th-order optimistic method requires at most $\mathcal{O}(1/\epsilon^{2/(p+1)})$ iterations. 

\myalert{Comparison with \cite{Bullins2022higher}}. We note that similar iteration complexity bounds have also been reported in \cite{Bullins2022higher} for extragradient-type higher-order methods. However, the authors in \cite{Bullins2022higher} did not discuss how to solve the subproblem (which can be non-monotone) or how to select a step size that satisfies their specified conditions. In contrast, we propose and analyze our line search scheme in detail, and the total number of calls to the optimistic subsolver during the whole process after $N$ iterations can be also explicitly upper bounded, as we discuss later in Theorem~\ref{thm:steps_upper_high}.

\subsection{Convergence analysis: strongly-convex-strongly-concave case}
Next, we proceed to the setting where the smooth component $f$ of the objective in \eqref{eq:minimax} is $\mu$-strongly-convex-strongly-concave.
As discussed in Section~\ref{sec:second_order}, 
in this case the distance to the optimal solution ${D_{\Phi}(\vz^*,\vz_k)}$ is bounded above by the decreasing sequence $\{\zeta_k\}_{k \geq 0}$ defined in \eqref{eq:def_zeta}. Hence, it suffices to characterize the convergence behavior of the sequence $\zeta_k$.
To simplify the notation, we define 
\begin{equation}\label{eq:def_kappa_p_C_p}
    \tilde{\kappa}_p(\vz_0) := \frac{\gamma_p^{p-1} L_p(\Phi,\lambda)(D_{\Phi}(\vz^*,\vz_0))^{\frac{p-1}{2}}}{\mu} \quad \text{and} \quad C_p : = \left(1-\beta^{\frac{2}{p-1}}\right)^{-1} \biggl( \left(\frac{1}{\sigma_0 \mu}\right)^{\frac{2}{p-1}} +  (\tilde{\kappa}_p(\vz_0))^{\frac{2}{p-1}}\biggr). 
\end{equation}
where $L_p(\Phi,\lambda)$ and $\gamma_p$ are defined in \eqref{eq:def_L_p} and \eqref{eq:def_gamma_p}, respectively. \blue{The proofs for this section can be found in Appendix~\ref{appen:high-order-sc-sc}}. 
We {first} establish a global complexity bound for the proposed $p$-th-order method as in Theorem~\ref{thm:2nd_strongly_convex_global}.
\begin{theorem}[Global convergence]\label{thm:high_strongly_convex_global}
    Suppose that Assumptions~\ref{assum:strongly_monotone} and~\ref{assum:high_order} hold and the distance-generating function $\Phi$ is $L_{\Phi}$-smooth. Let $\{\vz_k\}_{k\geq 0}$ be the iterates generated by Algorithm~\ref{alg:higher_optimistic}, where $\hat{\eta}_k = {\eta_{k-1}}/{(1+\mu\eta_{k-1})}$, $\alpha\in (0,1)$, $\beta\in (0,1)$, $\sigma_0>0$ and $\lambda \geq L_p$. Further, recall the definition of $C_p$ in \eqref{eq:def_kappa_p_C_p}. 
    Then for any $\epsilon>0$, we have $\zeta_N \leq \frac{(2-\alpha)\epsilon}{2{D_{\Phi}(\vz^*,\vz_0)}}$ after at most the following number of iterations 
   \begin{equation}\label{eq:zeta_high_global_bound}
       N \leq \max\left\{\frac{1}{1-2^{-\frac{p-1}{p+1}}} C_p^{\frac{p-1}{p+1}}+ \log_2\left(\frac{2{D_{\Phi}(\vz^*,\vz_0)}}{(2-\alpha)\epsilon}\right)+1, 1\right\}.
   \end{equation}
 \end{theorem}

Theorem~\ref{thm:high_strongly_convex_global} guarantees the global convergence of our $p$-th-order optimistic method. Moreover, as in the case of our second-order optimistic method, the proposed higher-order optimistic method eventually achieves a fast local R-superlinear convergence rate, as stated in the following theorem. 

\begin{theorem}[Local convergence]\label{thm:high_strongly_convex_local}
    \blue{Under the same conditions as in Theorem~\ref{thm:high_strongly_convex_global},} for any $k \geq 0$, 
    we have \blue{${\zeta_{k+1}}   \leq \max\{\tilde{\kappa}_p(\vz_0), \frac{\beta^k}{\sigma_0\mu}\} {\zeta_k}^{\frac{p+1}{2}}$}, where $\tilde{\kappa}_p(\vz_0)$ is defined in \eqref{eq:def_kappa_p_C_p}. 
\end{theorem}

Theorem~\ref{thm:high_strongly_convex_local} implies that the sequence $\{\zeta_k\}_{k \geq 0}$ converges to 0 at a superlinear convergence rate of order $\frac{p+1}{2}$ once we have $\zeta_k < \min\left\{(\tilde{\kappa}_p(\vz_0))^{-\frac{2}{p-1}}, \left(\frac{\sigma_0 \mu}{\beta^k} \right)^{\frac{2}{p-1}} \right\}$. Due to the result in Theorem~\ref{thm:high_strongly_convex_global}, the sequence indeed eventually falls below this threshold after a finite number of iterations. Hence, in light of the bound in \eqref{eq:bound_with_zeta}, this in turn implies the local R-superlinear convergence of $D_{\Phi}(\vz^*,\vz_k)$.

By combining the global convergence result in Theorem~\ref{thm:high_strongly_convex_global} and the local convergence result in Theorem~\ref{thm:high_strongly_convex_local}, we characterize the overall computational complexity of the $p$-th-order optimistic method in Algorithm~\ref{alg:higher_optimistic}. 
As in Corollary~\ref{coro:2nd_complexity_bound} for the second-order method, the complexity result below is stated in two cases depending on the required accuracy $\epsilon$.  

\begin{corollary}\label{coro:high_complexity_bound}
    Suppose $\epsilon$ is the required accuracy that we aim for, i.e., $D_{\Phi}(\vz^*,\vz)\leq \epsilon$. {Under the same conditions as in Theorem~\ref{thm:high_strongly_convex_global}}, we have the following complexity bound for Algorithm~\ref{alg:higher_optimistic}.
        
        \begin{itemize}
            \item If $\epsilon\geq\frac{1}{(2-\alpha)\gamma^2_p}(\frac{\mu}{L_p(\Phi,\lambda)})^{\frac{2}{p-1}} $, {then =$D_{\Phi}(\vz^*,\vz_N) \leq \epsilon$ when the number of iterations $N$ exceeds}
            \begin{equation*}%
            \max\left\{\frac{1}{1-2^{-\frac{p-1}{p+1}}} C_p^{\frac{p-1}{p+1}}+ \log_2\left(\frac{2{D_{\Phi}(\vz^*,\vz_0)}}{(2-\alpha)\epsilon}\right)+1, 1\right\},
            \end{equation*}
            where $C_p$ is defined in \eqref{eq:def_kappa_p_C_p}. 
            \item If $\epsilon<\frac{1}{(2-\alpha)\gamma^2_p}(\frac{\mu}{L_p(\Phi,\lambda)})^{\frac{2}{p-1}}$, \blue{then $D_{\Phi}(\vz^*,\vz_N) \leq \epsilon$ when the number of iterations $N$ exceeds} 
            \begin{equation}\label{eq:high_complexity_after_local}
                \begin{aligned}
                    &\max\left\{\frac{1}{1-2^{-\frac{p-1}{p+1}}}C_p^{\frac{p-1}{p+1}}+ \frac{2}{p-1}\log_2 \tilde{\kappa}_p(\vz_0)+2, -\log_{\frac{1}{\beta}}(\sigma_0 \mu \tilde{\kappa}_p(\vz_0))+1,1\right\} \\
                    & + \log_{\frac{p+1}{2}}\log_2 \Biggl(\frac{2\mu^{\frac{2}{p-1}}}{(2-\alpha)\gamma^2_p (L_p(\Phi,\lambda))^{\frac{2}{p-1}}\epsilon}\Biggr)+1. 
                \end{aligned}
            \end{equation}
        \end{itemize}
    \end{corollary}

 Corollary~\ref{coro:high_complexity_bound} characterizes the overall iteration complexity of our proposed method. To simplify the discussions, assume that the regularization parameter is  $\lambda=\bigO(L_p)$ and the smoothness parameter $L_{\Phi}$ of $\Phi$ is bounded by some constant. As a result, $L_{p}(\Phi,\lambda)$ defined in \eqref{eq:def_L_p} is on the same order of $L_p$  and further $\tilde{\kappa}_p(\vz_0)$ in \eqref{eq:def_kappa_p_C_p} is on the same order of $\kappa_p(\vz_0)$ defined in \eqref{eq:condition_number}. If the required accuracy $\epsilon$ is larger than $\frac{1}{(2-\alpha)\gamma^2_p}(\frac{\mu}{L_p(\Phi,\lambda)})^{\frac{2}{p-1}}$, then the iteration complexity can be bounded by
\begin{equation}\label{eq:global_complexity_high}
    \mathcal{O}\biggl( \biggl(\frac{L_p (D_{\Phi}(\vz^*,\vz_0))^{\frac{p-1}{2}}}{\mu} \biggr)^{\frac{2}{p+1}}+\log \left(\frac{{D_{\Phi}(\vz^*,\vz_0)}}{\epsilon}\right)\biggr).
\end{equation}
Otherwise, for $\epsilon$ smaller than $\frac{1}{(2-\alpha)\gamma^2_p}(\frac{\mu}{L_p(\Phi,\lambda)})^{\frac{2}{p-1}}$, the overall iteration complexity is bounded by
\begin{equation}\label{eq:global_and_local_complexity_high}
\mathcal{O}\biggl( \biggl(\frac{L_p (D_{\Phi}(\vz^*,\vz_0))^{\frac{p-1}{2}}}{\mu} \biggr)^{\frac{2}{p+1}}
+
\log\log \biggl(\frac{\mu^\frac{2}{p-1}}{ L^{\frac{2}{p-1}}_p\epsilon}\biggr)
\biggr),
\end{equation}
The first term in \eqref{eq:global_and_local_complexity_high} captures the number of iterations required to reach the local neighborhood, while the second term corresponds to the number of iterations in the local neighborhood to achieve accuracy~$\epsilon$. It is worth noting that the work~\cite{Ostroukhov2020} also reported complexity bounds similar to \eqref{eq:global_complexity_high} and \eqref{eq:global_and_local_complexity_high}. The limitations of their results we discussed at the end of Section~\ref{subsec:2nd_sc_sc} also apply for the higher-order setting, and we refer the reader to the discussions therein.

\subsection{Complexity bound of the line search scheme}
As in Section~\ref{sec:second_order}, the final piece of our analysis is to characterize the complexity of our line search scheme in Algorithm~\ref{alg:higher_optimistic}. In the following theorem, we derive an upper bound on  the total number of calls to the optimistic subsolver required after $N$ iterations,  where each call solves an instance of \eqref{eq:high_order_subproblem}. 
This result holds for both convex-concave and strongly-convex strongly-concave settings. 
\begin{theorem}\label{thm:steps_upper_high}
    Consider Algorithm~\ref{alg:higher_optimistic} with parameters  $\alpha\in (0,1)$, $\beta\in (0,1)$, $\sigma_0>0$ and $\lambda \geq L_p$, and recall the definition of  $\gamma_p$ in \eqref{eq:def_gamma_p}. For both convex-concave (Theorem~\ref{thm:OMD_high_ls_convex}) and strongly-convex-strongly-concave (Theorem~\ref{thm:high_strongly_convex_global}) settings, after $N$ iterations, the total number of calls to the optimistic subsolver  does not exceed 
    $\max\left\{2N-1 + \log_{\frac{1}{\beta}}\left(\sigma_0\gamma_p^{p-1} L_p(\Phi,\lambda) (D_{\Phi}(\vz^*,\vz_0))^{\frac{p-1}{2}}\right), N \right\}$, where $\gamma_p$ is a constant defined in \eqref{eq:def_gamma_p}. 
\end{theorem}
\begin{proof}
    See Appendix~\ref{appen:steps_upper_high}. 
\end{proof}

This result shows that if we run our proposed $p$-th-order optimistic method in Algorithm~\ref{alg:higher_optimistic} for $N$ iterations to achieve a specific accuracy $\epsilon$,  the total number of calls to a subsolver of \eqref{eq:high_order_subproblem} is of the order $2N+\bigO(\log(\sigma_0 L_p (D_{\Phi}(\vz^*,\vz_0))^{\frac{p-1}{2}}))$. Hence, the average number of calls per iteration can be bounded by a constant close to 2 for large~$N$.

%% file: numerical.tex
\section{Numerical experiments}\label{sec:numerical}

\subsection{First-order optimistic method}

\blue{
\myalert{Convex-concave setting.} 
Consider the following matrix game
\begin{equation}\label{eq:matrix_game}
    \min_{\vx \in \mathrm{\Delta}^m} \max_{\vy\in \mathrm{\Delta}^n} \langle \mA\vx, \vy\rangle,
\end{equation}
where $\mA\in \mathbb{R}^{n\times m}$ and $\mathrm{\Delta}^m$ and $\mathrm{\Delta}^n$ are the standard unit simplices in $\mathbb{R}^m$ and $\mathbb{R}^n$, respectively. As a common strategy for simplex constraints, we choose 
$\Phi_\mathcal{X}(\vx) = \sum_{i=1}^m x_i\log x_i$ and $\Phi_{\mathcal{Y}}(\vy) = \sum_{i=1}^n y_i\log y_i$ with $\|\cdot\|_{\mathcal{X}}=\|\cdot\|_1$ and $\|\cdot\|_{\mathcal{Y}}=\|\cdot\|_1$. It is easy to verify that Assumption~\ref{assum:first_order} holds: the operator $F$ defined in \eqref{eq:def_of_F_H} is $L_1$-Lipschitz with $L_1 = \max_{i,j} |A_{i,j}|$. Moreover, our first-order optimistic method in \eqref{eq:OMD_first} can be instantiated as 
    \begin{align}
        \hat{\vx}_{k+1} &= \vx_k\cdot \exp(-\eta_k \nabla_{\vx} f(\vx_k,\vy_k)-{\eta}_k(\nabla_{\vx} f(\vx_k,\vy_k)-\nabla_{\vx}f(\vx_{k-1},\vy_{k-1}))), \label{eq:mirror_descent_x}\\
        \hat{\vy}_{k+1} &= \vy_k\cdot \exp(\eta_k \nabla_{\vy} f(\vx_k,\vy_k)+{\eta}_k(\nabla_{\vy} f(\vx_k,\vy_k)-\nabla_{\vy}f(\vx_{k-1},\vy_{k-1}))), \label{eq:mirror_descent_y}\\
        x_{k+1,i} &= \frac{\hat{x}_{k+1,i}}{\sum_{l=1}^m \hat{x}_{k+1,l}} \;\forall i=1,\dots,m \qquad \text{and} \qquad
        y_{k+1,j} = \frac{\hat{y}_{k+1,j}}{\sum_{l=1}^n \hat{y}_{k+1,l}} \; \forall j=1,\dots,n. \nonumber
    \end{align}
where in \eqref{eq:mirror_descent_x} and \eqref{eq:mirror_descent_y} we perform elementwise multiplication. The step size $\eta_k$ is fixed as $1/M$ in our fixed step size scheme (Option I in Algorithm~\ref{alg:first_optimistic}), while it is chosen adaptively in our line search scheme (Option II in Algorithm~\ref{alg:first_optimistic}). We set the initial point $(\vx_0,\vy_0)$ as $x_{0,i} = 1/m$ for all $i=1,2,\dots,m$ and $y_{0,j} = 1/n$ for all $j=1,2,\dots,n$.
}

\blue{
Since the problem in \eqref{eq:matrix_game} is convex-concave, we measure the quality of a solution $(\vx,\vy)$ by 
the primal-dual gap function, which has a closed form of $\Delta(\vx,\vy) = \max_i (\mA \vx)_i - \min_{j} (\mA^\top \vy)_j$. By taking the supremum on both sides of \eqref{eq:gap_bound_1st}, we obtain the following upper bound for the fixed step size scheme via Theorem~\ref{thm:OMD_first}:
\begin{equation}\label{eq:primal_dual_gap_1st}
    \Delta(\bar{\vx}_N,\bar{\vy}_N) \leq \frac{M }{N}\cdot \Bigl(\max_{\vx\in \mathrm{\Delta}^m} D_{\Phi_{\mathcal{X}}}(\vx,\vx_0) + \max_{\vy\in \mathrm{\Delta}^n} D_{\Phi_{\mathcal{Y}}}(\vy,\vy_0) \Bigr) = \frac{M(\log m+\log n)}{N},
\end{equation}
where $\bar{\vx}_N = \frac{1}{N} \sum_{k=0}^{N-1} \vx_{k+1} $ and $\bar{\vy}_N = \frac{1}{N} \sum_{k=0}^{N-1} \vy_{k+1} $. A similar bound for the gap at the averaged iterates of the line search scheme can also be derived from Theorem~\ref{thm:OMD_first_ls}.}

\begin{figure}[!t]
    \vspace{-1em}
    \centering
    \subfloat[Results by iteration]{\includegraphics[width=0.45\linewidth, clip=true]{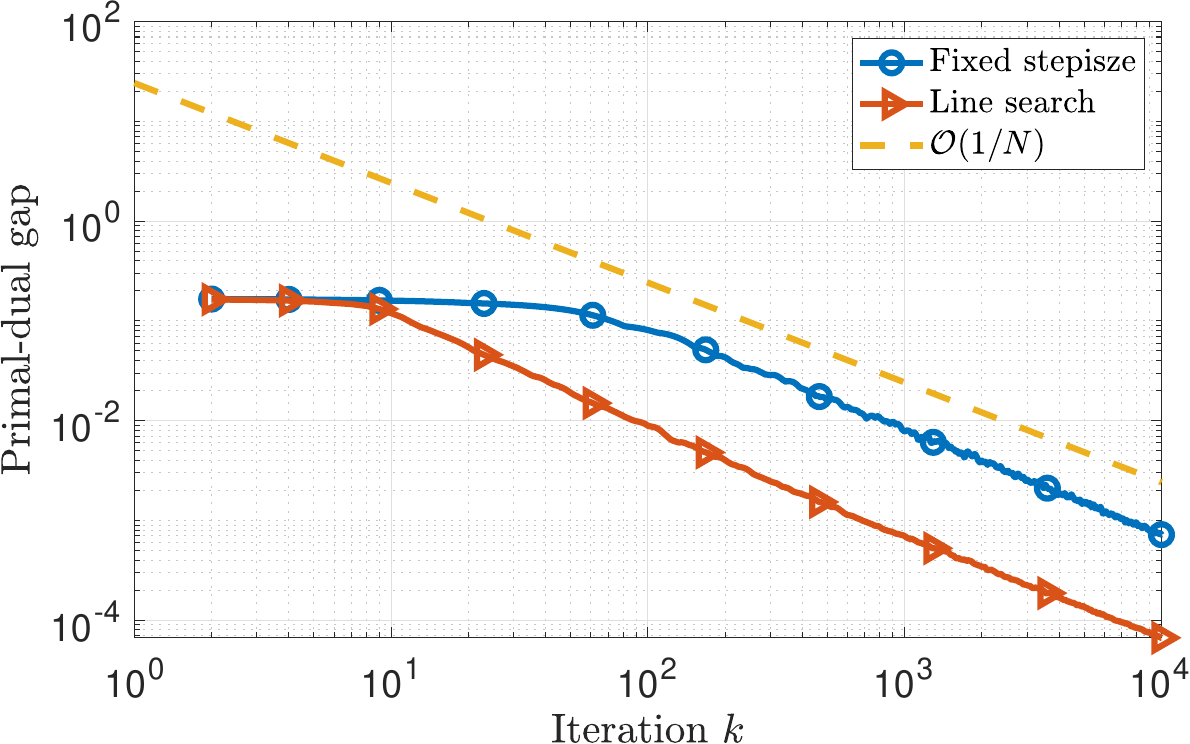}}
    \qquad
    \subfloat[Results by subsolver calls]{\includegraphics[width=0.45\linewidth, clip=true]{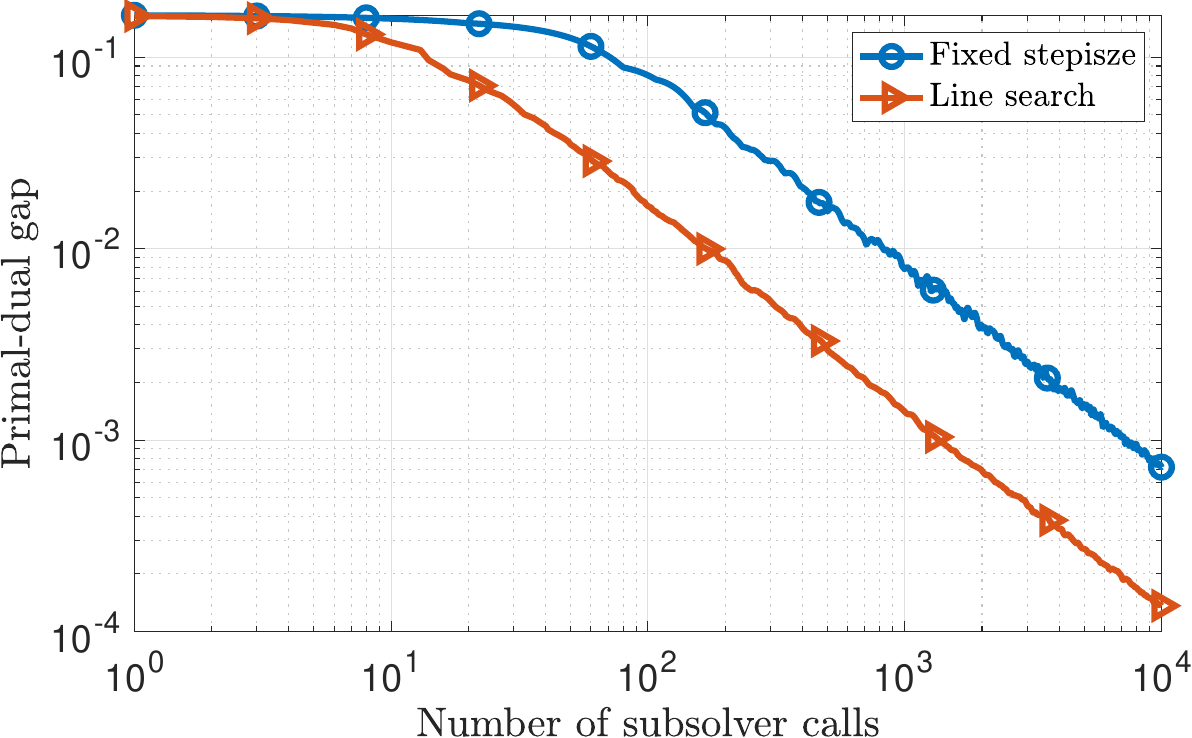}}
    \caption{The 
    {performance of}
    the first-order optimistic methods on solving the convex-concave saddle point problem in~\eqref{eq:matrix_game}. }\label{fig:1st_convex_concave}
\end{figure}

\blue{
In our experiment, we choose $m=600$, $n=300$, and generate the entries of $\mA$ independently and uniformly from the interval $[-1,1]$. For the fixed step size scheme, we choose $M=2L_1$, while for the adaptive line search scheme we choose $\alpha=1,\beta=0.8$, and $\sigma_0=1$.  
In Fig.~\ref{fig:1st_convex_concave}, we plot the primal-dual gap at the averaged iterate $(\bar{\vx}_N,\bar{\vy}_N)$ 
of both the fixed step size scheme and the line search scheme. The $\bigO(1/N)$ upper bound in \eqref{eq:primal_dual_gap_1st} is also shown for comparison. We can see that the averaged iterates in both schemes converge exactly at the rate of $\mathcal{O}(1/N)$, as predicted by our convergence analysis. Moreover, \blue{the line search scheme performs comparably to the fixed step size scheme},
and we note that it is able to achieve so without any prior knowledge of the Lipschitz constant of $F$.}

\myalert{Strongly-convex-strongly-concave setting.}
Consider the following composite saddle point problem with box constraints: 
\begin{equation}\label{eq:test_prob_1}
    \min_{\vx \in \mathcal{X}} \max_{\vy\in \mathcal{Y}}\;\;\; \langle\mA \vx-\vb, \vy\rangle+ \lambda \|\vx\|_1+\frac{\mu}{2}\|\vx\|_2^2-\lambda\|\vy\|_1-\frac{\mu}{2}\|\vy\|_2^2,
\end{equation}
where $\mA\in \mathbb{R}^{n\times m}$, $\vb\in \mathbb{R}^n$, $\mathcal{X} = \{\vx \in \mathbb{R}^m:\|\vx\|_{\infty}\leq R\}$ and $\mathcal{Y} = \{\vy\in \mathbb{R}^n: \|\vy\|_{\infty}\leq R\}$. This problem is an instance of Problem \eqref{eq:minimax} with $f(\vx,\vy)=\langle\mA \vx-\vb, \vy\rangle+\frac{\mu}{2}\|\vx\|_2^2-\frac{\mu}{2}\|\vy\|_2^2$, $h_1(\vx) = \lambda\|\vx\|_1$ and $h_2(\vy)=\lambda \|\vy\|_1$.  
 We consider the Euclidean case where $\|\cdot\|_{\mathcal{X}}=\|\cdot\|_2$, $\|\cdot\|_{\mathcal{Y}}=\|\cdot\|_2$ and $D_{\Phi}(\vz,\vz')=\frac{1}{2}\|\vz-\vz'\|^2_2$.
 In this case, the operator $F$ in \eqref{eq:def_of_F_H} is $L_1$-Lipschitz with $L_1$ being the operator norm of the matrix $    \begin{bsmallmatrix}
    \mu \mI_m & \mA^\mathsf{T} \\
    -\mA & \mu \mI_n
\end{bsmallmatrix}.$
Moreover, our first-order optimistic method in \eqref{eq:OMD_first} can be written as
\begin{equation*}%
    \begin{aligned}
        \vx_{k+1} &= T_{\eta_k\lambda,R}(\vx_{k}-\eta_k \nabla_{\vx} f(\vx_k,\vy_k)-\hat{\eta}_k(\nabla_{\vx} f(\vx_k,\vy_k)-\nabla_{\vx}f(\vx_{k-1},\vy_{k-1}))), \\
        \vy_{k+1} &= T_{\eta_k\lambda,R}(\vy_{k}+\eta_k \nabla_{\vy} f(\vx_k,\vy_k)+\hat{\eta}_k(\nabla_{\vy} f(\vx_k,\vy_k)-\nabla_{\vy}f(\vx_{k-1},\vy_{k-1}))),
    \end{aligned}
\end{equation*}
where the operator $T_{t,R}(\cdot):\mathbb{R} \rightarrow \mathbb{R}$ is defined by
\begin{equation*}
    T_{t,R}(z) = 
    \begin{cases}
        0, & \text{if }|z| \leq t; \\
        (|z|-t)\cdot \mathrm{sgn}(z) & \text{if } t<|z| \leq t+R; \\
        R \cdot \mathrm{sgn}(z) & \text{Otherwise},
    \end{cases}
\end{equation*}
and it is applied elementwise. The step size $\eta_k$ is fixed as $1/M$ in our fixed step size scheme (Option I in Algorithm~\ref{alg:first_optimistic}), while it is chosen adaptively in our line search scheme (Option II in Algorithm~\ref{alg:first_optimistic}). The initial point $(\vx_0,\vy_0)$ is chosen as the origin in $\mathbb{R}^{m+n}$.

When $\mu$ is positive, $f(\vx,\vy)$ is $\mu$-strongly-convex-strongly-concave. In this case, the problem in \eqref{eq:test_prob_1} has a unique saddle point $(\vx^*,\vy^*)$ and we measure the quality of a solution $(\vx,\vy)$ by its distance to $(\vx^*,\vy^*)$. Theorem~\ref{thm:OMD_first} provides the following linear convergence rate for the fixed step size scheme:  
\begin{equation}\label{eq:dist_to_opt_1st}
   \|\vz_N-\vz^*\|_2^2 \leq \|\vz_0-\vz^*\|_2^2 \left(\frac{M}{(1+s(\Phi))\mu+M}\right)^{N}=\|\vz_0-\vz^*\|_2^2 \left(\frac{M}{2\mu+M}\right)^{N},
\end{equation} 
where we also take the symmetry coefficient of $\Phi$ into account (cf. Remark~\ref{rem:symmetry}). A similar bound can also be derived for the line search scheme from Theorem~\ref{thm:OMD_first_ls}. Since the saddle point of the problem in \eqref{eq:test_prob_1} does not admit a closed form, in the experiment we approximate it by running the first-order optimistic method for a very long time (more than $10^5$ iterations).

\begin{figure}[!t]
    \centering
    \subfloat[Results by iteration]{\includegraphics[width=0.45\linewidth, clip=true]{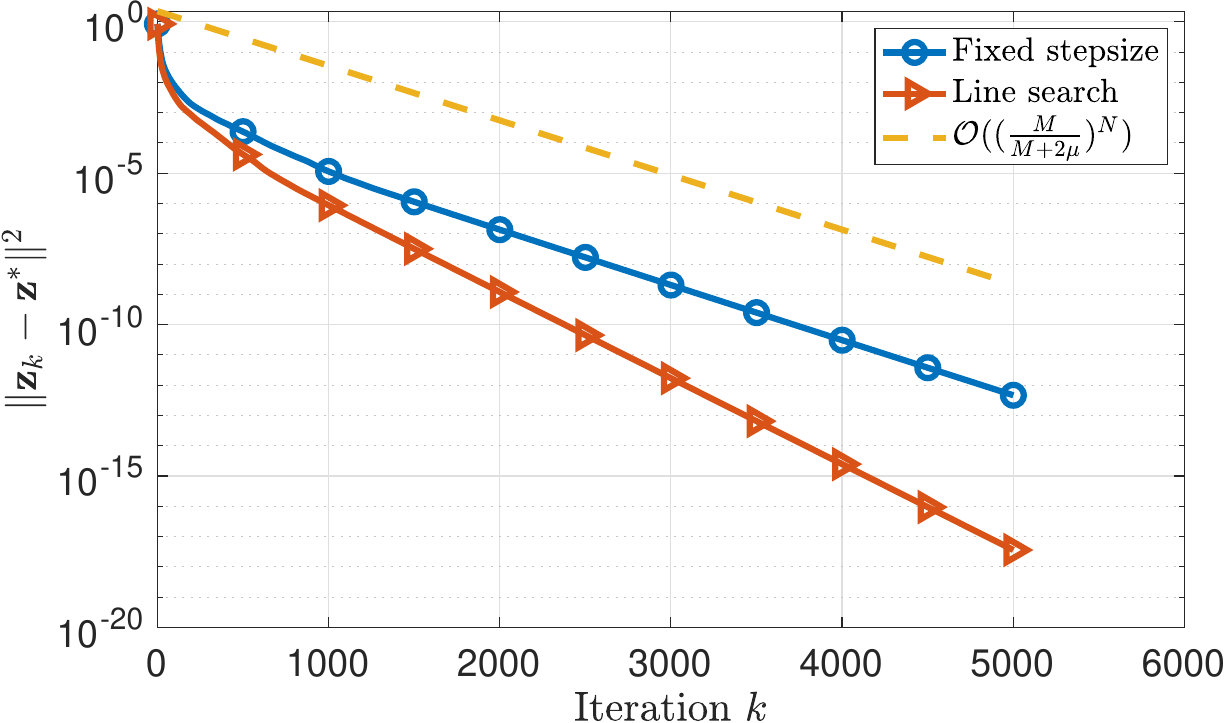}}
    \qquad 
    \subfloat[Results by subsolver calls]{\includegraphics[width=0.45\linewidth, clip=true]{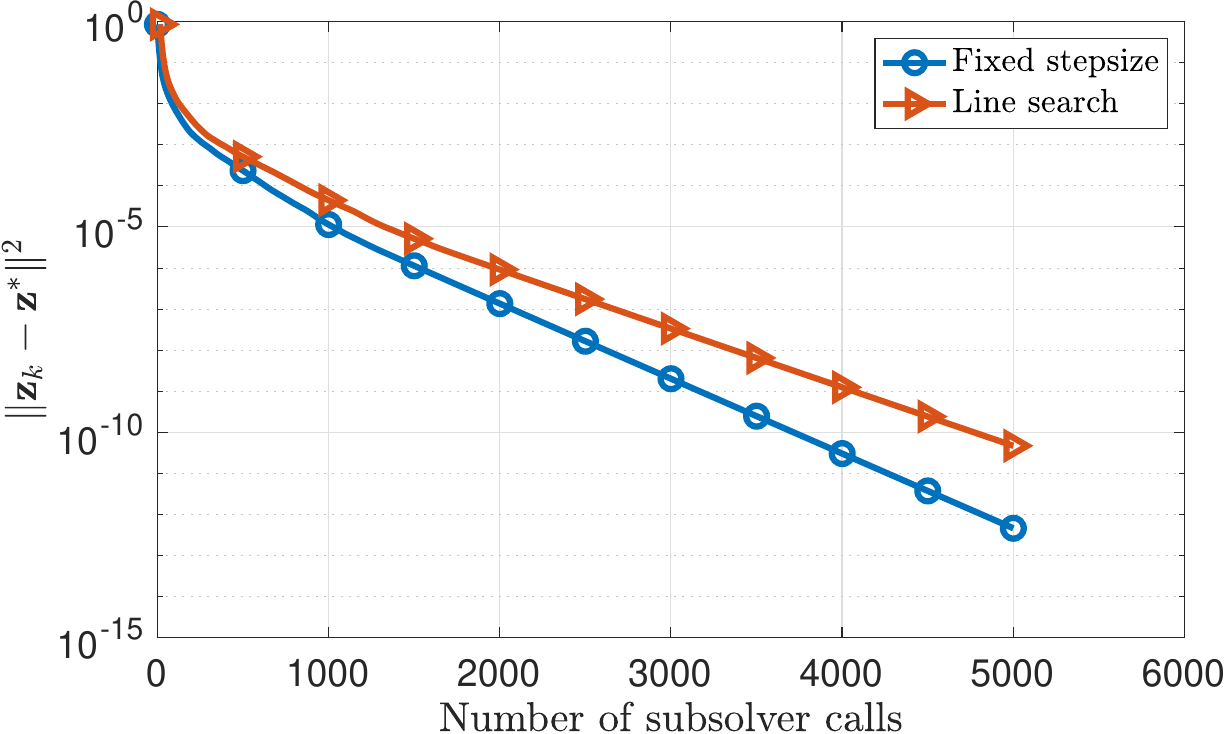}}
    \caption{The {performance of} the first-order optimistic methods on solving a strongly-convex-strongly-concave saddle point problem in \eqref{eq:test_prob_1}.}\label{fig:1st_strong}
\end{figure}

In our experiment, the problem parameters are chosen as $m=600$, $n=300$, $\lambda=0.1$, $\mu=0.1$, and $R=0.05$. For the fixed step size scheme, we choose $M=2L_1$, while for the line search scheme we choose $\alpha=1,\beta=0.8$, and $\sigma_0=1$. We generate the entries of $\mA$ and $\vb$ independently and uniformly from the interval $[-1,1]$.
In Fig.~\ref{fig:1st_strong}, we plot the distance to the saddle point $\vz^* = (\vx^*,\vy^*)$ of the fixed step size scheme and the line search scheme versus the number of iterations, along with the convergence bound in \eqref{eq:dist_to_opt_1st}. 
We observe that both schemes exhibit linear convergence, and the rate of the fixed step size scheme agrees well with our theory.
Moreover, as in the convex-concave setting, we can see that the line search scheme \blue{achieves a similar convergence rate as the fixed step size scheme}.

\myalert{The complexity of line search.} We also empirically evaluate the number of subsolver calls made in our line search scheme. To see the impacts of the line search parameters,  
we vary $\sigma_0$ and $\beta$ and run our method on both the convex-concave and strongly-convex-strongly-concave saddle point problems in \eqref{eq:matrix_game} and \eqref{eq:test_prob_1}. For each configuration, we generate 50 random instances and run our method until it finds a solution with accuracy $10^{-9}$ or completes 1,000 iterations. Then we compute the average number of subsolver calls per iteration in each run and report the maximum among the 50 runs in Table~\ref{tab:complexity_1}. From the table, we can see that the line search complexity is insensitive to the choice of the line search parameters. In fact, in all test instances our method requires nearly two calls to the subsolver per iteration on average, which verifies that our line search scheme is highly practical. 
\begin{table}[tbhp]
    {\footnotesize
    \caption{The maximum average number of subsolver calls per iteration in the first-order optimistic method among 50 random instances.}\label{tab:complexity_1}
    \vspace{-3mm}
    \begin{center}  
    \begin{tabular}{|c|c|c|}
        \hline
        \multirow{2}{*}{Parameters} & \multicolumn{2}{c|}{Max average number of subsolver calls per iteration} \\ \cline{2-3} 
                                                & Convex-concave setting & Strongly-convex-strongly-concave setting \\ \hline
        $\sigma_0=1,\beta = 0.5$                & 1.998           & 2.004 \\ \hline
        $\sigma_0=100, \beta =0.5$              & 2.004           & 2.011 \\ \hline
        $\sigma_0=10,000, \beta = 0.5$          & 2.011           & 2.018 \\ \hline
        $\sigma_0=1,\beta = 0.9$                & 1.986           & 2.033 \\ \hline
        $\sigma_0=100, \beta =0.9$              & 2.031           & 2.076 \\ \hline
        $\sigma_0=10,000, \beta = 0.9$          & 2.075           & 2.120 \\ \hline
        \end{tabular}  
    \end{center}
    }
    \vspace{-6mm}
    \end{table}
\subsection{Second-order optimistic method}
To test our second-order optimistic method, we consider unconstrained saddle point problems:

\begin{equation}\label{eq:test_prob_2}
    \min_{\vx\in \mathbb{R}^n} \max_{\vy\in \mathbb{R}^n}\;\;\; 
    f(\vx,\vy) = \frac{L_2}{6}\|\vx\|^3 + \langle \mA\vx-\vb,\vy \rangle + \frac{\mu}{2}\|\vx\|^2 - \frac{\mu}{2}\|\vy\|^2,
\end{equation}
where $L_2>0$, $\mu \geq 0$,  and the matrix $\mA\in \mathbb{R}^{n\times n}$ is chosen as $\mA = 
\begin{bsmallmatrix}
    1 & -1 & & & \\
      & 1  & -1 & & \\
      &    & \ddots & \ddots & \\
      & &  & 1 & -1 \\
      & & & & 1
\end{bsmallmatrix}$. Moreover, we generate the entries of the vector $\vb\in \mathbb{R}^n$ independently and randomly from the interval $[-1,1]$ and then normalize the vector such that $\|\vb\| = 1$. 
We can verify that Problem \eqref{eq:test_prob_2} is convex-concave when $\mu=0$ and strongly-convex-strongly-concave when $\mu >0$. Moreover, $F$ defined in \eqref{eq:def_of_F_H} is $L_2$-second-order smooth.
The subproblem in \eqref{eq:2nd_order_subproblem} for the optimistic subsolver is equivalent to solving a system of linear equations. Specifically, the update rule in \eqref{eq:OMD_2nd_ls} can be instantiated as
\begin{equation}\label{eq:omd_test_2}
    \vz_{k+1} = \vz_k - ({\mI}+\eta_k DF(\vz_k) )^{-1}(\eta_k F(\vz_k)+\vv_k),
\end{equation}
where $\vv_k = \hat{\eta}_k(F(\vz_k)-F(\vz_{k-1})-DF(\vz_{k-1})(\vz_k-\vz_{k-1}))$ and $\eta_k$ is determined by the adaptive line search scheme (see Algorithm~\ref{alg:2nd_optimistic}). 
In our experiments, we set the initial point as $(\vx_0,\vy_0)=(\mathbf{0},\mathbf{0})$.

\myalert{Convex-concave setting.}
We consider Problem~\eqref{eq:test_prob_2} with $\mu=0$. 
It can be verified that it has a unique saddle point $\vz^*=(\vx^*,\vy^*)$ given by $\vx^* = \mA^{-1}\vb$ and $\vy^* =  -\frac{L_2}{2}\|\vx^*\|(\mA^\mathsf{T})^{-1}\vx^*$. 
Since the feasible set is unbounded, we use the restricted primal-dual gap defined in \eqref{eq:restricted_gap_function} as the performance metric.  With $\mathcal{B}_1 = \mathbb{R}^n$ and $\mathcal{B}_2 = \{\vy:\|\vy\| \leq R\}$, the gap can be computed by
\begin{equation*}
    \Delta_{\mathcal{B}_1\times \mathcal{B}_2}(\vx,\vy) = \frac{L_2}{6}\|\vx\|^3+R\|\mA \vx -\vb\| + \frac{2}{3}\sqrt{\frac{2}{L_2}}\|\mA^\mathsf{T}\vy\|^{\frac{3}{2}}+\langle \vb, \vy \rangle.
\end{equation*}
By taking the supremum over $(\vx,\vy)\in \mathcal{B}_1\times \mathcal{B}_2$ on both sides of \eqref{eq:duality_gap}, we obtain $\vx=-\sqrt{\frac{2}{L_2}}\frac{\mA^\mathsf{T}\bar{\vy}_N}{\sqrt{\|\mA^\mathsf{T}\bar{\vy}_N\|}}$ and $\vy = \frac{R(\mA\bar{\vx}_N-\vb)}{\|\mA\bar{\vx}_N-\vb\|}$, leading to the following upper bound 
\begin{equation}\label{eq:eta_bound_2nd}
    \Delta_{\mathcal{B}_1\times \mathcal{B}_2}(\bar{\vx}_N,\bar{\vy}_N) \leq \frac{1}{2} \left(\frac{2}{L_2}\|\mA^\mathsf{T}\bar{\vy}_N\|+R^2\right)\left(\sum_{k=0}^{N-1} \eta_k\right)^{-1},
\end{equation}
where $\bar{\vx}_N=\frac{1}{\sum_{k=0}^{N-1} \eta_k}\sum_{k=0}^{N-1} \eta_k\vx_{k+1}$ and $\bar{\vy}_N=\frac{1}{\sum_{k=0}^{N-1} \eta_k}\sum_{k=0}^{N-1} \eta_k\vy_{k+1}$. 

\begin{figure}[t!]
    \centering
    \subfloat[Convex-concave setting]{\includegraphics[width=0.45\linewidth, clip=true]{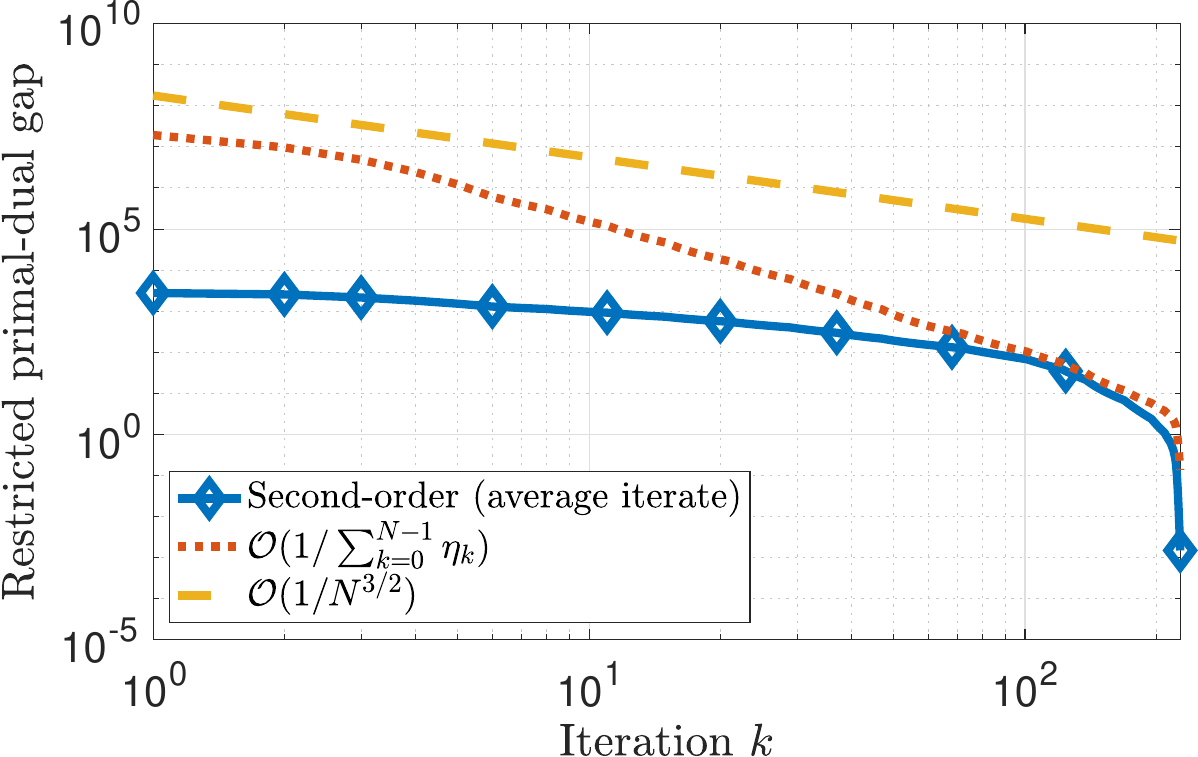}\label{fig:2nd_convex_concave}}\qquad
    \subfloat[Strongly-convex-strongly-concave setting]{\includegraphics[width=0.45\linewidth, clip=true]%
    {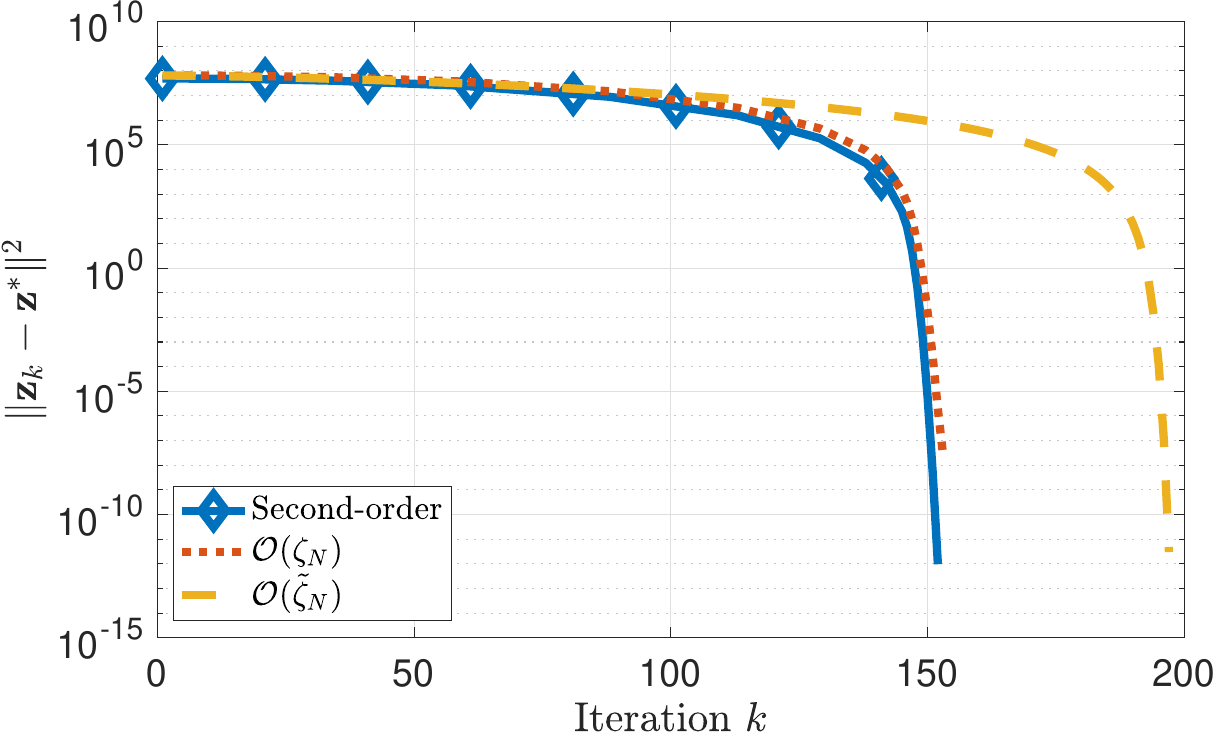}\label{fig:2nd_strong}}
    
    \caption{The {performance of}
    the second-order optimistic method on solving the saddle point problems in \eqref{eq:test_prob_2} with $\mu = 0$ and $\mu = 10^{-3}$, respectively.
    }%
\end{figure}

In our experiment, the problem parameters are chosen as $n=200$ and $L_2=10$, and the line search parameters are $\alpha=0.5$, $\beta=0.5$, $\sigma_0 = 1$, and $\epsilon = 10^{-10}$.  In Fig.~\ref{fig:2nd_convex_concave}, we plot the restricted primal-dual gap at the averaged iterate $(\bar{\vx}_N,\bar{\vy}_N)$ and the convergence bounds in \eqref{eq:eta_bound_2nd} for comparison.  
We observe that the second-order method exhibits a sublinear convergence rate in the early stage, which is comparable to our theoretical rate of $\bigO(N^{-\frac{3}{2}})$. 
Moreover, the bound in \eqref{eq:eta_bound_2nd} provides a tight upper bound for the second-order method especially in the later stage. 

\myalert{Strongly-convex-strongly-concave setting.}
We consider Problem~\eqref{eq:test_prob_2} with $\mu >0$.  
In this case, it has a unique saddle point $\vz^*=(\vx^*,\vy^*)$, which we compute numerically using MATLAB's built-in nonlinear equation solver. 
By Proposition~\ref{prop:GOMD}, {for the second-order optimistic method in Algorithm~\ref{alg:2nd_optimistic}}, we have 
\begin{equation}\label{eq:dist_to_opt_zeta}
    \|\vz_k-\vz^*\|^2 \leq \frac{2\|\vz_0-\vz^*\|^2}{2-\alpha}\ \zeta_k,
\end{equation}
where $\zeta_k$ is given in \eqref{eq:def_zeta}. To compare the empirical results with our theoretical findings, we consider a ``simulated'' sequence $\{\tilde{\zeta}_k\}_{k\geq 0}$ with $\tilde{\zeta}_0=1$ following the dynamic 
\vspace{-2mm}
\begin{equation}\label{eq:simulated_zeta}
\vspace{-2mm}
    \frac{1}{\tilde{\zeta}_k} = 1+ \frac{1}{C} \left(\sum_{l=0}^{k-1} \frac{1}{\tilde{\zeta}_l}\right)^{\frac{3}{2}}
\end{equation}
for any $k\geq 1$ for some $C>0$. In light of Lemma~\ref{lem:multiplying_factor}, we can see that $\zeta_k \leq \tilde{\zeta}_k$ for all $k\geq 0$ if the parameter $C$ is chosen as $\gamma_2 \kappa_2(\vz_0)$.

In our experiment, {the problem parameters are chosen as $n=200$, $L_2=10,000$, $\mu=10^{-3}$, and the line search parameters are the same as in the convex-concave setting.} In Fig.~\ref{fig:2nd_strong} we plot the distance to the saddle point $\vz^*$ of the second-order optimistic methods versus the number of iterations. For comparison, 
we also plot 
the bound in \eqref{eq:dist_to_opt_zeta} using the actual sequence $\{\zeta_k\}_{k\geq 0}$ as well as the simulated sequence $\{\tilde{\zeta}_k\}_{k\geq 0}$. 
In particular, our second-order method exhibits local superlinear convergence in the neighborhood of the saddle point, doubling the accuracy of the solution within a few iterations. 
Also, the bound in \eqref{eq:dist_to_opt_zeta} using $\{\zeta_k\}_{k\geq 0}$ provides a tight upper bound for the second-order method. Moreover, if the parameter $C$ in \eqref{eq:simulated_zeta} is carefully chosen, we observe that the bound in \eqref{eq:dist_to_opt_zeta} using $\{\tilde{\zeta}_k\}_{k\geq 0}$ has a qualitatively 
similar convergence behavior as the empirical result. 

\myalert{The complexity of line search.} We also empirically evaluate the number of subsolver calls made in our second-order optimistic method. 
we vary the line search parameters $\sigma_0$ and $\beta$ and run our method on Problem~\eqref{eq:test_prob_2} %
until it finds a solution with accuracy $10^{-10}$ or completes 500 iterations. The maximum is reported in Table~\ref{tab:complexity_2}. We can see that the line search complexity is insensitive to the choice of the initial trial step size $\sigma_0$, while a larger $\beta$ could lead to more calls to the subsolver in the strongly-convex-strongly-concave setting. Still, in our test instances, the average number of subsolver calls per iteration can be controlled around 2 when $\beta$ is chosen as 0.5, which is a reasonable price to pay given its superior convergence performance.

\begin{table}[tbhp]
    {\footnotesize
    \caption{The maximum average number of subsolver calls per iteration in the second-order optimistic method among 50 random instances.}\label{tab:complexity_2}
    \begin{center}    
    \begin{tabular}{|c|c|c|}
      \hline
      \multirow{2}{*}{ Parameters}       & \multicolumn{2}{c|}{Max average number of subsolver calls per iteration}                           \\ \cline{2-3} 
                                                    & {Convex-concave setting} & {Strongly-convex-strongly-concave setting} \\ \hline
      $\sigma_0=1,\beta = 0.5$                               & 1.9780                 & 2.0174         \\ \hline
      $\sigma_0=10, \beta =0.5$                            & 1.9860                  & 2.0492         \\ \hline
      $\sigma_0=100, \beta = 0.5$                         & 1.9920                  & 2.0964         \\ \hline
      $\sigma_0=1,\beta = 0.9$                               & 1.8580                & 2.1504         \\ \hline
      $\sigma_0=10, \beta =0.9$                             & 1.9020                 & 2.1681         \\ \hline
      $\sigma_0=100, \beta = 0.9$                        & 1.9440                 & 2.4609         \\ \hline
      \end{tabular}
    \end{center}
    }
    \vspace{-2em}
\end{table}

%% file: appendix.tex
\section{Proofs for preliminary lemmas}\label{appen:function_operator}

\begin{proof}[Proof of Lemma~{\ref{lem:function_operator}}]

We prove this lemma in a united way by regarding Assumption~\ref{assum:monotone} as a special case of Assumption~\ref{assum:strongly_monotone} with $\mu=0$. 
    Let $\vz=(\vx,\vy)$ and $\vz'=(\vx',\vy')$ be any two points in $\mathcal{X}\times \mathcal{Y}$. Since $f$ is $\mu$-strongly convex in $\vx$ w.r.t. $\Phi_{\mathcal{X}}$, we have $f(\vx',\vy) \geq f(\vx,\vy)+\langle \nabla_{\vx} f(\vx,\vy), \vx'-\vx\rangle + \mu D_{\Phi_{\mathcal{X}}}(\vx',\vx)$ and $f(\vx,\vy') \geq f(\vx',\vy')+\langle \nabla_{\vx} f(\vx',\vy'), \vx-\vx' \rangle + \mu D_{\Phi_{\mathcal{X}}}(\vx,\vx')$.
    Similarly, since $f$ is $\mu$-strongly concave in $\vy$ w.r.t.~$\Phi_{\mathcal{\mathcal{Y}}}$, we have $f(\vx,\vy) + \langle \nabla_{\vy} f(\vx,\vy), \vy'-\vy\rangle - \mu D_{\Phi_{\mathcal{Y}}}(\vy',\vy) \geq f(\vx,\vy')$ and $f(\vx',\vy')+\langle \nabla_{\vy} f(\vx',\vy'), \vy-\vy' \rangle - \mu D_{\Phi_{\mathcal{Y}}}(\vy,\vy') \geq f(\vx',\vy)$.   
    By summing all the inequalities above, we get $\langle F(\vz)-F(\vz'), \vz-\vz' \rangle \geq \mu(D_{\Phi_{\mathcal{Z}}}(\vz',\vz)+D_{\Phi_{\mathcal{Z}}}(\vz,\vz'))$,
    where we used the definition of $F$ in \eqref{eq:def_of_F_H} and the equality in \eqref{eq:mirror_map}. This proves that $F$ is $\mu$-strongly monotone w.r.t. $\Phi_{\mathcal{Z}}$ on $\mathcal{Z}$. 
\end{proof}

\begin{proof}[Proof of Lemma~{\ref{lem:duality_gap}}]%

 For any $\vz_k = (\vx_k,\vy_k)$, $\vz = (\vx,\vy)\in\mathcal{X}\times \mathcal{Y}$, by the definitions of $\ell$ and $h$ we have 
    \begin{align}
        \ell(\vx_k,\vy)-\ell(\vx,\vy_k) &= f(\vx_k,\vy)-f(\vx,\vy_k)+h_1(\vx_k)+h_2(\vy_k)-h_1(\vx)-h_2(\vy) \nonumber \\
        &= f(\vx_k,\vy)-f(\vx,\vy_k)+h(\vz_k)-h(\vz). \label{eq:gap}
    \end{align}
    Since $f$ is convex in $\vx$ and concave in $\vy$, by Assumption~\ref{assum:monotone} we have $f(\vx_k,\vy_k) - f(\vx,\vy_k) \leq \langle \nabla f_{\vx}(\vx_k,\vy_k),\vx_k-\vx\rangle$ and $f(\vx_k,\vy) - f(\vx_k,\vy_k) \leq \langle -\nabla f_{\vy}(\vx_k,\vy_k),\vy_k-\vy\rangle$. 
    Adding these two inequalities gives us 
    \begin{equation}\label{eq:linearized}
        f(\vx_k,\vy) - f(\vx,\vy_k) \leq \langle F(\vz_k), \vz_k-\vz \rangle,
    \end{equation}
    where we used the definition of $F$ in \eqref{eq:def_of_F_H}. Combining \eqref{eq:gap} and \eqref{eq:linearized} leads to 
    \begin{equation*}
        \sum_{k=1}^N \theta_k \ell(\vx_k,\vy) - \sum_{k=1}^N \theta_k \ell(\vx,\vy_k) \leq 
        \sum_{k=1}^N \theta_k
        ( \langle F(\vz_k), \vz_k-\vz \rangle
        + h(\vz_k)-h(\vz)).
    \end{equation*}
    Finally, since $\ell$ is convex in $\vx$ and concave in $\vy$,  we have {$\sum_{k=1}^N \theta_k \ell(\vx_k,\vy) \geq \ell(\bar{\vx}_N,\vy)$ and $\sum_{k=1}^N \theta_k \ell(\vx,\vy_k) \leq \ell(\vx,\bar{\vy}_N)$ by Jensen's inequality,}
    from which the result follows.
\end{proof}

\begin{proof}[Proof of Lemma~\ref{lem:saddle_point}]\label{sec:proof_of_lem:saddle_point}
    
    Since $\vz^*$ is a saddle point of Problem~\eqref{eq:minimax}, it also solves the variational inequality in \eqref{eq:VI}. Hence, for any $\vz\in \mathcal{Z}$, we have $\langle F(\vz^*), \vz-\vz^* \rangle +h(\vz)-h(\vz^*) \geq 0$. 
    Moreover, since $F$ is $\mu$-strongly monotone w.r.t.\@ $\Phi$, we have $\langle F(\vz)-F(\vz^*), \vz-\vz^* \rangle \geq \mu D_{\Phi}(\vz,\vz^*) +\mu D_{\Phi}(\vz^*,\vz) \geq \mu D_{\Phi}(\vz^*,\vz)$. 
    Adding these two inequalities gives us the desired result. 
\end{proof}

\section{Proof of \texorpdfstring{Proposition~\ref{prop:BPP}}{Proposition 3.1}}
\label{sec:proof_of_BPP}
The update rule for the Bregman PPM in \eqref{eq:BPP_inclusion} implies that 
$
    \eta_k F(\vz_{k+1})+\nabla \Phi(\vz_{k+1})-\nabla \Phi(\vz_k)
    \in -\eta_k H(\vz_{k+1})
$. 
Hence, by using the definitions of $H$ in \eqref{eq:def_of_F_H},  we have 
$\eta_k \langle F(\vz_{k+1}),\vz_{k+1}-\vz\rangle 
+\langle \nabla \Phi(\vz_{k+1})-\nabla \Phi(\vz_k), \vz_{k+1}-\vz \rangle
\leq \eta_k (h(\vz)-h(\vz_{k+1}))$ for any $\vz\in \mathcal{Z}$. 
We apply the three-point identity \blue{of Bregman distance~\cite{Chen1993}} and rearrange the terms to get 
\begin{align}
    & \phantom{{}\leq {}}\eta_k (\langle F(\vz_{k+1}),\vz_{k+1}-\vz\rangle+h(\vz_{k+1})-h(\vz)) \nonumber \\ &\leq 
    D_{\Phi}(\vz,\vz_k)-D_{\Phi}(\vz,\vz_{k+1})-D_{\Phi}(\vz_{k+1},\vz_k)
    \leq D_{\Phi}(\vz,\vz_k)-D_{\Phi}(\vz,\vz_{k+1}).\label{eq:key_relation_BPP}
\end{align}
\myalert{Proof of Part~\ref{item:BPP_a}.}
In this part, we assume that Assumption~\ref{assum:monotone} holds. We first set $\vz=\vz^*$ in \eqref{eq:key_relation_BPP}. By invoking Lemma~\ref{lem:saddle_point} with $\vz=\vz_{k+1}$ and $\mu=0$, we get $D_{\Phi}(\vz^*,\vz_{k+1}) \leq D_{\Phi}(\vz^*,\vz_k)$ for any $k\geq 0$. 
Next, we sum both sides of \eqref{eq:key_relation_BPP} over $k=0,1,\ldots,N-1$ to get 
\begin{equation}\label{eq:BPP_unnormalized}
    \sum_{k=0}^{N-1} \eta_k (\langle F(\vz_{k+1}),\vz_{k+1}-\vz\rangle+h(\vz_{k+1})-h(\vz))
    \leq D_{\Phi}(\vz,\vz_0)-D_{\Phi}(\vz,\vz_N) \leq D_{\Phi}(\vz,\vz_0).
\end{equation}
The result follows by dividing both sides of \eqref{eq:BPP_unnormalized} by $\sum_{k=0}^{N-1} \eta_k$ and applying Lemma~\ref{lem:duality_gap}. 

\myalert{Proof of Part~\ref{item:BPP_b}.}
In this part, we assume that Assumption~\ref{assum:strongly_monotone} holds. We set $\vz=\vz^*$ in \eqref{eq:key_relation_BPP} and apply Lemma~\ref{lem:saddle_point} with $\vz=\vz_{k+1}$ to get $(1+\mu\eta_k) D_{\Phi}(\vz^*,\vz_{k+1}) \leq D_{\Phi}(\vz^*,\vz_k)$ for any $k \geq 0$. 
The result now follows by induction on the iteration counter $k$. 

\section{Proofs for generalized optimistic method}
Our proof relies on a carefully designed Lyapunov function. We define the function and discuss its properties in the following lemma, which is the cornerstone of proving Proposition~\ref{prop:GOMD}. 
\begin{lemma}\label{lem:Lyapunov_function}
    Let $\{\vz_k\}_{k\geq 0}$ be the iterates generated by Algorithm~\ref{alg:GOMD} with $\mu\geq 0$, 
    $0<\alpha\leq 1$, and $\hat{\eta}_k$ chosen as in Proposition~\ref{prop:GOMD}. Define a Lyapunov function 
    \begin{equation}\label{eq:def_Lyapunov}
        V(\vz_k,\vz_{k-1};\vz) =  -\frac{\eta_{k-1}}{1+\eta_{k-1}\mu}\langle F(\vz_{k})-P(\vz_{k};\mathcal{I}_{k-1}),\vz_{k}-\vz \rangle+D_{\Phi}(\vz,\vz_k)%
        +\frac{\alpha\|\vz_{k}-\vz_{k-1}\|^2}{4(1+\eta_{k-1}\mu)^2},
    \end{equation}
    where $\vz\in \mathcal{Z}$ is an arbitrary point. Then for any $\vz\in \mathcal{Z}$ and $k\geq 0$, we have:%
    \begin{enumerate}[label=(\alph*),leftmargin=0.7cm]
        \item \label{item:bound_on_Lyapunov} %
        $V(\vz_{0},\vz_{-1};\vz)=D_{\Phi}(\vz,\vz_0)$ and $V(\vz_k,\vz_{k-1};\vz)\geq \frac{2-\alpha}{2}D_{\Phi}(\vz,\vz_k)$;
        \item \label{item:Lyapunov_relation}
        $\eta_k(\langle F(\vz_{k+1}),\vz_{k+1}-\vz\rangle+h(\vz_{k+1})-h(\vz)-\mu D_{\Phi}(\vz,\vz_{k+1})) \!\leq\! V(\vz_k,\vz_{k-1};\vz)-(1+\eta_k\mu)V(\vz_{k+1},\vz_{k};\vz)-\frac{1-\alpha}{2}\|\vz_{k+1}-\vz_k\|^2.$ %
    \end{enumerate}  
\end{lemma}
\begin{proof}
    To simplify the notation, we define $\mathcal{E}_k := F(\vz_{k})-P(\vz_{k};\mathcal{I}_{k-1})$ and thus Condition \eqref{eq:stepsize_upper_bound} implies that $\eta_k\|\mathcal{E}_k\| \leq \frac{\alpha}{2}\|\vz_{k+1}-\vz_k\|$. 
    For Part~\ref{item:bound_on_Lyapunov}, note that we initialize Algorithm~\ref{alg:GOMD} with $\vz_{-1}=\vz_0$ and 
    $P(\vz_0;\mathcal{I}_{-1})=F(\vz_0)$. Hence, it is straightforward to verify that $V(\vz_{0},\vz_{-1};\vz)=D_{\Phi}(\vz,\vz_0)$. 
    Moreover, we can lower bound $V(\vz_k,\vz_{k-1};\vz)$ by 
    \begin{align}
        V(\vz_k,\vz_{k-1};\vz) &\geq -\frac{\eta_{k-1}}{1+\eta_{k-1}\mu}\|\mathcal{E}_k\|_*\|\vz_{k}-\vz\|+D_{\Phi}(\vz,\vz_k)
        +\frac{\alpha\|\vz_{k}-\vz_{k-1}\|^2}{4(1+\eta_{k-1}\mu)^2} \label{eq:line1_CS}\\
        &\geq -\frac{\alpha}{2(1+\eta_{k-1}\mu)}\|\vz_{k}-\vz_{k-1}\|\|\vz_{k}-\vz\|+D_{\Phi}(\vz,\vz_k)
        +\frac{\alpha\|\vz_{k}-\vz_{k-1}\|^2}{4(1+\eta_{k-1}\mu)^2}\label{eq:line2_stepsize}\\
        &\geq -\frac{\alpha}{4}\|\vz_{k}-\vz\|^2+D_{\Phi}(\vz,\vz_k)\label{eq:line3_young}\\
        &\geq \frac{2-\alpha}{2}D_{\Phi}(\vz,\vz_k),\label{eq:line4_sc}
    \end{align}
    where we used the generalized Cauchy-Schwarz inequality in \eqref{eq:line1_CS}, Condition \eqref{eq:stepsize_upper_bound}
    in \eqref{eq:line2_stepsize}, Young's inequality\footnote{In this paper, it refers to the elementary inequality that $ab\leq \frac{1}{2}a^2+\frac{1}{2}b^2$ for any $a,b\in \mathbb{R}$.} in \eqref{eq:line3_young}, 
    and the strong convexity of $\Phi$ in \eqref{eq:line4_sc}. This completes the proof for Part~\ref{item:bound_on_Lyapunov}. 
    
    For Part~\ref{item:Lyapunov_relation}, by the definition of $H$ in \eqref{eq:def_of_F_H}, the update rule \eqref{eq:GOM_inclusion} implies that 
    \(%
        \langle \eta_k P(\vz_{k+1};\mathcal{I}_k)+\hat{\eta}_k\mathcal{E}_k+\nabla\Phi(\vz_{k+1})-\nabla\Phi(\vz_k),\vz_{k+1}-\vz\rangle 
        \leq \eta_k(h(\vz)-h(\vz_{k+1}))  
    \)
    for any $\vz\in \mathcal{Z}$. Moreover, note that the step size $\hat{\eta}_k$ in Proposition~\ref{prop:GOMD} can be written as $\hat{\eta}_k=\eta_{k-1}/(1+\eta_{k-1}\mu)$ for $\mu\geq 0$. 
    Hence, by applying the three-point identity \blue{of Bregman distance~\cite{Chen1993}}, using $\mathcal{E}_{k+1} = F(\vz_{k+1})-P(\vz_{k+1};\mathcal{I}_{k})$, and rearranging the terms, we obtain 
        \begin{align}
            &\phantom{\leq}\;\;\eta_k(\langle F(\vz_{k+1}),\vz_{k+1}-\vz\rangle+h(\vz_{k+1})-h(\vz)-\mu D_{\Phi}(\vz,\vz_{k+1})) \nonumber\\ 
            &\leq \eta_k\langle F(\vz_{k+1})-P(\vz_{k+1};\mathcal{I}_k),\vz_{k+1}-\vz \rangle
            -\frac{\eta_{k-1}}{1+\eta_{k-1}\mu}\langle \mathcal{E}_{k},\vz_{k+1}-\vz \rangle %
            -\eta_k\mu D_{\Phi}(\vz,\vz_{k+1}) \nonumber\\ 
            &\phantom{{}\leq{}}
            -D_{\Phi}(\vz_{k+1},\vz_k)-D_{\Phi}(\vz,\vz_{k+1})+D_{\Phi}(\vz,\vz_k) \nonumber\\ 
            &= (1+\eta_k\mu)\Bigl[\frac{\eta_k\langle \mathcal{E}_{k+1},\vz_{k+1}-\vz \rangle}{1+\eta_k\mu}\!-\!D_{\Phi}(\vz,\vz_{k+1})\Bigr]
            \!-\!\Bigl[\frac{\eta_{k-1} \langle \mathcal{E}_{k},\vz_{k}-\vz \rangle}{1+\eta_{k-1}\mu}\!-\!D_{\Phi}(\vz,\vz_k)\Bigr] \nonumber\\ 
            &\phantom{{}\leq{}} +\frac{\eta_{k-1}}{1+\eta_{k-1}\mu}\langle \mathcal{E}_k,\vz_{k}-\vz_{k+1} \rangle-D_{\Phi}(\vz_{k+1},\vz_k). \label{eq:GOMD_key}
        \end{align}
        The first two bracketed terms in \eqref{eq:GOMD_key} resemble the Lyapunov function defined in \eqref{eq:def_Lyapunov}, and 
        now we upper bound the remaining terms. By the generalized Cauchy-Schwarz inequality, the condition in \eqref{eq:stepsize_upper_bound} and Young's inequality, we have 
        \begin{align*}
            \frac{\eta_{k-1}\langle \mathcal{E}_k,\vz_{k}-\vz_{k+1} \rangle }{1+\eta_{k-1}\mu}
            \leq \frac{\eta_{k-1}\|\mathcal{E}_k\|_{*}\|\vz_{k+1}-\vz_{k}\|}{1+\eta_{k-1}\mu}&\leq \frac{\alpha \|\vz_{k}-\vz_{k-1}\|\|\vz_{k+1}-\vz_{k}\|}{2(1+\eta_{k-1}\mu)}\\
            &\leq \frac{\alpha\|\vz_{k}-\vz_{k-1}\|^2}{4(1+\eta_{k-1}\mu)^2}+\frac{\alpha}{4}\|\vz_{k+1}-\vz_{k}\|^2 %
        \end{align*}
        Thus, by using the strong convexity of $\Phi$, we have $\frac{\eta_{k-1}}{1+\eta_{k-1}\mu}\langle \mathcal{E}_k,\vz_{k}-\vz_{k+1} \rangle-D_{\Phi}(\vz_{k+1},\vz_k) \leq \frac{\alpha\|\vz_{k}-\vz_{k-1}\|^2}{4(1+\eta_{k-1}\mu)^2}+\frac{\alpha}{4}\|\vz_{k+1}-\vz_{k}\|^2 - \frac{1}{2}\|\vz_{k+1}-\vz_k\|^2 \leq \frac{\alpha\|\vz_{k}-\vz_{k-1}\|^2}{4(1+2\eta_{k-1}\mu)^2}-\frac{\alpha\|\vz_{k+1}-\vz_{k}\|^2}{4(1+2\eta_{k}\mu)}-\frac{1-\alpha}{2}\|\vz_{k+1}-\vz_{k}\|^2$. Then  
        Part~\ref{item:Lyapunov_relation}
        follows directly from combining \eqref{eq:GOMD_key} and the above inequality. 
\end{proof}
\subsection{Proof of \texorpdfstring{Proposition~\ref{prop:GOMD}}{Proposition 4.1}}
\label{appen:proof_of_GOMD}

\myalert{Proof of Part~\ref{item:GOMD_monotone}.}
In this part, we assume that Assumption~\ref{assum:monotone} holds and hence $\mu=0$. By Lemma~\ref{lem:saddle_point}, for any $k\geq 0$ we have $0 \leq \eta_k(\langle F(\vz_{k+1}),\vz_{k+1}-\vz^*\rangle+h(\vz_{k+1})-h(\vz^*))$. Moreover, by using Lemma~\ref{lem:Lyapunov_function}(b) with $\vz=\vz^*$, we further have $ 0 \leq \eta_k(\langle F(\vz_{k+1}),\vz_{k+1}-\vz^*\rangle+h(\vz_{k+1})-h(\vz^*)) \leq V(\vz_k,\vz_{k-1};\vz^*)-V(\vz_{k+1},\vz_{k};\vz^*)-\frac{1-\alpha}{2}\|\vz_{k+1}-\vz_k\|^2 \leq V(\vz_k,\vz_{k-1};\vz^*)-V(\vz_{k+1},\vz_{k};\vz^*)$. 
\blue{Thus, the Lyapunov function $V(\vz_k,\vz_{k-1};\vz)$ is nonincreasing w.r.t. $k$}, which implies that $V(\vz_{N},\vz_{N-1};\vz^*)\leq V(\vz_0,\vz_{-1};\vz^*)$. Then by Lemma~\ref{lem:Lyapunov_function}(a), we get $ \frac{2-\alpha}{2} D_{\Phi}(\vz^*,\vz_N)\leq V(\vz_N,\vz_{N-1};\vz^*)\leq V(\vz_0,\vz_{-1};\vz^*)=D_{\Phi}(\vz^*,\vz_0)$,  
which proves \eqref{eq:bounded_iterates}. 

Next, note that $\frac{1-\alpha}{2}\|\vz_{k+1}-\vz_k\|^2\geq 0$ and hence from Lemma~\ref{lem:Lyapunov_function}(b) we get 
\begin{equation*}
        \eta_k(\langle F(\vz_{k+1}),\vz_{k+1}-\vz\rangle+h(\vz_{k+1})-h(\vz))
        \leq V(\vz_k,\vz_{k-1};\vz)-V(\vz_{k+1},\vz_{k};\vz)
\end{equation*}
for any $\vz\in \mathcal{Z}$ and $k\geq 0$. Summing the above inequality over $k=0,1,\ldots,N-1$ leads to 
\begin{align*}
    \sum_{k=0}^{N-1}\eta_k(\langle F(\vz_{k+1}),\vz_{k+1}-\vz\rangle+h(\vz_{k+1})-h(\vz))
    &\leq V(\vz_0,\vz_{-1};\vz)-V(\vz_{N},\vz_{N-1};\vz) \\ 
    &\leq D_{\Phi}(\vz,\vz_0)-\frac{2-\alpha}{2}D_{\Phi}(\vz,\vz_{N-1}) \leq D_{\Phi}(\vz,\vz_0),
\end{align*}
where we used Lemma~\ref{lem:Lyapunov_function}(a) in the second inequality. Now Part~\ref{item:GOMD_monotone} follows by similar arguments as in the proof of Proposition~\ref{prop:BPP}(a).

\myalert{Proof of Part~\ref{item:GOMD_strongly_monotone}.}
In this part, we assume that Assumption~\ref{assum:strongly_monotone} holds. We invoke Lemma~\ref{lem:Lyapunov_function}(b) with $\vz=\vz^*$ and Lemma~\ref{lem:saddle_point} to get
\(    0 \leq \eta_k(\langle F(\vz_{k+1}),\vz_{k+1}-\vz^*\rangle+h(\vz_{k+1})-h(\vz^*)-\mu D_{\Phi}(\vz^*,\vz_{k+1})) 
    \leq V(\vz_k,\vz_{k-1};\vz^*)-(1+\eta_k\mu)V(\vz_{k+1},\vz_{k};\vz^*) \).
This implies that $V(\vz_{k+1},\vz_{k};\vz^*) \leq V(\vz_{k},\vz_{k-1};\vz^*) (1+\eta_k\mu)^{-1}$ for any $k \geq 0$, and hence
by induction we have $V(\vz_N,\vz_{N-1};\vz^*) \leq V(\vz_0,\vz_{-1};\vz^*) \prod_{k=0}^{N-1} (1+\eta_k\mu)^{-1}$.
Finally, the claim follows from the fact that $V(\vz_{0},\vz_{-1};\vz)=D_{\Phi}(\vz,\vz_0)$ and $V(\vz_k,\vz_{k-1};\vz)\geq \frac{2-\alpha}{2}D_{\Phi}(\vz,\vz_k)$ (cf. Lemma~\ref{lem:Lyapunov_function}(a)).

\subsection{Proof of \texorpdfstring{Lemma~\ref{lem:sum_of_square}}{Lemma 4.2}}
\label{appen:proof_of_SOS}
Since we can regard the convex-concave setting as a special case of the strongly-convex-strongly-concave setting with $\mu=0$, we will prove Lemma~\ref{lem:sum_of_square} under both assumptions in a united way. To simplify the notation, we define the sequence $\{\Gamma_k\}_{k \geq 0}$ by $\Gamma_0 = 1$ and $\Gamma_k = \prod_{l=0}^{k-1} \left(1+{\eta_l}\mu\right)$ for $k\geq 1$. 

We first apply Lemma~\ref{lem:Lyapunov_function}(b) with $\vz=\vz^*$. Together with Lemma~\ref{lem:saddle_point}, 
it implies that 
\begin{equation}\label{eq:desired_bound_rk}
    \frac{1-\alpha}{2}\|\vz_{k+1}-\vz_k\|^2 \leq V(\vz_k,\vz_{k-1};\vz^*)-(1+\eta_k\mu)V(\vz_{k+1},\vz_{k};\vz^*).
\end{equation}
We then multiply both sides of the inequality by $\Gamma_k$ and note that $\Gamma_k (1+\eta_k\mu) = \Gamma_{k+1}$. Summing over $k=0,1,\ldots,N-1$, we obtain that
\begin{equation*}
\begin{aligned}
\frac{1-\alpha}{2}\sum_{k=0}^{N-1}(\|\vz_{k+1}-\vz_{k}\|^2 \Gamma_k) &\leq
\sum_{k=0}^{N-1} (V(\vz_k,\vz_{k-1};\vz^*)\Gamma_k-V(\vz_{k+1},\vz_{k};\vz^*)\Gamma_{k+1}) \\
&= V(\vz_{0},\vz_{-1};\vz^*)-V(\vz_{N},\vz_{N-1};\vz^*)\Gamma_N.
\end{aligned}
\end{equation*}
Then Lemma~\ref{lem:sum_of_square} follows from the facts that
$V(\vz_{0},\vz_{-1};\vz^*)=D_{\Phi}(\vz^*,\vz_0)$ and $V(\vz_{N},\vz_{N-1};\vz^*)\geq 0$ 
(cf. Lemma~\ref{lem:Lyapunov_function}(a)).

\section{Proofs for the second-order optimistic method}

\subsection{Maximal monotonicity of the inclusion problem}\label{appen:maximal}%
\blue{
We first recap some useful properties of monotone operators; see, e.g., \cite{Facchinei2003,Ryu2022}. 
\begin{proposition}\label{prop:facts_on_monotone}
    Let $T:\;\mathbb{R}^d \rightarrow \mathbb{R}^d$ be a single-valued operator.  
    \begin{enumerate}[label=(\alph*)]
        \item A differentiable operator $T$ is monotone if and only if 
        $\langle DT(\vz)\vw,\vw \rangle \geq 0$ for all $\vw\in \mathbb{R}^d$, $\vz\in \dom T$, where 
        $DT(\vz)$ is the Jacobian matrix of $T$ at $\vz$. 
        \item If $T$ is continuous, monotone, and defined on the whole space $\mathbb{R}^d$, then 
        $T$ is maximal monotone. 
    \end{enumerate}
\end{proposition}
Now we are ready to prove that the inclusion problem in \eqref{eq:OMD_2nd_ls} is maximal monotone. We break the proof into two steps: 
\begin{itemize}
    \item To begin with, we show that $T^{(1)}(\vz;\vz_k)=F(\vz_k)+DF(\vz_k)(\vz-\vz_k)$ is maximal monotone in $\vz$. Since $F$ is monotone, it follows from Proposition~\ref{prop:facts_on_monotone}(a) that 
    $
        \langle T^{(1)}(\vz;\vz_k)-T^{(1)}(\vz';\vz_k), \vz-\vz' \rangle = \langle DF(\vz_k)(\vz-\vz'), \vz-\vz' \rangle \geq 0
    $ for any $\vz,\vz'\in \mathbb{R}^{m+n}$. By definition, this implies that $T^{(1)}(\vz;\vz_k)$ is monotone in $\vz$. Furthermore, 
    since $T^{(1)}(\cdot;\vz_k)$ is continuous and defined on the whole space $\mathbb{R}^{m+n}$, by Proposition~\ref{prop:facts_on_monotone}(b) $T^{(1)}(\vz;\vz_k)$ is also maximal monotone.
    \item Moreover, since $H$ and $\nabla \Phi$ are the subdifferentials of proper closed convex functions, both are maximal monotone operators. Thus, we conclude that the inclusion problem in \eqref{eq:OMD_2nd_ls} is maximal monotone. 
\end{itemize}
}

\subsection{Proof of \texorpdfstring{Lemma~\ref{lem:beta_optimal_point_2nd}}{Lemma 6.1}}
\label{appen:beta_optimal_2nd}
We will need the following lemma adapted from \cite[Lemma~4.3]{Monteiro2012}.
\begin{lemma}\label{lem:resolvent}
    Suppose $A$ is a maximal monotone operator and $\Phi$ is $1$-strongly convex and $L_{\Phi}$-smooth w.r.t. $\|\cdot\|$ on~$\mathcal{Z}$.  
    \blue{Let $\vz(\eta;\vw)$ denote the unique solution of the monotone inclusion problem $0 \in \eta A(\vz) + \nabla \Phi(\vz)-\vw.$} 
    If $0<\eta<\eta'$, then 
    \begin{equation}\label{eq:eta_eta_prime}
        \blue{\|\nabla \Phi(\vz(\eta';\vw))-\vw\|_* \leq \frac{\eta'}{\eta} \sqrt{L_{\Phi}}\|\nabla \Phi(\vz(\eta;\vw))-\vw\|_*}.
    \end{equation}
\end{lemma}
\begin{proof}
    \blue{To simplify the notation, let $\vz:=\vz(\eta;\vw)$, $\vu:=(\vw-\nabla \Phi(\vz))/\eta$, $\vz':=\vz(\eta';\vw)$ and $\vu':=(\vw-\nabla \Phi(\vz'))/\eta'$.} 
    In these notations, our goal in \eqref{eq:eta_eta_prime} is equivalent to $\|\vu'\|_*\leq \sqrt{L_{\Phi}}\|\vu\|_*$.
    
    By definition, we can write 
    \begin{equation}\label{eq:inclusion_A_inverse}
        \vu=\frac{\vw-\nabla \Phi(\vz)}{\eta} \quad \Leftrightarrow \quad \vz = \nabla\Phi^*(\vw-\eta\vu) \qquad \text{and} \qquad 0 \in \eta A(\vz) + \nabla \Phi(\vz)-\vw \quad \Leftrightarrow \quad \vz \in A^{-1}(\vu),
    \end{equation}
    where $\nabla \Phi^*$ denotes the Fenchel conjugate of $\Phi$. 
    Moreover, by the assumptions on $\Phi$, we note that $\Phi^*$ is $1/L_{\Phi}$-strongly convex and $1$-smooth w.r.t. $\|\cdot\|_{*}$ (see \cite[Theorem 5.26]{Beck2017}). 
    From \eqref{eq:inclusion_A_inverse} we get $\nabla\Phi^*(\vw-\eta\vu) \in A^{-1}(\vu)$, and similarly 
    $
        \nabla\Phi^*(\vw-\eta'\vu')\in A^{-1}(\vu')
    $. 
    Since $A$ is maximal monotone, so is the operator $A^{-1}$. Thus, we get 
\begin{align}
  \langle\nabla \Phi^{*}(\vw-\eta \vu)-\nabla \Phi^{*}(\vw-\eta' \vu'), \vu-\vu'\rangle &\geq 0 \nonumber \\
  \Leftrightarrow \quad \langle \nabla \Phi^{*}(\vw-\eta \vu')-\nabla \Phi^{*}(\vw-\eta' \vu'),\vu-\vu'\rangle &\geq \langle \nabla \Phi^{*}(\vw-\eta \vu')-\nabla \Phi^{*}(\vw-\eta \vu),\vu-\vu'\rangle. \label{eq:monotone_A_inverse}
\end{align}
For the right-hand side of \eqref{eq:monotone_A_inverse}, we have from the convexity of $\Phi^*$ that   
\begin{equation}
  \langle \nabla \Phi^{*}(\vw-\eta \vu')-\nabla \Phi^{*}(\vw-\eta \vu),\vu-\vu'\rangle
   = \frac{1}{\eta}\langle \nabla \Phi^{*}(\vw-\eta \vu')-\nabla \Phi^{*}(\vw-\eta \vu),\vw-\eta \vu'-(\vw-\eta \vu)\rangle \geq 0. \label{eq:RHS_monotone}
\end{equation}
For the left-hand side of \eqref{eq:monotone_A_inverse}, we can use the three-point identity \blue{of Bregman distance~\cite{Chen1993}} to get 
\begin{align}
    &\langle \nabla \Phi^{*}(\vw-\eta \vu')-\nabla \Phi^{*}(\vw-\eta' \vu'),\vu-\vu'\rangle \nonumber\\
    =&\;\frac{1}{\eta'-\eta}\langle \nabla \Phi^{*}(\vw-\eta \vu')-\nabla \Phi^{*}(\vw-\eta' \vu'),[\vw-\eta'\vu'+(\eta'-\eta)\vu]-(\vw-\eta\vu')\rangle \nonumber\\
    =& -\frac{1}{\eta'-\eta} [D_{\Phi^*}(\vw-\eta \vu',\vw-\eta' \vu')+D_{\Phi^*}(\vw-\eta'\vu'+(\eta'-\eta)\vu,\vw-\eta \vu') \nonumber\\
    &\qquad \qquad -D_{\Phi^*}(\vw-\eta'\vu'+(\eta'-\eta)\vu,\vw-\eta' \vu')]. \label{eq:LHS_monotone}
\end{align}
Since $\eta'>\eta$, \eqref{eq:monotone_A_inverse}, \eqref{eq:RHS_monotone}, and \eqref{eq:LHS_monotone} together imply that 
\begin{equation}\label{eq:D_inequality}
  D_{\Phi^*}(\vw-\eta \vu',\vw-\eta' \vu') \leq D_{\Phi^*}(\vw-\eta'\vu'+(\eta'-\eta)\vu,\vw-\eta' \vu').
\end{equation}
Furthermore, since $\Phi^*$ is $1/L_{\Phi}$-strongly convex and $1$-smooth w.r.t. $\|\cdot\|_{*}$, we have
\begin{equation}\label{eq:phi_conjugate}
    D_{\Phi^*}(\vw-\eta \vu',\vw-\eta' \vu') \geq \frac{1}{2L_{\Phi}}\|(\eta'-\eta)\vu'\|_*^2 \quad \text{and} \quad
    D_{\Phi^*}(\vw-\eta'\vu'+(\eta'-\eta)\vu,\vw-\eta' \vu') \leq \frac{1}{2}\|(\eta'-\eta)\vu\|_*^2.
\end{equation} 
Combining \eqref{eq:D_inequality} and \eqref{eq:phi_conjugate} gives us $\|\vu'\|_*\leq \sqrt{L_{\Phi}}\|\vu\|_*$. The proof is complete.
\end{proof}

As a corollary of Lemma~\ref{lem:resolvent}, we have the following result. 
\begin{lemma}\label{lem:eta_eta_prime}
    Suppose that $0<\eta<\eta'$. Let $\vz = \vz(\eta;\vz^{-})$ and $\vz'=\vz(\eta';\vz^{-})$ be the solution of \eqref{eq:simplified_subproblem} with step size $\eta$ and $\eta'$, respectively. Then we have $\|\vz'-\vz^{-}\| \leq 
    \frac{L_{\Phi}^{3/2}\eta'}{\eta} \|\vz-\vz^{-}\| + \left(1+\frac{\sqrt{L_{\Phi}}\eta'}{\eta}\right)\|\vv^{-}\|_*$.
\end{lemma}
\begin{proof}
    First note that 
    the inclusion problem in \eqref{eq:simplified_subproblem} can be written as $0\in \eta A(\vz)+\nabla \Phi(\vz)-(\nabla \Phi(\vz^{-})-\vv^{-})$, 
    where $A=P+H$ is a maximal monotone 
    operator. Hence, by Lemma~\ref{lem:resolvent} we have $\|\nabla \Phi(\vz')-\nabla \Phi(\vz^{-})+\vv^{-}\|_* \leq \frac{\eta'}{\eta}\sqrt{L_\Phi}\|\nabla \Phi(\vz)-\nabla \Phi(\vz^{-})+\vv^{-}\|_*$. 
    Thus, we can further use the triangle inequality to bound $\|\nabla \Phi (\vz')-\nabla \Phi (\vz^{-})\|_* \leq \|\nabla \Phi(\vz')-\nabla \Phi(\vz^{-})+\vv^{-}\|_*+\|\vv^{-}\|_* \leq \frac{\eta'}{\eta}\sqrt{L_{\Phi}}\|\nabla \Phi(\vz)-\nabla \Phi(\vz^{-})+\vv^{-}\|_*+\|\vv^{-}\|_* \leq \frac{\eta'}{\eta}\sqrt{L_{\Phi}}\|\nabla \Phi(\vz)-\nabla \Phi(\vz^{-})\|_*+\Bigl(1+\frac{\eta'}{\eta}\sqrt{L_{\Phi}}\Bigr)\|\vv^{-}\|_*$. 
    Moreover, since $\Phi$ is 1-strongly convex and $L_{\Phi}$-smooth w.r.t. $\|\cdot\|$, we have 
    \begin{align*}
        \|\vz'-\vz^{-}\| \leq \|\nabla \Phi (\vz')-\nabla \Phi (\vz^{-})\|_* &\leq \frac{\eta'}{\eta}\sqrt{L_{\Phi}}\|\nabla \Phi(\vz)-\nabla \Phi(\vz^{-})\|_*+\Bigl(1+\frac{\eta'}{\eta}\sqrt{L_{\Phi}}\Bigr)\|\vv^{-}\|_* \\
        &\leq \frac{\eta'}{\eta}L_{\Phi}^{{3}/{2}}\|\vz-\vz^{-}\|+\Bigl(1+\frac{\eta'}{\eta}\sqrt{L_{\Phi}}\Bigr)\|\vv^{-}\|_* .
    \end{align*}
    This completes the proof.  
\end{proof}

Now we are ready to prove Lemma~\ref{lem:beta_optimal_point_2nd}. To simplify the notation, we drop the subscript $k$ and denote $\vz_{k}$ by $\vz^{-}$. 
Since $k\in \mathcal{B}$, by Algorithm~\ref{alg:ls}, the condition in \eqref{eq:simplified_step_bound} is violated for $\eta' = \eta/\beta$ and $\vz' = \vz(\eta'; \vz^-)$ (Otherwise, the line search scheme would choose $\eta'$ instead of $\eta$).
Hence, we have {$\frac{1}{2}\alpha\|\vz'-\vz^{-}\| 
<\eta' \|F(\vz')-T^{(1)}(\vz';\vz^{-})\|_*
\leq \frac{\eta' L_2}{2} \|\vz'-\vz^{-}\|^2$,}   
where we used \eqref{eq:2nd_smooth} in the last inequality. This implies that 
$\eta'>\frac{\alpha}{L_2\|\vz'-\vz^{-}\|}$, and we further have $\eta \geq \beta\eta'>\frac{\alpha\beta}{L_2\|\vz'-\vz^{-}\|}$. Together with Lemma~\ref{lem:eta_eta_prime} and the fact that $\eta'/\eta \leq 1/\beta$, this leads to Lemma~\ref{lem:beta_optimal_point_2nd}.

\section{Proofs for the higher-order optimistic method}\label{appen:higher_order}
\subsection{Proof of \texorpdfstring{Lemma~\ref{lem:property_high}}{Lemma 7.1}}
\label{appen:property_high}

\myalert{Proof of Part~\ref{item:approx_high}.}
By using the triangle inequality and \eqref{eq:residual_F}, we have 
\begin{align}
    \| F(\vz')-T^{(p-1)}_{\lambda}(\vz';\vz)\|_* &\leq \|F(\vz')-T^{(p-1)}(\vz';\vz)\|_*+\frac{\lambda}{(p-1)!} (2D_{\Phi}(\vz',\vz))^{\frac{p-1}{2}}\|\nabla \Phi(\vz')-\nabla \Phi(\vz)\|_{*} \nonumber\\
    &\leq \frac{L_{p}}{p!}\|\vz'-\vz\|^p+\frac{\lambda (L_{\Phi}\|\vz'-\vz\|^2)^{\frac{p-1}{2}}L_{\Phi}\|\vz'-\vz\|}{(p-1)!}  %
    = \frac{L_p+p\lambda L_{\Phi}^{\frac{p+1}{2}}}{p!}\|\vz'-\vz\|^p, \nonumber
\end{align}
where we used the fact that $\Phi(\vz)$ is $L_{\Phi}$-smooth in the second inequality. 
This proves Part~\ref{item:approx_high}. 

\myalert{Proof of Part~\ref{item:maximal_monotone_high}.} Note that $T^{(p-1)}_{\lambda}(\cdot;\vz)$ is single-valued and continuous with full domain. Hence, it suffices to show that it is monotone as the maximality follows from Proposition~\ref{prop:facts_on_monotone}(b). %
By definition, we only need to verify that 
\begin{equation}\label{eq:regularized_monotone}
    \langle T^{(p-1)}_{\lambda}(\vz_2;\vz) - T^{(p-1)}_{\lambda}(\vz_1;\vz), \vz_2-\vz_1 \rangle \geq 0, \qquad \forall \vz_1,\vz_2 \in \mathbb{R}^{m+n}. 
\end{equation}
To simplify the notations, denote $R(\vz';\vz)=(2D_{\Phi}(\vz',\vz))^{\frac{p-1}{2}}(\nabla \Phi(\vz')-\nabla \Phi(\vz))$. We first prove that 
\begin{equation}\label{eq:regularization_monotone}
    \langle R(\vz_2;\vz)-R(\vz_1;\vz), \vz_2-\vz_1 \rangle \geq \frac{1}{2}\left(\|\vz_1-\vz\|^{p-1}+\|\vz_2-\vz\|^{p-1}\right)\|\vz_2-\vz_1\|^2.
\end{equation}
Note that we have \blue{$
    D_{\Phi}(\vz_2,\vz) = D_{\Phi}(\vz_1,\vz)+\langle \nabla \Phi(\vz_1)- \nabla \Phi(\vz), \vz_2-\vz_1 \rangle + D_{\Phi}(\vz_2,\vz_1)$ by the three-point identity \blue{of Bregman distance~\cite{Chen1993}}. }
Further, by Bernoulli's inequality that $(1+x)^n \geq 1+nx$ for all $x\geq -1$ and $n\geq 1$, we have 
\begin{align}
    (2D_{\Phi}(\vz_2,\vz))^{\frac{p+1}{2}}  &= (2D_{\Phi}(\vz_1,\vz))^{\frac{p+1}{2}}\Bigl(1+\frac{\langle \nabla \Phi(\vz_1)-\nabla \Phi(\vz), \vz_2-\vz_1 \rangle + D_{\Phi}(\vz_2,\vz_1)}{D_{\Phi}(\vz_1,\vz)}\Bigr)^{\frac{p+1}{2}} \nonumber\\
    & \geq (2D_{\Phi}(\vz_1,\vz))^{\frac{p+1}{2}}\Bigl(1+\frac{p+1}{2}\frac{\langle \nabla \Phi(\vz_1)-\nabla \Phi(\vz), \vz_2-\vz_1 \rangle + D_{\Phi}(\vz_2,\vz_1)}{D_{\Phi}(\vz_1,\vz)}\Bigr) \nonumber\\
    & = (2D_{\Phi}(\vz_1,\vz))^{\frac{p+1}{2}}+(p+1)\langle R(\vz_1;\vz), \vz_2-\vz_1 \rangle + (p+1)(2D_{\Phi}(\vz_1,\vz))^{\frac{p-1}{2}}D_{\Phi}(\vz_2,\vz_1) \nonumber\\
    & \geq  (2D_{\Phi}(\vz_1,\vz))^{\frac{p+1}{2}}+(p+1)\langle R(\vz_1;\vz), \vz_2-\vz_1 \rangle + \frac{p+1}{2}\|\vz_1-\vz\|^{{p-1}}\|\vz_2-\vz_1\|^2. \label{eq:Bernoulli_1}
\end{align}
Similarly, we have 
\begin{equation}\label{eq:Bernoulli_2}
    (2D_{\Phi}(\vz_1,\vz))^{\frac{p+1}{2}} \geq (2D_{\Phi}(\vz_2,\vz))^{\frac{p+1}{2}}+(p+1)\langle R(\vz_2;\vz), \vz_1-\vz_2 \rangle + \frac{p+1}{2}\|\vz_2-\vz\|^{{p-1}}\|\vz_1-\vz_2\|^2.
\end{equation}
Adding \eqref{eq:Bernoulli_1} and \eqref{eq:Bernoulli_2} together yields \eqref{eq:regularization_monotone}. 

Next, we will prove that 
\begin{equation}\label{eq:approx_monotone}
    \langle T^{(p-1)}(\vz_2;\vz)-T^{(p-1)}(\vz_1;\vz), \vz_2-\vz_1\rangle \geq -\frac{L_p}{2(p-1)!}\left(\|\vz_2-\vz\|^{p-1} + \|\vz_1-\vz\|^{p-1}\right)\|\vz_2-\vz_1\|^2. 
\end{equation}
Indeed, by the fundamental theorem of calculus, we can write 
\begin{equation}\label{eq:integral_form}
    \langle T^{(p-1)}(\vz_2;\vz)-T^{(p-1)}(\vz_1;\vz), \vz_2-\vz_1\rangle = 
    \int_{0}^1 \langle DT^{(p-1)}(\vz_1+t(\vz_2-\vz_1);\vz) (\vz_2-\vz_1), \vz_2-\vz_1\rangle dt.
\end{equation}
Furthermore, for any $\vw\in \mathbb{R}^d$ we have 
\begin{align}
    \langle DT^{(p-1)}(\vz';\vz)\vw, \vw \rangle &= 
    \langle DF(\vz')\vw, \vw \rangle + \langle (DT^{(p-1)}(\vz';\vz)-DF(\vz'))\vw, \vw \rangle \nonumber\\
    & \geq -\|DT^{(p-1)}(\vz';\vz)-DF(\vz')\|_{\mathrm{op}}\|\vw\|^2 \label{eq:apply_monotone}\\
    & \geq -\frac{L_p}{(p-1)!}\|\vz'-\vz\|^{p-1}\|\vw\|^2, \label{eq:apply_residual}
\end{align}
where \eqref{eq:apply_monotone} holds as $F$ is monotone (cf. Proposition~\ref{prop:facts_on_monotone}(a)) and in \eqref{eq:apply_residual} we used \eqref{eq:residual_DF}. Hence, \eqref{eq:integral_form} implies that 
\begin{align}
    \langle T^{(p-1)}(\vz_2;\vz)-T^{(p-1)}(\vz_1;\vz), \vz_2-\vz_1\rangle &\geq 
    -\frac{L_p\|\vz_2-\vz_1\|^2}{(p-1)!} \int_{0}^1 \|\vz_1+t(\vz_2-\vz_1)-\vz\|^{p-1} dt \nonumber\\
    &\geq -\frac{L_p\|\vz_2-\vz_1\|^2 }{(p-1)!}\int_{0}^1 \left(t\|\vz_2-\vz\|^{p-1} + (1-t)\|\vz_1-\vz\|^{p-1}\right) dt \label{eq:convexity_of_norm}\\
    & = -\frac{L_p\|\vz_2-\vz_1\|^2}{2(p-1)!}\left(\|\vz_2-\vz\|^{p-1} + \|\vz_1-\vz\|^{p-1}\right),\nonumber
\end{align}
where we used the convexity of the function $\|\cdot\|^{p-1}$ in \eqref{eq:convexity_of_norm}. Since $T^{(p-1)}_{\lambda}(\vz';\vz)=T^{(p-1)}(\vz';\vz)+\frac{\lambda}{(p-1)!}R(\vz';\vz)$, combining \eqref{eq:regularization_monotone} and \eqref{eq:approx_monotone} we conclude that 
\eqref{eq:regularized_monotone} is satisfied when $\lambda\geq L_p$. This proves Part~\ref{item:maximal_monotone_high}.

\subsection{Intermediate results}
\label{appen:beta_optimal_high}
To establish the convergence properties of the proposed $p$-th-order optimistic method, we first state a few intermediate results that will be used in the following sections.

At the $k$-th iteration of our $p$-th-order optimistic method, our goal is to find a pair $(\eta,\vz)$ such that 
\begin{equation}\label{eq:higher_stepsize_condition}
    \eta \|F(\vz)-T^{(p-1)}_{\lambda}(\vz;\vz_{k}))\|_{*}\leq \frac{\alpha}{2}\|\vz-\vz_k\|,
\end{equation}
where $\vz$ is computed via \eqref{eq:high_order_subproblem}. Recall that $ \mathcal{B} =\{k: \eta_k < \sigma_k\}$, i.e., the set of iteration indices where the line search scheme backtracks the step size.
The following lemma is the higher-order counterpart of Lemma~\ref{lem:beta_optimal_point_2nd} and can be proved in a similar way. In particular, we will reuse Lemma~\ref{lem:eta_eta_prime} from Appendix~\ref{appen:beta_optimal_2nd}. 
\begin{lemma}\label{lem:beta_optimal_point_high}
    Let $\{\eta_k\}_{k \geq 0}$ be the step sizes in \eqref{eq:OMD_high_ls} generated by Algorithm~\ref{alg:2nd_optimistic}. Suppose that $F$ is $p$-th-order $L_p$-smooth and $\Phi$ is $L_{\Phi}$-smooth. Further define $\vv_k=\hat{\eta}_k(F(\vz_k)-T^{(p-1)}_\lambda(\vz_k;\vz_{k-1}))$. If $k\in \mathcal{B}$, then we have  
    \begin{equation*}
        \eta_k \geq \frac{p!\alpha\beta^p}{2L_p(\Phi,\lambda)\left(L_{\Phi}^{3/2}\|\vz_{k+1}-\vz_k\|+(\beta+L_{\Phi}^{1/2})\|\vv_k\|_*\right)^{p-1}}.
    \end{equation*}  
\end{lemma}
\begin{proof}
To simplify the notation, we drop the subscript $k$ and denote $\vz_{k}$ by $\vz^{-}$.
Since $k\in \mathcal{B}$, by Algorithm~\ref{alg:ls}, the condition in \eqref{eq:simplified_step_bound} is violated for $\eta' = \eta/\beta$ and $\vz' = \vz(\eta'; \vz^-)$ (Otherwise, the line search scheme would choose $\eta'$ instead of $\eta$).
Hence, we have  
\begin{equation*}
    \frac{1}{2}\alpha\|\vz'-\vz^{-}\| 
    <\eta' \|F(\vz')-T^{(p-1)}_{\lambda}(\vz';\vz_{k}))\|_*
    \leq \frac{\eta' L_p(\Phi,\lambda)}{p!} \|\vz'-\vz^{-}\|^p, 
\end{equation*}
where we used Lemma~\ref{lem:property_high}(a) in the last inequality. This gives us $    \eta \geq \beta \eta'>\frac{p!\alpha\beta}{2L_p(\Phi,\lambda)\|\vz(\eta';\vz^{-})-\vz^{-}\|^{p-1}}$. 
Together with Lemma~\ref{lem:eta_eta_prime}, this leads to Lemma~\ref{lem:beta_optimal_point_high}. 
\end{proof}

By combining Lemma~\ref{lem:beta_optimal_point_2nd} and Lemma~\ref{lem:sum_of_square}, we further present the following result.
\begin{lemma}\label{lem:bound_on_sum_square_inverse_high}
    We have 
    $\sum_{k\in \mathcal{B}} \eta_k^{-\frac{2}{p-1}}{\prod_{l=0}^{k-1} (1+\eta_l\mu)}  \leq \gamma_p^2 (L_p(\Phi,\lambda))^\frac{2}{p-1}D_{\Phi}(\vz^*,\vz_0)$, 
    where $\gamma_p$ is an absolute constant defined in \eqref{eq:def_gamma_p}. 
\end{lemma}
\begin{proof}
    Similar to the proof of Lemma~\ref{lem:sum_of_square}, we will prove both results in a united way by regarding the convex-concave setting as 
    a special case of the strongly-convex-strongly-concave setting with $\mu=0$.
    We first apply Lemma~\ref{lem:beta_optimal_point_high} for $k\in \mathcal{B}$. Note that by the choice of $\hat{\eta}_k$ and \eqref{eq:stepsize_higher}, we can upper bound 
    $\|\vv_k\|_* =\frac{\eta_{k-1}}{1+\eta_{k-1}\mu}\|F(\vz_k)-T^{(p)}_{\lambda}(\vz_k;\vz_{k-1})\|_* \leq \frac{\alpha}{2(1+\eta_{k-1}\mu)}\|\vz_k-\vz_{k-1}\|$.
    Together with Lemma~\ref{lem:beta_optimal_point_high}, this leads to 
    \begin{align*}
        \eta_k^{-\frac{1}{p-1}} & \leq \left(\frac{2L_p(\Phi,\lambda)}{p!\alpha\beta^p}\right)^{\frac{1}{p-1}}\left(L_{\Phi}^{3/2}\|\vz_{k+1}-\vz_k\|+(\beta+L_{\Phi}^{1/2})\|\vv_k\|_*\right) \\
        & \leq c_1(L_p(\Phi,\lambda))^{\frac{1}{p-1}}\|\vz_{k+1}-\vz_k\| + c_2(L_p(\Phi,\lambda))^{\frac{1}{p-1}}\frac{\|\vz_{k}-\vz_{k-1}\|}{1+\eta_{k-1}\mu},
    \end{align*}
    where we let $c_1:=L_{\Phi}^{3/2}(2/(p!\alpha\beta^p))^{\frac{1}{p-1}}$ and $c_2:=L_{\Phi}^{-3/2}(\beta+L_{\Phi}^{1/2})\alpha c_1/2$ to simplify the notation.
    Furthermore, we square both sides of the above inequality and use Young's inequality to get  
    \begin{equation}\label{eq:after_Young_high}
        \eta_k^{-\frac{2}{p-1}} \leq (c_1^2+c_1c_2) (L_p(\Phi,\lambda))^{\frac{2}{p-1}}\|\vz_{k+1}-\vz_k\|^2+\frac{(c_2^2+c_1c_2)(L_p(\Phi,\lambda))^{\frac{2}{p-1}}}{(1+\eta_{k-1}\mu)^2}\|\vz_{k}-\vz_{k-1}\|^2. 
    \end{equation}
    Multiplying both sides of \eqref{eq:after_Young_high} by $\prod_{l=0}^{k-1}(1+\eta_l\mu)$, we further have 
    \begin{equation}\label{eq:bound_Gamma_ls_high}
        \begin{aligned}
            \eta_k^{-\frac{2}{p-1}} \prod_{l=0}^{k-1}(1+\eta_l\mu) 
            &\leq (c_1^2+c_1c_2)(L_p(\Phi,\lambda))^{\frac{2}{p-1}}\|\vz_{k+1}-\vz_k\|^2\prod_{l=0}^{k-1}(1+\eta_l\mu)\\
            &\phantom{{}\leq{}}+(c_2^2+c_1c_2)(L_p(\Phi,\lambda))^{\frac{2}{p-1}}\|\vz_{k}-\vz_{k-1}\|^2\prod_{l=0}^{k-2}(1+\eta_l\mu). 
        \end{aligned}
    \end{equation}
    By summing both sides of \eqref{eq:bound_Gamma_ls_high} over $k\in \mathcal{B}$ and applying Lemma~\ref{lem:sum_of_square}, we get the desired result in Lemma~\ref{lem:bound_on_sum_square_inverse_high}. 
\end{proof}

Finally, we build on Lemma~\ref{lem:bound_on_sum_square_inverse} to control step sizes $\eta_k$ for all $k \geq 0$. 

\begin{lemma}\label{lem:induction_bound_high}
    For any $k \geq 0$, we have 
    \begin{equation}\label{eq:induction_bound_high}
        \eta_k^{-\frac{2}{p-1}}  \prod_{l=0}^{k-1} (1+\eta_l\mu) \leq \max\left\{\gamma_p^2 (L_p(\Phi,\lambda))^\frac{2}{p-1}D_{\Phi}(\vz^*,\vz_0), {\beta^{\frac{2k}{p-1}}} \sigma^{-\frac{2}{p-1}}_0  \right\}. 
    \end{equation}
\end{lemma}
\begin{proof}
    We shall prove Lemma~\ref{lem:induction_bound} by induction. For the base case where $k=0$, we consider two subcases. If $0 \notin \mathcal{B}$, then we have $\eta_0 = \sigma_0$, which directly implies \eqref{eq:induction_bound_high}. Otherwise, since $0\in \mathcal{B}$, we can use Lemma~\ref{lem:bound_on_sum_square_inverse_high} to obtain that ${\eta_0^{-\frac{2}{p-1}}} \leq \sum_{k\in \mathcal{B}} \left(\eta_k^{-\frac{2}{p-1}} \prod_{l=0}^{k-1} (1+\eta_l\mu) \right) \leq \gamma_p^2 (L_p(\Phi,\lambda))^\frac{2}{p-1}D_{\Phi}(\vz^*,\vz_0)$. This proves \eqref{eq:induction_bound_high} for $k=0$. 

    Now assume that \eqref{eq:induction_bound_high} holds for $k = s$ where $s\geq 0$. If $s+1\in \mathcal{B}$,  then similarly we have ${\eta_{s+1}^{-\frac{2}{p-1}}} \prod_{l=0}^{s} (1+\eta_l\mu) \leq \sum_{k\in \mathcal{B}} \left(\eta_{k}^{-\frac{2}{p-1}} \prod_{l=0}^{k-1} (1+\eta_l\mu) \right) \leq \gamma_p^2 (L_p(\Phi,\lambda))^\frac{2}{p-1}D_{\Phi}(\vz^*,\vz_0)$ by Lemma~\ref{lem:bound_on_sum_square_inverse_high}, which implies~\eqref{eq:induction_bound_high}. Otherwise, if $s+1\notin \mathcal{B}$, then by definition we have $\eta_{s+1} = \sigma_{s+1} = \eta_{s} (1+\eta_{s}\mu)^{\frac{p-1}{2}}/\beta$, and furthermore
\begin{equation*}
    \eta_{s+1}^{-\frac{2}{p-1}} \prod_{l=0}^{s} (1+\eta_l\mu) = \beta^{\frac{2}{p-1}}\eta_{s}^{-\frac{2}{p-1}}\prod_{l=0}^{s-1} (1+\eta_l\mu) \leq \beta^{\frac{2}{p-1}}\max\left\{\gamma_p^2 (L_p(\Phi,\lambda))^\frac{2}{p-1}D_{\Phi}(\vz^*,\vz_0), {\beta^{\frac{2s}{p-1}}} \sigma^{-\frac{2}{p-1}}_0 \right\}, 
\end{equation*}
where we used the induction hypothesis in the last inequality. 
In both cases, we prove that \eqref{eq:induction_bound} is satisfied for $k=s+1$, and hence the claim follows by induction.
\end{proof}

\begin{lemma}\label{lem:bound_on_sum_square_inverse_full_high}
    Let $\{\eta_k\}_{k \geq 0}$ be the step sizes in \eqref{eq:OMD_high_ls} generated by Algorithm~\ref{alg:higher_optimistic}. Recall the definition of $\zeta_k$ in \eqref{eq:def_zeta}. Then, 
    we have %
    \(
        \sum_{k=0}^{N-1} \eta_k^{-\frac{2}{p-1}} \zeta_k^{-1}\leq
        \left(1-\beta^{\frac{2}{p-1}}\right)^{-1} \left(\sigma_0^{-\frac{2}{p-1}} +  \gamma_p^2 (L_p(\Phi,\lambda))^\frac{2}{p-1}D_{\Phi}(\vz^*,\vz_0) \right)
    \), where  $\gamma_p$  is defined in \eqref{eq:def_gamma_p}. 
\end{lemma}
\begin{proof}
Recall that $ \mathcal{B} =\{k: \eta_k < \sigma_k\}$, i.e., the set of iteration indices where the line search backtracks.  %
Moreover, in Algorithm~\ref{alg:higher_optimistic}, we have  $\eta_k=\sigma_k$ when $k \notin \mathcal{B}$ and $\sigma_k = \eta_{k-1}(1+\eta_{k-1}\mu)^{\frac{p-1}{2}}/\beta$ for $k \geq 1$. 
Hence, following similar arguments as in the proof of Lemma~\ref{lem:bound_on_sum_square_inverse_strongly_convex}, we first bound 
\begin{align*}
    \sum_{k\notin \mathcal{B}} \eta_k^{-\frac{2}{p-1}} \zeta_k^{-1} =\sum_{k\notin \mathcal{B}} \sigma_k^{-\frac{2}{p-1}} \zeta_k^{-1} &\leq \sigma_0^{-\frac{2}{p-1}} + \beta^{\frac{2}{p-1}} \sum_{k \geq 1, k\notin \mathcal{B}} \eta^{-\frac{2}{p-1}}_{k-1} (1+\eta_{k-1}\mu)^{-1} \zeta_k^{-1} \\
    &=  \sigma_0^{-\frac{2}{p-1}}  + \beta^{\frac{2}{p-1}} \sum_{k \geq 1, k\notin \mathcal{B}} \eta^{-\frac{2}{p-1}}_{k-1} \zeta_{k-1}^{-1} \leq \sigma_0^{-\frac{2}{p-1}}  + \beta^{\frac{2}{p-1}} \sum_{k=0}^{N-1} \eta^{-\frac{2}{p-1}}_{k} \zeta_{k}^{-1},
\end{align*}
where we used $\zeta_{k-1} = \zeta_k(1+\eta_{k-1}\mu)$ in the last equality. This further leads to 
\begin{equation*}
    \sum_{k=0}^{N-1} \eta_k^{-\frac{2}{p-1}} \zeta_k^{-1}  \leq \sum_{k\in \mathcal{B}} \eta_k^{-\frac{2}{p-1}} \zeta_k^{-1} + \sum_{k\notin \mathcal{B}} \eta_k^{-\frac{2}{p-1}} \zeta_k^{-1} \leq \sigma_0^{-\frac{2}{p-1}} + \sum_{k \in \mathcal{B}} \eta_k^{-\frac{2}{p-1}} \zeta_k^{-1} +  \beta^{\frac{2}{p-1}}\sum_{k=0}^{N-1} \eta_{k}^{-\frac{2}{p-1}} \zeta_k^{-1}. 
\end{equation*}
Combining the left-hand side with the last term in the right-hand side, we obtain $\sum_{k=0}^{N-1} \eta_k^{-\frac{2}{p-1}} \zeta_k^{-1} \leq \left(1-\beta^{\frac{2}{p-1}}\right)^{-1} \left(\sigma_0^{-\frac{2}{p-1}} +  \sum_{k \in \mathcal{B}} \eta_k^{-\frac{2}{p-1}} \zeta_k^{-1} \right)$. By applying Lemma~\ref{lem:bound_on_sum_square_inverse_high}, we obtain the desired result in Lemma~\ref{lem:bound_on_sum_square_inverse_full_high}.
\end{proof}

\subsection{Convex-concave case}
\label{appen:OMD_high_ls_convex}

\begin{proof}[Proof of Theorem~\ref{thm:OMD_high_ls_convex}]

By H{\"o}lder's inequality, we have $\left(\sum_{k=0}^{N-1} \eta_k\right)^{\frac{2}{p+1}}
\Bigl(\sum_{k=0}^{N-1} \eta_k^{-\frac{2}{p-1}} \Bigr)^{\frac{p-1}{p+1}} 
\geq N$. 
Together with Lemma~\ref{lem:bound_on_sum_square_inverse_full_high}, this further implies 
\begin{equation*}
    \sum_{k=0}^{N-1} \eta_k \geq \left(\sum_{k=0}^{N-1} \eta_k^{-\frac{2}{p-1}}\right)^{-\frac{p-1}{2}}N^{\frac{p+1}{2}} \geq \frac{\left(1-\beta^{\frac{2}{p-1}}\right)^{\frac{p-1}{2}}}{\left(\sigma_0^{-\frac{2}{p-1}} +  \gamma_p^2 (L_p(\Phi,\lambda))^\frac{2}{p-1}D_{\Phi}(\vz^*,\vz_0) \right)^{\frac{p-1}{2}}}  N^{\frac{p+1}{2}}. 
\end{equation*}
The rest follows from Proposition~\ref{prop:GOMD}.
\end{proof}
\subsection{Strongly-convex-strongly-concave case}
\label{appen:high-order-sc-sc}
Similar to Lemma~\ref{lem:multiplying_factor}, the following lemma is the key to our convergence results.  
\begin{lemma}\label{lem:multiplying_factor_high}
    Let $\{\eta_k\}_{k \geq 0}$ be the step sizes in \eqref{eq:OMD_high_ls} generated by Algorithm~\ref{alg:higher_optimistic} and $\{\zeta_k\}_{k\geq 0}$ be defined in \eqref{eq:def_zeta}. 
    Further, recall the definition of $\gamma_p$ in \eqref{eq:def_gamma_p}. Then for any $0 \leq k_1<k_2\leq N$, we have {$\frac{1}{\zeta_{k_2}} \geq \frac{1}{\zeta_{k_1}}+C_p^{-\frac{p-1}{2}}\left(\sum_{k=k_1}^{k_2-1} \frac{1}{\zeta_k}\right)^{\frac{p+1}{2}}$, where $C_p$ is defined in \eqref{eq:def_kappa_p_C_p}. 
    }
\end{lemma}
\begin{proof}

By the definition of $\zeta_k$ in \eqref{eq:def_zeta}, we have $\eta_k=(\zeta_{k}/\zeta_{k+1}-1)/\mu$.  
We apply Lemma~\ref{lem:bound_on_sum_square_inverse_full_high}
    and rewrite %
    the inequality
    in terms of $\{\zeta_k\}_{k=0}^{N-1}$ as
\begin{align}
    & \sum_{k=0}^{N-1}  \left(\frac{\mu\zeta_{k+1}}{\zeta_k-\zeta_{k+1}}\right)^{\frac{2}{p-1}}\frac{1}{\zeta_k} \leq \left(1-\beta^{\frac{2}{p-1}}\right)^{-1} \left(\sigma_0^{-\frac{2}{p-1}} +  \gamma_p^2 (L_p(\Phi,\lambda))^\frac{2}{p-1}D_{\Phi}(\vz^*,\vz_0) \right) \nonumber\\
    \Leftrightarrow \quad    
    & \sum_{k=0}^{N-1}  \left(\frac{1}{1/\zeta_{k+1}-1/\zeta_{k}}\right)^{\frac{2}{p-1}} \zeta_k^{-\frac{p+1}{p-1}} \leq C_p, \label{eq:bound_sum_square_zeta_high}
\end{align}
where in \eqref{eq:bound_sum_square_zeta_high} we used the definition of $C_p$. 
Since each summand in \eqref{eq:bound_sum_square_zeta_high} is nonnegative, it follows that $    \sum_{k=k_1}^{k_2-1}  \left(\frac{1}{1/\zeta_{k+1}-1/\zeta_{k}}\right)^{\frac{2}{p-1}} \zeta_k^{-\frac{p+1}{p-1}} \leq C_p $ for any $k_2>k_1\geq 0$. 
Furthermore, by applying H{\"o}lder's inequality we get
\begin{gather*}
    \left[\sum_{k=k_1}^{k_2-1} (\frac{1}{\zeta_{k+1}}-\frac{1}{\zeta_k})\right]^{\frac{2}{p+1}}
    \left[\sum_{k=k_1}^{k_2-1}  \left(\frac{1}{1/\zeta_{k+1}-1/\zeta_{k}}\right)^{\frac{2}{p-1}} \zeta_k^{-\frac{p+1}{p-1}} \right]^{\frac{p-1}{p+1}}
    \geq \sum_{k=k_1}^{k_2-1} \frac{1}{\zeta_k}, %
\end{gather*}
and Lemma~\ref{lem:multiplying_factor_high} follows from the above two inequalities.  
\end{proof}

\begin{proof}[Proof of Theorem~\ref{thm:high_strongly_convex_global}]\label{appen:high_strongly_convex_global}

Without loss of generality, we can assume that $ \frac{(2-\alpha)\epsilon}{2{D_{\Phi}(\vz^*,\vz_0)}}\leq 1$; otherwise the result becomes trivial as $\zeta_0 = 1$. Based on \eqref{eq:def_zeta}, $\{\zeta_k\}_{k \geq 0}$ is non-increasing in $k$. Hence, from Lemma~\ref{lem:multiplying_factor_high} we get
\begin{equation*}
    \frac{1}{\zeta_{k_2}} 
    \geq \frac{1}{\zeta_{k_1}}+C_p^{-\frac{p-1}{2}}\left(\sum_{k=k_1}^{k_2-1} \frac{1}{\zeta_k}\right)^{\frac{p+1}{2}} 
    \geq \frac{1}{\zeta_{k_1}}+C_p^{-\frac{p-1}{2}}\left(\frac{k_2-k_1}{\zeta_{k_1}}\right)^{\frac{p+1}{2}}. 
\end{equation*}
In particular, this implies that
$\zeta_{k_2}\leq \frac{1}{2}\zeta_{k_1}$ when we have $k_2-k_1\geq C_p^{\frac{p-1}{p+1}} \zeta_{k_1}^{\frac{p-1}{p+1}}$.  Hence, for any integer $l\geq 0$, we can prove by induction that the number of iterations $N$ to achieve 
$\zeta_N \geq 2^{-l}$ does not exceed
\begin{equation*}
    \sum_{k=0}^{l-1} \left \lceil C_p^{\frac{p-1}{p+1}} 2^{-\frac{p-1}{p+1}k}\right\rceil 
    \leq \sum_{k=0}^{l-1} \left(C_p^{\frac{p-1}{p+1}} 2^{-\frac{p-1}{p+1}k}+1\right) 
    \leq \frac{1}{1-2^{-\frac{p-1}{p+1}}}C_p^{\frac{p-1}{p+1}}+l.
 \end{equation*}
 The bound in \eqref{eq:zeta_high_global_bound} immediately follows by setting $l=\lceil \log_2(\frac{2{D_{\Phi}(\vz^*,\vz_0)}}{(2-\alpha)\epsilon})\rceil \leq \log_2\left(\frac{2{D_{\Phi}(\vz^*,\vz_0)}}{(2-\alpha)\epsilon}\right)+1$. 
\end{proof}
\begin{proof}[Proof of Theorem~\ref{thm:high_strongly_convex_local}]
\label{appen:high_strongly_convex_local} 
By the definition of $\zeta_k$ in \eqref{eq:def_zeta}, we have 
    $\zeta_{k+1} = \frac{\zeta_k}{1+\eta_{k}\mu} \leq \frac{\zeta_k}{\eta_{k}\mu}$. Moreover, Lemma~\ref{lem:induction_bound_high} implies that $\frac{1}{\eta_k} \leq \max\left\{\gamma_p^{p-1} L_p(\Phi,\lambda) (D_{\Phi}(\vz^*,\vz_0))^{\frac{p-1}{2}}, \frac{\beta^{k}}{\sigma_0} \right\} \prod_{l=0}^{k-1} (1+\eta_l\mu)^{-\frac{p-1}{2}} = \max\left\{\gamma_p^{p-1} L_p(\Phi,\lambda) (D_{\Phi}(\vz^*,\vz_0))^{\frac{p-1}{2}}, \frac{\beta^{k}}{\sigma_0} \right\} \zeta_k^{-\frac{p-1}{2}}$.
    Theorem~\ref{thm:high_strongly_convex_local} follows from the two inequalities and the definition of $\tilde{\kappa}_p(\vz_0)$ in \eqref{eq:def_kappa_p_C_p}. 
\end{proof}

\begin{proof}[Proof of Corollary~\ref{coro:high_complexity_bound}]\label{appen:high_complexity_bound}

    Recall that $D_{\Phi}(\vz^*,\vz)\leq \epsilon$ if $\zeta_k \leq \frac{(2-\alpha)\epsilon}{2{D_{\Phi}(\vz^*,\vz_0)}}$ by \eqref{eq:bound_with_zeta}. Hence, it suffices to upper bound the number of iterations required such that the latter condition holds. Also, we only need to prove the second case where $\epsilon < \frac{1}{(2-\alpha)\gamma^2_p}(\frac{\mu}{L_p(\Phi,\lambda)})^{\frac{2}{p-1}}$, as the first case directly follows from Theorem~\ref{thm:high_strongly_convex_global}. 
    Let $N_1$ be the smallest integer such that $\zeta_{N_1} \leq \frac{1}{2}(\tilde{\kappa}_p(\vz_0))^{-\frac{2}{p-1}}$ and $\frac{\beta^{N_1}}{\sigma_0 \mu} \leq \tilde{\kappa}_p(\vz_0)$. By setting $\epsilon = %
    \frac{1}{(2-\alpha)\gamma^2_p}(\frac{\mu}{L_p(\Phi,\lambda)})^{\frac{2}{p-1}}$ in Theorem~\ref{thm:high_strongly_convex_global}, we obtain that the first condiiton in satisfied when 
    \begin{align}
        N_1 &\geq \max\Biggl\{\frac{1}{1-2^{-\frac{p-1}{p+1}}}C_p^{\frac{p-1}{p+1}}+ \log_2\Biggl(\frac{2\gamma_p^2L_p(\Phi,\lambda)^{\frac{2}{p-1}}{D_{\Phi}(\vz^*,\vz_0)} }{\mu^{\frac{2}{p-1}}}\Biggr)+1, 1\Biggr\} \nonumber \\
        & = \max\left\{\frac{1}{1-2^{-\frac{p-1}{p+1}}}C_p^{\frac{p-1}{p+1}}+ \frac{2}{p-1}\log_2 \tilde{\kappa}_p(\vz_0)+2, 1\right\}. \label{eq:bound_on_N1_high}
    \end{align}
    Moreover, the second condition is satisfied when $N_1 \geq -\log_{1/\beta}(\sigma_0 \mu \tilde{\kappa}_p(\vz_0))$. Since $N_1$ is the smallest integer to satisfy both conditions, we have 
    \begin{equation*}
        N_1 \leq \max\left\{\frac{1}{1-2^{-\frac{p-1}{p+1}}}C_p^{\frac{p-1}{p+1}}+ \frac{2}{p-1}\log_2 \tilde{\kappa}_p(\vz_0)+2, -\log_{\frac{1}{\beta}}(\sigma_0 \mu \tilde{\kappa}_p(\vz_0))+1,1\right\}.
    \end{equation*}
    Furthermore, since $\frac{\beta^{N_1}}{\sigma_0 \mu} \leq \tilde{\kappa}_p(\vz_0)$, we have ${\zeta_{k+1}}   \leq \tilde{\kappa}_p(\vz_0) {\zeta_k}^\frac{p+1}{2}$ by Theorem~\ref{thm:high_strongly_convex_local} for all $k \geq N_1$, which implies that 
    $(\tilde{\kappa}_p(\vz_0))^{\frac{2}{p-1}} \zeta_{k+1} \leq \left((\tilde{\kappa}_p(\vz_0))^{\frac{2}{p-1}} \zeta_k\right)^{\frac{p+1}{2}}$. By induction, we can prove that 
    $(\tilde{\kappa}_p(\vz_0))^{\frac{2}{p-1}} \zeta_{k}\leq 2^{-(\frac{p+1}{2})^{k-N_1}}$ for all $k\geq N_1$. Hence, the additional number of iterations required to achieve $\zeta_k \leq \frac{(2-\alpha)\epsilon}{2{D_{\Phi}(\vz^*,\vz_0)}}$ does not exceed
    \begin{equation}\label{eq:iterations_after_local_high}
        \Biggl\lceil \log_{\frac{p+1}{2}}\log_2 \Biggl(\frac{2D_{\Phi}(\vz^*,\vz_0)}{(2-\alpha) (\tilde{\kappa}_p(\vz_0))^{\frac{2}{p-1}}\epsilon}\Biggr)\Biggr\rceil \leq \log_{\frac{p+1}{2}}\log_2 \Biggl(\frac{2\mu^{\frac{2}{p-1}}}{(2-\alpha)\gamma^2_p (L_p(\Phi,\lambda))^{\frac{2}{p-1}}\epsilon}\Biggr)+1,
    \end{equation}
    where we used the definition of $\kappa_p$ in \eqref{eq:def_kappa_p_C_p}.  
    The result in \eqref{eq:high_complexity_after_local} now follows from \eqref{eq:bound_on_N1_high} and \eqref{eq:iterations_after_local_high}.
\end{proof}

\subsection{Proof of \texorpdfstring{Theorem~\ref{thm:steps_upper_high}}{Theorem 7.6}}\label{appen:steps_upper_high}

By Proposition~\ref{prop:ls_total}, the number of calls to the optimistic subsolver at the $k$-th iteration is given by $\log_{\frac{1}{\beta}}\frac{\sigma_k}{\eta_k}+1$. Since $\sigma_k=\eta_{k-1}(1+\eta_{k-1}\mu)^{\frac{p-1}{2}}/\beta$ for $k\geq 1$,  the total number of calls after $N$ iterations is
\begin{align}
    N_{\mathrm{total}} = \sum_{k=0}^{N-1} \left(\log_{\frac{1}{\beta}}\frac{\sigma_k}{\eta_k}+1\right) &= N + \log_{\frac{1}{\beta}}\frac{\sigma_0}{\eta_0} + \sum_{k=1}^{N-1} \log_{\frac{1}{\beta}}\frac{\eta_{k-1}(1+\eta_{k-1}\mu)^{\frac{p-1}{2}}}{\beta\eta_k} \\
    &= 2N-1 + \log_{\frac{1}{\beta}}\frac{\sigma_0}{\eta_0} + \sum_{k=1}^{N-1} \log_{\frac{1}{\beta}}\frac{\eta_{k-1}(1+\eta_{k-1}\mu)^{\frac{p-1}{2}}}{\eta_k} \\
    &= 2N -1 + \log_{\frac{1}{\beta}} \left(\frac{\sigma_0}{\eta_{N-1}}\prod_{k=0}^{N-2} (1+\eta_{k}\mu)^{\frac{p-1}{2}} \right). \label{eq:line_search_bound_high} %
\end{align}
In addition, we have ${\eta_{N-1}^{-\frac{2}{p-1}}} \prod_{k=0}^{N-2} (1+\eta_k\mu) \leq \max\left\{\gamma_p^2 (L_p(\Phi,\lambda))^\frac{2}{p-1}D_{\Phi}(\vz^*,\vz_0), {\beta^{\frac{2(N-1)}{p-1}}} \sigma^{-\frac{2}{p-1}}_0\right\}$ by Lemma~\ref{lem:induction_bound_high}. Together with \eqref{eq:line_search_bound_high}, this leads to 
\begin{align*}
N_{\mathrm{total}} &\leq 2N-1 + \max\left\{ \log_{\frac{1}{\beta}}(\sigma_0\gamma_p^{p-1} L_p(\Phi,\lambda) (D_{\Phi}(\vz^*,\vz_0))^{\frac{p-1}{2}}), -(N-1)\right\} \\
&=\max\left\{2N-1 + \log_{\frac{1}{\beta}}(\sigma_0\gamma_p^{p-1} L_p(\Phi,\lambda) (D_{\Phi}(\vz^*,\vz_0))^{\frac{p-1}{2}}), N \right\}
\end{align*} 
This completes the proof.